\documentclass[12pt,twoside,english,a4paper,numbers=endperiod,reqno]{amsart}

\usepackage{microtype}
\usepackage{siunitx}
\usepackage{xspace}
\usepackage{amsmath,amsfonts,amsthm,amssymb}
\usepackage{tikz}
\usepackage{tikz-cd}
\usepackage{enumitem}
\usepackage{subcaption}
\usepackage{dsfont}
\usepackage{ytableau}
\usepackage{biblatex}
\usepackage{mathtools}
\usepackage{hyperref}
\usepackage{xurl}
\hypersetup{breaklinks=true,hidelinks}
\usepackage{setspace}
\usepackage{appendix}

\usepackage{geometry}
\usepackage{fancyhdr}

\fancypagestyle{normal}{%
    
	\fancyhf{}%
	\fancyhead[CE]{\nouppercase{\leftmark}}
	\fancyhead[CO]{\nouppercase{Multiprojective Seshadri stratifications and Young-tableaux}}
	\fancyfoot[CE,CO]{\thepage}
}
\fancypagestyle{empty}{%
    
	\fancyhf{}%
	\fancyhead[CE]{}
	\fancyhead[CO]{}
	\fancyfoot[CE,CO]{\thepage}
}
\setlist{itemsep=4pt}

\title{Multiprojective Seshadri stratifications and Young-tableaux}
\author{Henrik Müller}
\address{Department Mathematik/Informatik, Universität zu Köln, Weyertal 86-90, 50931 Cologne, Germany}
\email{\href{henrikmueller.math@gmail.com}{henrikmueller.math@gmail.com}}
\date{}

\theoremstyle{plain}
\newtheorem{theorem}{Theorem}[section]
\numberwithin{theorem}{section}
\newtheorem{lemma}[theorem]{Lemma}
\newtheorem{corollary}[theorem]{Corollary}
\newtheorem{proposition}[theorem]{Proposition}

\theoremstyle{definition}
\newtheorem{definition}[theorem]{Definition}
\newtheorem{example}[theorem]{Example}

\newtheorem{remark}[theorem]{Remark}

\begin{document}

\setstretch{1.07}

\newcommand{\Abk}[1][Abk]{#1.\xspace}
\newcommand{\ie}{\Abk[i.\,e]}
\newcommand{\wwlog}{\Abk[w.\,l.\,o.\,g]}
\newcommand{\wrt}{\Abk[w.\,r.\,t]}
\newcommand{\WWlog}{\Abk[W.\,l.\,o.\,g]}
\newcommand{\loccit}{\textit{loc.\,cit.}}

\newcommand{\B}{\mathbb{B}}
\renewcommand{\F}{\mathbb{F}}
\renewcommand{\K}{\mathbb{K}}
\renewcommand{\A}{\mathbb{A}}
\renewcommand{\N}{\mathbb{N}}
\newcommand{\Z}{\mathbb{Z}}
\newcommand{\PP}{\mathbb{P}}
\newcommand{\Q}{\mathbb{Q}}
\newcommand{\R}{\mathbb{R}}

\newcommand{\set}[1]{\{#1\}}
\newcommand{\Set}[1]{\left\{#1\right\}}

\newcommand{\id}{\mathrm{id}}
\newcommand{\Spec}{\operatorname{Spec}}
\newcommand{\Proj}{\operatorname{Proj}}
\newcommand{\Multiproj}{\operatorname{Multiproj}}
\newcommand{\supp}{\operatorname{supp}}
\newcommand{\longhookrightarrow}{\lhook\joinrel\longrightarrow}
\newcommand{\ulW}{\underline W}

\newlist{abbrv}{itemize}{1}
\setlist[abbrv,1]{label=,labelwidth=1in,align=parleft,leftmargin=!,noitemsep}

\newcommand{\whatisthis}{article}

\numberwithin{equation}{section}


\begin{abstract}
We provide an algebraic-geometrical interpretation of the classical semistandard Young-tableaux via the notion of Seshadri stratifications. The columns appearing in such a tableau correspond to vanishing multiplicities of certain rational functions on Schubert varieties. To build a framework for this correspondence we generalize Seshadri stratifications to multiprojective varieties, which forms the largest part of this \whatisthis{}.
\end{abstract}

\maketitle
\thispagestyle{empty}

\section{Introduction}

Consider the simple algebraic group $G = \mathrm{SL}_n(\K)$ over a field $\K$ of characteristic zero and the Borel subgroup $B \subseteq G$ of all upper triangular matrices. The full flag variety $G/B$ can be embedded into a product of projective spaces via the Plücker embedding:
\begin{align*}
    G/B \mkern2mu \longhookrightarrow \mkern2mu \prod_{i=1}^{n-1} \mkern2mu \PP(\hbox{$\bigwedge\nolimits^{i} \K^n$}).
\end{align*}
The Plücker coordinates generate the multihomogeneous coordinate ring $R$ as a $\K$-algebra and they are indexed by the non-empty, proper subsets of $\set{1, \dots, n}$. It is well known that the monomials corresponding to semistandard Young-tableaux form a basis of $R$ (\cite[Chapter 2]{seshadri2016introduction}). This is one of the earliest examples of a standard monomial theory (SMT). It dates back to the 1940s, where Hodge described a similar basis for Grassmann varieties. However there still exists no clear definition of a SMT. This term rather refers to specific examples, which usually come from the representation theory of semisimple algebraic groups or Lie algebras. Given an algebra generated by a finite set $S$, the set of all monomials in $S$ generate this algebra as a vector space. One tries to extract a basis from this generating set via combinatorial methods. The basis vectors are then called \textit{standard} and every monomial in $S$ not belonging to this basis is called \textit{non-standard}.

This leads us to the following result motivating this \whatisthis{}. We obtain a geometrical interpretation of the above SMT by constructing a quasi-valuation $\mathcal V$ with at most one-dimensional leaves on $R$ and by proving the following statement.

\vskip 10pt
\noindent
\textbf{Theorem} (Corollary~\ref{cor:smt_type_A}). The elements in the image of $\mathcal V$ correspond bijectively to semistandard Young-tableaux with less than $n$ rows and entries in $\set{1, \dots, n}$.
\vskip 10pt

This geometrical interpretation of Young-tableaux is based on the work of Chiriv{\`i}, Fang and Littelmann in~\cite{seshstrat} and \cite{seshstratandschubvar}. They introduced the notion of a \textit{Seshadri strati\-fi\-cation} on an embedded projective variety $X \subseteq \PP(V)$. It consists of a family $(X_p)_{p \in A}$ of closed subvarieties $X_p \subseteq X$ indexed by a finite, graded poset $A$ of length $\dim X$ and a homogeneous function $f_p$, called \textit{extremal function}, in the homogeneous coordinate ring $\K[X]$ of $X$ for each $p \in A$. Of course, this data needs to fulfill certain conditions. For example, every variety $X_p$ has to be irreducible and smooth in codimension one and the grading on $A$ is required to be compatible with the dimensions of the subvarieties, \ie $X_q$ is a divisor in $X_p$, if and only if $q < p$ is a covering relation in $A$.

Seshadri stratifications use a web of subvarieties in contrast to the Newton-Okounkov theoretical approach (\cite{kaveh2012newton}, \cite{lazarsfeld2009convex}), which uses a flag of subvarieties. By taking successive vanishing multiplicities along this web, every Seshadri stratification induces a quasi-valuation $\mathcal V: \K[X] \setminus \set{0} \to \Q^A$, which can be thought of as a filtration of the homogeneous coordinate ring $\K[X]$. In general, the quasi-valuation $\mathcal V$ is not quite canonical, as it depends on the choice of a total order $\geq^t$ linearizing the partial order on $A$. The subquotients (called \textit{leaves}) of the filtration on $\K[X]$ are at most one-dimensional and they are indexed by the image $\Gamma$ of $\mathcal V$, which is a union of finitely generated semigroups $\Gamma_{\mathfrak C}$ over all maximal chains $\mathfrak C$ in the poset $A$. Hence $\Gamma$ is called the \textit{fan of monoids} to the stratification. The projective variety $X$ degenerates into a reduced union of the toric varieties to these semigroups $\Gamma_{\mathfrak C}$ via a Rees algebra construction. To each semigroup $\Gamma_{\mathfrak C}$ one can also associate a Newton-Okounkov body, which turns out to be a simplex. Hence for Seshadri stratifications, the Newton-Okounkov body of a flag of subvarieties is replaced by a simplicial complex.

For each \textit{normal} Seshadri stratification, \ie the semigroups $\Gamma_{\mathfrak C}$ are saturated, the fan of monoids $\Gamma$ defines a standard monomial theory on the homogeneous coordinate ring $\K[X]$. Every element in $\Gamma$ can be uniquely decomposed as a sum of indecomposable elements. When choosing a regular function $x_{\underline a}$ for each indecomposable element $\underline a \in \Gamma$ then the monomials in these functions generate $\K[X]$ as a vector space. Such a monomial $x_{\underline a^1} \cdots x_{\underline a^s}$ is called \textit{standard}, if and only if $\underline a^1 + \dots + \underline a^s$ is contained in $\Gamma$. The standard monomials form a basis of $\K[X]$. If the stratification is also \textit{balanced}, that is to say $\Gamma$ does not depend on the choice of the linearization $\geq^t$ of the partial order, then the resulting SMT is induced by the stratification itself.

To formally state and prove the above theorem, we therefore seek to construct a normal and balanced Seshadri stratification on $G/B$. This first requires answering the following question. 

\vskip 10pt
\noindent
\textbf{Question}. Can the Seshadri stratifications introduced in \cite{seshstrat} be generalized to multiprojective varieties, \ie projective varieties $X$ embedded into a product $\prod_{i=1}^m \PP(V_i)$ of projective spaces? These multiprojective stratifications should still induce a fan of monoids indexing a basis of the multihomogeneous coordinate ring.
\vskip 10pt

Therefore only the last section of this \whatisthis{} is devoted to the motivating example $G/B$. In all other sections we adapt the results of~\cite{seshstrat} to multiprojective varieties. We start by introducing multiprojective Seshadri stratifications. In contrast to the ordinary Seshadri stratifications in~\loccit{}, a stratum $X_p$ need not be a subvariety of $X$ itself, but of a projection of $X$ into a product $\prod_{i \in I_p} \PP(V_i)$ indexed by a non-empty subset $I_p \subseteq \set{1, \dots, m}$. The collection $\mathcal I = \set{I_p \mid p \in A}$ of these sets is called the \textit{index poset} of the stratification, which is an additional structure not visible for ordinary stratifications. By taking the affine multicones $\hat X_p$ of $X_p$, one can still view $\hat X_q$ as a closed subvariety of $\hat X_p$, if and only if $q \leq p$. The extremal function $f_p$ is chosen to be a multihomogeneous element of the multihomogeneous coordinate ring $\K[X]$, such that the degree of $f_p$ is compatible with the index set $I_p$.

For multiprojective stratifications one can still define a quasi-valuation $\mathcal V: \K[X] \setminus \set{0} \to \Q^A$ inducing a filtration on $\K[X]$ with at most one-dimensional leaves and a fan $\Gamma$ of finitely generated monoids. Most of the results in~\cite{seshstrat} can be directly generalized. The big difference to the original Seshadri stratifications lies in the Newton-Okounkov theory. Instead of a simplicial complex, we obtain a family of polytopal complexes, which is parametrized by the elements $\underline d \in \N_0^m$. Each semigroup $\Gamma_{\mathfrak C}$ to a maximal chain $\mathfrak C$ in $A$ defines a polytope $\Delta_{\mathfrak C}^{(\underline d)}$ and its faces correspond to certain subchains of $\mathfrak C$. Almost all of these polytopal complexes carry information about the variety $X$, e.\,g. its dimension, but in certain edge cases the dimension of the polytopal complex can actually be smaller than $\dim X$. Similar to \cite[Theorem 13.6]{seshstrat}, the volumes of these polytopal complexes with respect to certain lattices compute the leading term of the Hilbert polynomial. 

\vskip 10pt
\noindent
\textbf{Theorem} (Theorem~\ref{thm:G_R_formula}). The (homogeneous) leading term $G \in \Q[x_1, \dots, x_m]$ of the multivariate Hilbert polynomial of $X$ is given by
\begin{align*}
    G(\underline d) = \sum_{\mathfrak C} \operatorname{vol}(\mathrm{pr}_{\mathfrak C}(\Delta_{\mathfrak C}^{(\underline d)}))
\end{align*}
for all $\underline d \in \N_0^m$. The sum runs over all maximal chains in $A$ and $\mathrm{pr}_{\mathfrak C}$ is a suitable linear projection map.
\vskip 10pt

In the last section we return to the initial example $G/B$ and construct a multiprojective Seshadri stratification on this variety. In fact, we even construct a stratification on each partial flag variety in Dynkin type \texttt{A}, that means on all $G/Q$, where $Q \subseteq G$ is a parabolic subgroup containing $B$. This construction uses the combinatorial ideas from \cite{lakshmibaiGP4} and \cite{seshadri2016introduction}. As expected, the elements in the fan of monoids $\Gamma$ correspond to certain semistandard Young-tableaux. In addition, the stratification is normal and balanced and the resulting standard monomial theory coincides with the classical Hodge-Young theory of standard monomials in Pl\"ucker coordinates indexed by semistandard Young-tableaux.

It is a natural question, whether this stratification can be generalized to Schubert varieties, ideally even to other Dynkin types. The answer will be published in a separate article, where we construct multiprojective stratifications for certain representation-theoretically motivated embeddings into products of projective spaces. Being able to stratify Schubert varieties of $\mathrm{SL}_n(\K)$ will be the key for this generalization, as all difficulties already appear in type \texttt{A}. In the last subsection we therefore briefly illustrate the arising problems in the Schubert variety case.

This \whatisthis{} is based on the second and third chapter of the author's Ph.\,D. thesis \cite{thesis}. 

\section{Multiprojective Seshadri stratifications}

Throughout this article we fix an algebraically closed field $\K$ and a \textit{multiprojective} variety $X$, \ie a (Zariski-)closed subset $X \subseteq \PP(V_1) \times \dots \times \PP(V_m)$, where $V_1, \dots, V_m$ are finite-dimensional vector spaces over $\K$. Usually, we set $r$ to be the dimension of $X$. We included a section about multiprojective varieties in the Appendix~\ref{subsec:multiproj_varieties}, but for the most part they behave analogously to embedded projective varieties. There also exists a short list of notations after the appendix.

The multicone $\hat X$ of $X$ is a closed subvariety of the affine space $V = V_1 \times \dots \times V_m$. Let $R = \K[X] = \K[\hat X]$ be the multihomogeneous coordinate ring of $X$. We write $[k]$ for the set of all integers between $1$ and $k \in \N$. Each subset $I \subseteq [m]$ comes with the two natural projections
\begin{align}
    \label{eq:def_projection_maps}
    \pi_I: \prod_{i \in [m]} \PP(V_i) \twoheadrightarrow \prod_{i \in I} \PP(V_i) \quad \text{and} \quad \hat \pi_I: \prod_{i \in [m]} V_i \twoheadrightarrow \prod_{i \in I} V_i
\end{align}
as well as the multiprojective variety $X_I = \pi_I(X)$. Note that the multicone $\hat X_I$ of $X_I$ coincides with the image of $\hat X$ under the map $\hat \pi_I$. The surjection $\hat X \twoheadrightarrow \hat X_I$ induces an embedding of the multihomogeneous coordinate ring $\K[X_I]$ onto a graded subalgebra of $R$, namely the direct sum of all homogeneous components $R_{\underline d} \subseteq R$ for tuples $\underline d = (d_1, \dots, d_m) \in \N_0^m$ where $d_j = 0$ for all $j \notin I$.

Analogous to the definition of a Seshadri stratification in~\cite{seshstrat}, we fix a finite set $A$, a collection $\set{X_p \mid p \in A}$ of irreducible projective varieties, which are smooth in codimension one, and a collection of functions $\set{f_p \in R \mid p \in A}$ called \textbf{extremal functions}. The main difference to the original definition is that $X_p$ no longer needs to be a subvariety or even a subset of $X$. Instead we fix a third collection $\set{I_p \subseteq [m] \mid p \in A}$ of non-empty subsets of $[m]$ and require that $X_p$ is a closed subvariety of $X_{I_p} = \pi_{I_p}(X)$. If we view the affine space $\prod_{i \in I_p} V_i$ as a closed subvariety of $V$ via the linear embedding $\prod_{i \in I_p} V_i \hookrightarrow V$, then $\hat X_p$ can be seen as a closed subvariety of $\hat X$. This allows us to equip the set $A$ with the partial order $\leq$, such that $q \leq p$ if and only if $\hat X_q \subseteq \hat X_p$. The function $f_p$ needs to be non-constant, multihomogeneous and included in the subring $\K[X_{I_p}] \subseteq R$.

\begin{definition}[Multiprojective Seshadri stratification]
    \label{def:multiproj_seshadri_strat}
    These three collections of varieties, extremal functions and index sets are called a \textbf{(multiprojective) Seshadri stratification}, if there exists an element $p_{\text{max}} \in A$ with $I_{p_{\text{max}}} = [m]$ and $X_{p_{\text{max}}} = X$ and the following three conditions are fulfilled:
    \begin{enumerate}[label=(S\arabic{enumi})]
        \item \label{itm:seshadri_strat_a} If $q < p$ is a covering relation, then $\hat X_q \subseteq \hat X_p$ is a codimension one subvariety (where both are seen as subvarieties of $V$);
        \item \label{itm:seshadri_strat_b} The function $f_q$ vanishes on $\hat X_p$, if $q \nleq p$;
        \item \label{itm:seshadri_strat_c} For each $p \in A$ holds the set-theoretic equality
        \begin{equation}
        \label{eq:s3}
        \set{x \in \hat X \mid f_p(x) = 0} \cap \hat X_p = \set{0} \cup \bigcup_{p \;\text{covers}\; q} \hat X_q.
        \end{equation}
    \end{enumerate}
\end{definition}

Notice, that for $m = 1$ the notion of a multiprojective Seshadri stratification coincides with notion of a Seshadri stratification introduced in \cite{seshstrat}. In this case all strata $X_p$ for $p \in A$ are closed subvarieties of $X$, since $I_p = \set{1}$. 

The affine multicone $\hat X \subseteq V$ of $X$ is the affine cone of a projective variety $\widetilde X \subseteq \PP(V)$. Hence every multiprojective Seshadri stratification on $X \subseteq \prod_{i=1}^m \PP(V_i)$ can also be seen as a Seshadri stratification on $\widetilde X$. Therefore one can informally say that every result in \loccit{}, where the grading on $R$ is not involved, does also hold in the multiprojective case. As a first example: The poset $A$ is a graded poset of length $\dim \widetilde X = \dim \hat X - 1$, that is to say all maximal chains have length $\dim \widetilde X$. The rank of an element $p \in A$ is given by $r(p) = \dim \hat X_p - 1$.

\begin{proposition}
    Every multiprojective variety $X \subseteq \PP(V_1) \times \dots \times \PP(V_m)$ admits a Seshadri stratification.
\end{proposition}
\begin{proof}
    We embed the variety $X$ into the projective space over $W = V_1 \otimes \dots \otimes V_m$ via the Segre embedding, so we have two different coordinate rings of $X$: The multihomogeneous coordinate ring $R = \K[X]$ and the homogeneous coordinate ring $S$ of the embedding $X \hookrightarrow \PP(W)$. Now choose any Seshadri stratification of $X \subseteq \PP(W)$, which exists by \cite[Proposition~2.11]{seshstrat}. Hence for each $p \in A$ we have the closed, irreducible subvariety $X_p \subseteq X$ which is smooth in codimension one and the extremal function $f_p \in S$. This function can be pulled back to a multihomogeneous function in $R$ and its degree is a multiple of $(1, \dots, 1)$. Clearly the conditions~\ref{itm:seshadri_strat_b} to \ref{itm:seshadri_strat_c} are preserved under the pullback.  We also need to define a subset of $[m]$ for all $p \in A$: Here we take $I_p = [m]$.

    However, for $m \geq 2$ we do not obtain a multiprojective Seshadri stratification on $X$ in this way, as the dimension of the multicones $\hat X_p$ for $p \in A$ minimal is greater than $1$. Indeed, this multicone $\hat X_p$ is of the form
    \begin{align*}
        \hat X_p = L_1^{(p)} \times \dots \times L_m^{(p)},
    \end{align*}
    where $L_i^{(p)}$ is a one-dimensional linear subspace of $V_i$. Therefore we need to extend the graded poset $A$ by $m-1$ additional ranks. Set-theoretically this extension is of the form
    \begin{align*}
        \overline A = A \cup \set{L_1^{(p)} \times \dots \times L_i^{(p)} \mid \text{$p \in A$ minimal and $1 \leq i \leq m-1$}}.
    \end{align*}
    For each $q \in \overline A \setminus A$ of the form $L_1^{(p)} \times \dots \times L_i^{(p)}$ we define the subset $I_q = [i]$ and the projective variety $X_q = \PP(L_1^{(p)}) \times \dots \times \PP(L_i^{(p)}) \subseteq X_{I_q}$. Note that we mentioned at the beginning that in general $X_p$ does not need to be a subvariety of $X$, but there exists a subset $I \subseteq [m]$ such that $X_p$ is a closed subvariety of $X_I = \pi_I(X)$. This case did not occur so far because the stratification is induced by an ordinary Seshadri stratification (as it was defined in~\cite{seshstrat}). The variety $X_q$ only consists of one point, so it is irreducible and smooth in codimension one. The affine variety $q$ is the multicone over $X_q$. The partial order on $\overline A$ is now determined by the equivalence of $q \leq p$ and $\hat X_q \subseteq \hat X_p$.

    It remains to define the extremal functions for the elements in $\overline A \setminus A$. Let $\mathcal L_i$ be the set of all lines $L_i^{(p)}$ for $p \in A$ minimal. For each $i \in [m-1]$ and $L \in \mathcal L_i$ we choose a linear function $h_{L} \in V_i^*$ which vanishes on $L$ and that does not vanish on all other lines in $\mathcal L_i$. To an element $q \in \overline A \setminus A$ of the form $L_1^{(p)} \times \dots \times L_i^{(p)}$ we then associate the function
    \begin{align*}
        f_q = \prod_{j \in [i]} \prod_{L \in \mathcal L_i \atop L \neq L_i^{(p)}} h_{L}.
    \end{align*}
    This definition ensures that both conditions~\ref{itm:seshadri_strat_b} and \ref{itm:seshadri_strat_c} are fulfilled. The extremal functions $f_p$ for $p \in A$ also vanish on all varieties $\hat X_q$ for $q \in \overline A \setminus A$. We therefore have constructed a multiprojective Seshadri stratification on $X$.
\end{proof}

\begin{example}
    \label{ex:strat_y_1}
    Let $X$ be the image of the closed diagonal embedding $\PP^1 \hookrightarrow \PP(V_1) \times \PP(V_2)$ for $V_1 = V_2 = \K^2$. The coordinate ring $\K[V] = \K[x_0, x_1, y_0, y_1]$ of $V = V_1 \times V_2$ is graded with $\deg x_0 = \deg x_1 = (1,0)$ and $\deg y_0 = \deg y_1 = (0,1)$ and the vanishing ideal of $X$ is equal to $I_\PP(X) = (x_0 y_1 - x_1 y_0)$. We write $I(-)$ and $I_\PP(-)$ to distinguish between affine and projective vanishing ideals (see Appendix~\ref{subsec:multiproj_varieties}). Analogously, we differentiate between the affine vanishing set $V(-)$ and the projective vanishing set $V_\PP(-)$. The multicone $\hat X$ of $X$ is given by the vanishing set $V(x_0 y_1 - x_1 y_0) \subseteq V = \A^2 \times \A^2$. We define a poset $A$ via the Hasse-diagram
    
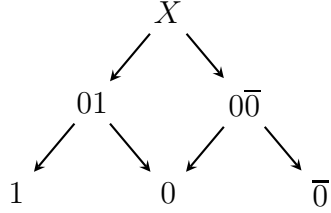
\begin{center}
        \begin{tikzpicture}
        \node (X) at (0,0) {$X$};
        \node (1100) at (-1,-1.2) {$01$};
        \node (1010) at (1,-1.2) {$0 \overline 0$};
        \node (0100) at (-2,-2.4) {$1$};
        \node (1000) at (0,-2.4) {$0$};
        \node (0010) at (2,-2.4) {$\overline 0$};
        
        \draw [thick, -stealth] (X) -- (1010);
        \draw [thick, -stealth] (X) -- (1100);
        \draw [thick, -stealth] (1010) -- (1000);
        \draw [thick, -stealth] (1010) -- (0010);
        \draw [thick, -stealth] (1100) -- (1000);
        \draw [thick, -stealth] (1100) -- (0100);
        \end{tikzpicture}
    \end{center}
    and choose the following index sets, strata and extremal functions:
    \begin{center}
        \begin{tabular}{c wc{1.2cm} wc{6.1cm} c}
            $p \in A$ & $I_p$ & $X_p$ & $f_p$ \\[6pt] 
            $X$ & $\set{1,2}$ & $X \subseteq \PP(V_1) \times \PP(V_2)$ & $y_1$ \\ 
            $01$ & $\set{1}$ & $\PP(V_1)$ & $x_0 x_1$ \\ 
            $0 \overline 0$ & $\set{1,2}$ & $V_\PP(x_1) \times V_\PP(y_1) \subseteq \PP(V_1) \times \PP(V_2)$ & $x_0 y_0$ \\ 
            $1$ & $\set{1}$ & $V_\PP(x_0) \subseteq \PP(V_1)$ & $x_1$ \\ 
            $0$ & $\set{1}$ & $V_\PP(x_1) \subseteq \PP(V_1)$ & $x_0$ \\ 
            $\overline 0$ & $\set{2}$ & $V_\PP(y_1) \subseteq \PP(V_2)$ & $y_0$
        \end{tabular}
    \end{center}
    This data defines a Seshadri stratification on $X$, which can be summarized by a diagram of all multicones $\hat X_p$ and $f_p$ for $p \in A$:
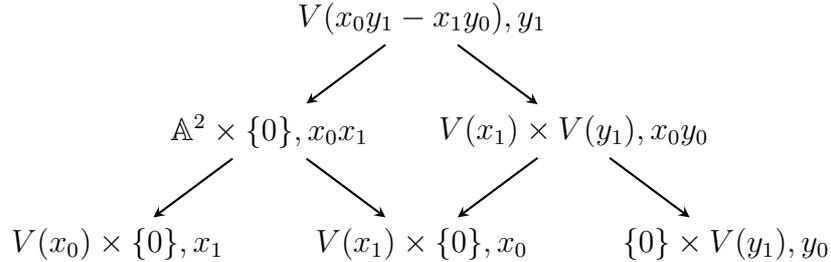
    \begin{center}
        \begin{tikzpicture}
        \node (X) at (0,0) {$V(x_0 y_1 - x_1 y_0), y_1$};
        \node (1100) at (-2,-1.5) {$\A^2 \times \set{0}, x_0 x_1$};
        \node (1010) at (2,-1.5) {$V(x_1) \times V(y_1), x_0 y_0$};
        \node (0100) at (-4,-3) {$V(x_0) \times \set{0}, x_1$};
        \node (1000) at (0,-3) {$V(x_1) \times \set{0}, x_0$};
        \node (0010) at (4,-3) {$\set{0} \times V(y_1), y_0$};
        
        \draw [thick, -stealth] (X) -- (1010);
        \draw [thick, -stealth] (X) -- (1100);
        \draw [thick, -stealth] (1010) -- (1000);
        \draw [thick, -stealth] (1010) -- (0010);
        \draw [thick, -stealth] (1100) -- (1000);
        \draw [thick, -stealth] (1100) -- (0100);
        \end{tikzpicture}
    \end{center}
\end{example}

\begin{example}
    \label{ex:strat_y_0_y_1}
    Of course, for each multiprojective variety there can exist many different Seshadri stratifications. For example, there is another stratification on the variety $X = V_\PP(x_0 y_1 - x_1 y_0) \subseteq \PP^1 \times \PP^1$ with underlying poset
    
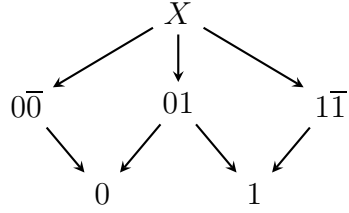
\begin{center}
        \begin{tikzpicture}
        \node (X) at (0,0) {$X$};
        \node (1010) at (-2,-1.2) {$0 \overline 0$};
        \node (1100) at (0,-1.2) {$01$};
        \node (0101) at (2,-1.2) {$1 \overline 1$};
        \node (1000) at (-1,-2.4) {$0$};
        \node (0100) at (1,-2.4) {$1$};
        
        \draw [thick, -stealth] (X) -- (1010);
        \draw [thick, -stealth] (X) -- (1100);
        \draw [thick, -stealth] (X) -- (0101);
        \draw [thick, -stealth] (1010) -- (1000);
        \draw [thick, -stealth] (1100) -- (1000);
        \draw [thick, -stealth] (1100) -- (0100);
        \draw [thick, -stealth] (0101) -- (0100);
        \end{tikzpicture}
    \end{center}
    that is defined via the following diagram of multicones and extremal functions:
    
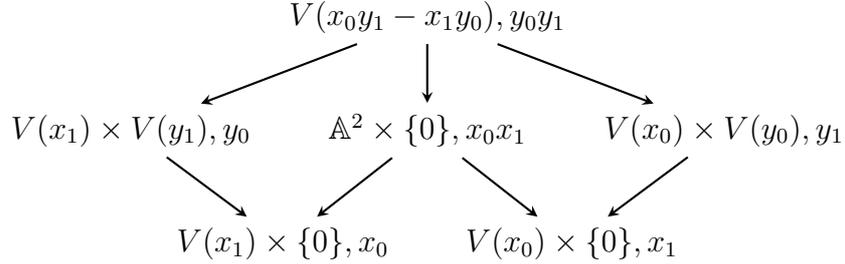
\begin{center}
        \begin{tikzpicture}
        \node (X) at (0,0) {$V(x_0 y_1 - x_1 y_0), y_0 y_1$};
        \node (1010) at (-3.9,-1.5) {$V(x_1) \times V(y_1), y_0$};
        \node (1100) at (0,-1.5) {$\A^2 \times \set{0}, x_0 x_1$};
        \node (0101) at (3.9,-1.5) {$V(x_0) \times V(y_0), y_1$};
        \node (1000) at (-1.9,-3) {$V(x_1) \times \set{0}, x_0$};
        \node (0100) at (1.9,-3) {$V(x_0) \times \set{0}, x_1$};
        
        \draw [thick, -stealth] (X) -- (1010);
        \draw [thick, -stealth] (X) -- (1100);
        \draw [thick, -stealth] (X) -- (0101);
        \draw [thick, -stealth] (1010) -- (1000);
        \draw [thick, -stealth] (1100) -- (1000);
        \draw [thick, -stealth] (1100) -- (0100);
        \draw [thick, -stealth] (0101) -- (0100);
        \end{tikzpicture}
    \end{center}
\end{example}

In contrast to the Seshadri stratifications introduced in~\cite{seshstrat}, their multiprojective generali\-zations have an additional underlying structure, namely the poset
\begin{align}
    \label{eq:def_mathcal_I}
    \mathcal I = \{ I_p \subseteq [m] \mid p \in A \},
\end{align}
which is ordered by inclusion. We call it the \textbf{index poset}.

\begin{lemma}
    \label{lem:properties_A_mathcal_I}
    The map $A \to \mathcal I$, $p \mapsto I_p$ is monotone and has the following properties:
    \begin{enumerate}[label=(\alph{enumi})]
        \item \label{itm:properties_A_mathcal_I_a} Let $q < p$ be a covering relation in $A$. Then $I_p \setminus I_q$ contains at most one element. In the case $I_q \neq I_p$ it holds $\pi_{I_q}(X_p) = X_q$.
        \item \label{itm:properties_A_mathcal_I_b} If $p \in A$ is a minimal element, then $I_p$ is a one-element set.
    \end{enumerate}
    In particular, $\mathcal I$ is a graded poset of length $m-1$.
\end{lemma}
\begin{proof}
    The map $A \to \mathcal I$ is monotone, since for all $q \leq p$ in $A$ we have the inclusion $\hat X_q \subseteq \hat X_p$ of their multicones and this implies $I_q \subseteq I_p$.
    \begin{enumerate}[label=(\alph{enumi})]
        \item Let $q < p$ be a covering relation and suppose that $I_q$ is a proper subset of $I_p$. For every subset $J \subseteq I_p$ we have the linear projection $\hat \pi_{J}$ (see (\ref{eq:def_projection_maps})). If
        \begin{align*}
            I_q = J_0 \subsetneq J_1 \subsetneq \dots \subsetneq J_s = I_p
        \end{align*}
        is a chain in $\mathcal I$ then we have the closed, irreducible subvarieties
        \begin{align*}
            \hat X_q \subseteq \hat \pi_{J_0}(\hat X_p) \subsetneq \dots \subsetneq \hat \pi_{J_s}(\hat X_p) = \hat X_p,
        \end{align*}
        which we see as subvarieties of $V$. As $\hat X_q$ is of codimension one in $\hat X_p$, it follows $I_p \setminus I_q = \set{i}$ for some $i \in I_p$ and $\hat X_q = \hat \pi_{I_q}(\hat X_p)$, because their dimensions agree.
        \item The condition~\ref{itm:seshadri_strat_c} implies, that an element $p \in A$ is minimal, if and only if the vanishing set of $f_p$ inside of $\hat X_p$ is just the point $0 \in V$, because this is the only point in $\hat X_p$, that does not belong to a projective subvariety of $X_J$ for some non-empty $J \subseteq I_p$. But as $f_p$ is multihomogeneous and non-constant, its vanishing set $V(f_p) \subseteq \hat X$ is a multicone (\ie stable under the $(\K^\times)^m$-action) and its irreducible components have codimension one in $\hat X$. Hence $V(f_p) \cap \hat X_p$ can only be zero, when $\dim X_p = 0$ and $|I_p| = 1$. It now follows from \ref{itm:properties_A_mathcal_I_a} and \ref{itm:properties_A_mathcal_I_b} that every maximal chain in $\mathcal I$ contains exactly $m$ elements, so $\mathcal I$ is graded of length $m-1$. \hfill\qedhere
    \end{enumerate}
\end{proof}

For multiprojective stratifications we have the following new kind of covering relations in $A$ which do not appear for $m = 1$.

\begin{lemma}
    \label{lem:projective_covering_relation}
    Let $q < p$ be a covering relation in $A$ with $I_p \setminus I_q = \set{i}$. 
    \begin{enumerate}[label=(\alph{enumi})]
        \item \label{itm:projective_covering_relation_a} The algebra $\K[\hat X_q]$ can be seen as an $\N_0^{I_q}$-graded subalgebra of $\K[\hat X_p]$ and it holds
        \begin{align*}
            \K[\hat X_q] \cong \bigoplus_{\underline d \in \N_0^m \atop d_i = 0} \K[\hat X_p]_{\underline d} \quad \text{and} \quad I(\hat X_q) = \bigoplus_{\underline d \in \N_0^m \atop d_i > 0} \K[\hat X_p]_{\underline d}.
        \end{align*}
        \item \label{itm:projective_covering_relation_b} The vanishing multiplicity of a multihomogeneous function $g \in \K[\hat X_p] \setminus \set{0}$ along the prime divisor $\hat X_q \subseteq \hat X_p$ is equal to the $i$-th component of $\deg g \in \N_0^m$. In particular, the $i$-th component of $\deg f_p$ is non-zero.
    \end{enumerate}
\end{lemma}
\begin{proof}
    \begin{enumerate}[label=(\alph{enumi})]
        \item This first statement is immediate from the equality $X_q = \pi_{I_q}(X_p)$.
        \item The discrete valuation ring $\mathcal O_{\hat X_p, \hat X_q} \subseteq \K(\hat X_p)$ is isomorphic to the localization of $\K[\hat X_p]$ at the prime ideal $I(\hat X_q)$. By viewing $g$ as an element of $\mathcal O_{\hat X_p, \hat X_q} \supseteq \K[\hat X_p]$, one can characterize the vanishing multiplicity of $g$ along $\hat X_q$ as the unique integer $n \in \N_0$ with $(g) = \mathfrak m^n$, where $\mathfrak m$ denotes the unique maximal ideal in $\mathcal O_{\hat X_p, \hat X_q}$. As the algebra $\K[\hat X_p]$ is generated in total degree one, it follows
        \begin{align*}
            \mathfrak m^n = \bigoplus_{\underline d \in \N_0^m \atop d_i \geq n} \K[\hat X_p]_{\underline d} \subseteq \mathcal O_{\hat X_p, \hat X_q}.
        \end{align*}
        Let $\deg g = (c_1, \dots, c_m)$. Clearly we have $(g) \subseteq \mathfrak m^{c_i}$ but $(g) \nsubseteq \mathfrak m^{c_i+1}$. Hence $(g) = \mathfrak m^{c_i}$, since every ideal of $\mathcal O_{\hat X_p, \hat X_q}$ is a power of $\mathfrak m$. \hfill\qedhere
    \end{enumerate}
\end{proof}

\section{The quasi-valuation and its associated graded algebra}

Until the end of Section~\ref{sec:LS-type} we fix a Seshadri stratification on $X \subseteq \prod_{i=1}^m \PP(V_i)$.

In this section we summarize some constructions and results from \cite{seshstrat}, since they are crucial for this \whatisthis{}. Among these results are the quasi-valuation $\mathcal V$ and the properties of the associated graded algebra. It is strongly recommended to read the original papers, as we cannot do justice to their results on just a few pages and this section mainly serves as a reminder for all the notation introduced for Seshadri stratifications.

We fix the following notation: \label{txt:def_support}If $K$ is any field of characteristic zero and $S$ is a finite set, then we write $K^S$ for the vector space over $K$ with basis $\set{e_s \mid s \in S}$ indexed by $S$. Let $\N_0^S$ be the monoid generated by these basis elements and $\Z^S \subseteq K^S$ be the smallest group containing $\N_0^S$. For each element $x = \sum_{s \in S} x_s e_s \in K^S$ with coefficients $x_s \in K$ the set
\begin{align*}
    \supp x = \set{s \in S \mid x_s \neq 0}
\end{align*}
is called the \textit{support} of $x$.

By definition, the multicone $\hat X_q$ is a prime divisor of $\hat X_p$ for every covering relation $q < p$ in $A$. If one extends the poset $A$ by a unique minimal element $p_{-1}$ with associated index set $I_{p_{-1}} = \varnothing$, then the multicone $\hat X_{p_{-1}} = \set{0}$ is a prime divisor of $\hat X_p \cong \A^1$ for each minimal element $p \in A$. To each covering relation $p > q$ in the extended poset $\hat A = A \cup \set{p_{-1}}$ we have an associated valuation, namely the discrete valuation
\begin{align*}
    \nu_{p, q}: \K(\hat X_p) \setminus \set{0} \to \Z,
\end{align*}
sending a non-zero, rational function $g$ to its vanishing multiplicity at the prime divisor $\hat X_q \subseteq \hat X_p$. Its value 
\begin{align*}
    b_{p,q} = \nu_{p,q}(f_p\big|_{\hat X_p}) \in \N
\end{align*}
at the extremal function $f_p$ is called the \textit{bond} of the covering relation $q < p$. If $p$ is minimal in $A$, then $b_{p, p_{-1}}$ coincides with the total degree $\vert \mkern-1mu \deg f_p \mkern1mu \vert$, which is the sum of all entries in the degree $\deg f_p \in \N_0^m$. 

Every Seshadri stratification gives rise to a collection of valuations on $R$, one for each maximal chain $\mathfrak C$ in $A$. Let $p_r > \dots > p_0$ be the elements of $\mathfrak C$. To a regular function $g \in R \setminus \set{0}$ one associates a sequence $g_{\mathfrak C} = (g_r, \dots, g_0)$ of rational functions inductively via $g_r \coloneqq g$ and 
\begin{align*}
    g_{i-1} = \frac{ g_i^{\mkern-1mu b_{p_i, p_{i-1}}} }{ f_{p_i}^{\mkern2mu \nu_{p_i, p_{i-1}}(g_i)} } \bigg|_{\hat X_{p_{i-1}}} \in \K(\hat X_{p_{i-1}}).
\end{align*}
for $i = r, \dots, 1$. Further one defines the element
\begin{align*}
    \mathcal V_{\mathfrak C}(g) = \sum_{j=0}^r \frac{\nu_{p_j, p_{j-1}}(g_j)}{\prod_{k=j}^r b_{p_k, p_{k-1}}} \, e_{p_j} \in \Q^{\mathfrak C}.
\end{align*}
By this definition, each extremal function $f_p$ for $p \in \mathfrak C$ is mapped to the vector $\mathcal V_{\mathfrak C}(f_p) = e_p$. We equip the abelian group $\Q^{\mathfrak C}$ with the lexicographic order induced by the total order on the maximal chain $\mathfrak C$, \ie for all elements $\underline a = \sum_{i=0}^r a_i e_{p_i}, \underline b = \sum_{i=0}^r b_i e_{p_i}$ in $\Q^{\mathfrak C}$ it holds
\begin{align*}
    \underline a \geq \underline b \quad \Longleftrightarrow \quad \text{$\underline a = \underline b$ or $a_i > b_i$ for the maximal index $i \in \set{0, \dots, r}$ with $a_i \neq b_i$}.
\end{align*}
Then the map $\mathcal V_{\mathfrak C}: R \setminus \set{0} \to \Q^{\mathfrak C}$ is a valuation. Chiriv{\`i}, Fang and Littelmann also gave another, equivalent definition in \cite{seshstrat}, which we do not use here, as it is less suited for computations. Note that, by its definition, $\mathcal V_{\mathfrak C}$ takes values in the lattice
\begin{align}
    \label{eq:def_L_C}
    L^{\mathfrak C} = \set{ (a_r, \dots, a_0) \in \Q^{\mathfrak C} \mid b_r \cdots b_{i+1} b_i a_i \in \Z \mkern5mu \forall i = 0, \dots, r}.
\end{align}
In general, the lattice $L_{\mathcal V}^{\mathfrak C}$ generated by the valuation monoid $\mathbb V_{\mathfrak C}(X) = \set{\mathcal V_{\mathfrak C}(g) \in \Q^{\mathfrak C} \mid g \in R \setminus \set{0}}$ is not equal to the lattice $L^{\mathfrak C}$. With the following results from~\cite{seshstrat} one can determine the lattice $L_{\mathcal V}^{\mathfrak C}$.

\begin{proposition}[{\cite[Lemma 6.12, Propositions 6.13 and 6.14]{seshstrat}}]
    \label{prop:valuation_lattice}
    There exist rational functions $F_r, \dots F_0 \in \K(\hat X) \setminus \set{0}$, such that their valuations are of the form
    \begin{align*}
        \mathcal V_{\mathfrak C}(F_j) = \sum_{i=0}^r a_{i,j} \mkern1mu e_{p_j}
    \end{align*}
    with coefficients $a_{i,j} \in \K$, $a_{j,j} = b_{p_j, p_{j-1}}^{-1}$, $a_{i,j} = 0$ for all $i > j$. For each such choice of functions $F_r, \dots, F_0$ the matrix $(a_{i,j})_{i,j = 0, \dots, r}$ is invertible and the entries of its inverse matrix $B_{\mathfrak C}$ are integers. Furthermore, an element $v = a_r e_{p_r} + \dots + a_0 e_{p_0} \in \Q^{\mathfrak C}$ is contained in the lattice $L_{\mathcal V}^{\mathfrak C}$, if and only if 
    \begin{align*}
        B_{\mathfrak C} \cdot \begin{pmatrix}
            a_r \\ \vdots \\ a_0
        \end{pmatrix} \in \Z^{\mathfrak C}.
    \end{align*}
\end{proposition}

\begin{remark}
    \label{rem:induced_strat}
    Every element $p \in A$ induces a Seshadri stratification on the multiprojective variety $X_p \subseteq \prod_{i \in I_p} \PP(V_i)$ via the poset $A_p = \set{q \in A \mid q \leq p}$, where we take the same strata, extremal functions and index sets as in the stratification on $X$. By its definition, the valuation $\mathcal V_{\mathfrak C}$ is compatible in the following sense with the valuation $\mathcal V_{\mathfrak C_p}$ of the induced stratification along the maximal chain $\mathfrak C_p = \mathfrak C \cap A_p$: For every $g \in R \setminus \set{0}$, that does not vanish identically on $\hat X_p$, the valuation $\mathcal V_{\mathfrak C_p}(g\big|_{\hat X_p}) \in \Q^{\mathfrak C_p}$ coincides with $\mathcal V_{\mathfrak C}(g)$, when extended by zeros to an element of $\Q^{\mathfrak C}$.
\end{remark}

The collection of all valuations $\mathcal V_{\mathfrak C}$ define a quasi-valuation $\mathcal V$, which respects the structure of the whole poset $A$, not just of one maximal chain. A \textbf{quasi-valuation} is defined similar to valuation, only the condition $\mathcal V(gh) = \mathcal V(g) + \mathcal V(h)$ for all $g, h \in R$ with $gh \neq 0$ is replaced by the inequality $\mathcal V(gh) \geq \mathcal V(g) + \mathcal V(h)$. To obtain this quasi-valuation one needs to extend $\mathcal V_{\mathfrak C}$ to a valuation $R \setminus \set{0} \to \Q^{\mathfrak C} \hookrightarrow \Q^A$, such that all valuations take values in the same abelian group. In order to make sense of this, we need a total order on $\Q^A$ such that each linear inclusion $\Q^{\mathfrak C} \hookrightarrow \Q^A$ is monotone. In general, there is no natural candidate for this total order. For this reason, one needs to choose and fix a total order $\geq^t$ on $A$ linearizing the partial order, \ie for each elements $p, q \in A$ the relation $p \geq q$ implies $p \geq^t q$. This total order induces the lexicographic order on $\Q^A$ and each map $\mathcal V_{\mathfrak C}: R \setminus \set{0} \to \Q^A$ is a valuation. One obtains the quasi-valuation $\mathcal V$ by taking their minimum with respect to this total order on $\Q^A$:
\begin{align*}
    \mathcal V: R \setminus \set{0} \to \Q^A, \quad g \mapsto \operatorname{min} \set{\mathcal V_{\mathfrak C}(g) \mid \text{$\mathfrak C$ maximal chain in $A$}}.
\end{align*}
Hence the quasi-valuation depends on the choice of this total order $\geq^t$ on $A$.

There is also the following inductive way of describing the quasi-valuation $\mathcal V$. Let $p$ be any element in $A$, $g \in \K(\hat X_p)$ be a non-zero rational function. We write $\mathcal V_p$ for the quasi-valuation on the induced Seshadri stratification on $X_p$ with underlying poset $A_p = \set{q \in A \mid q \leq p}$. Then it holds
\begin{align}
    \label{eq:inductive_def_quasi_val}
    \mathcal V_p(g) = \frac{\nu_{p, q}(g)}{b_{p,q}} e_p + \frac{1}{b_{p,q}} \mathcal V_q \big( \frac{g^{b_{p,q}}}{f_p^{\nu_{p, q}(g)}} \big\vert_{\hat X_q} \big),
\end{align}
where $q$ is the unique minimal element covered by $p$ with respect to the total order $\geq^t$, such that it holds
\begin{align*}
    \frac{\nu_{p, q}(g)}{b_{p,q}} = \operatorname{min} \Set{ \frac{\nu_{p, q'}(g)}{b_{p,q'}} \ \middle\vert \ \text{$q' \in A$ covered by $p$}}.
\end{align*}

The quasi-valuation $\mathcal V$ has the following important properties, which we use many times throughout this \whatisthis{} without mention (see \cite[Section 8]{seshstrat}). 
\begin{itemize}
    \item The values of $\mathcal V$ have non-negative entries, \ie the quasi-valuation $\mathcal V(g)$ of every function $g \in R \setminus \set{0}$ is contained in the non-negative orthant $\Q_{\geq 0}^A$.
    \item One can characterize combinatorially for which maximal chains $\mathfrak C$ the quasi-valuation attains its minimum. For each $g \in R \setminus \set{0}$ it holds $\mathcal V_{\mathfrak C}(g) = \mathcal V(g)$, if and only if the support $\supp \mathcal V(g) \subseteq A$ lies in $\mathfrak C$. As a consequence: If $g, h \in R$ are non-zero and there exists a maximal chain $\mathfrak C$ containing both $\supp \mathcal{V}(g)$ and $\supp \mathcal{V}(h)$, then the quasi-valuation is additive, \ie we have $\mathcal{V}(gh) = \mathcal{V}(g) + \mathcal{V}(h)$. 
    \item Every extremal function $f_p$ for $p \in A$ has the quasi-valuation $\mathcal V(f_p) = e_p$, so the support is given by $\supp \mathcal V(f_p) = \set{p}$. In particular: If $p_1, \dots, p_s \in A$ are contained in a chain in $A$ and $n_1, \dots, n_s \in \N_0$, then it follows
    \begin{align*}
        \mathcal V(f_{p_1}^{n_1} \cdots f_{p_s}^{n_s}) = \sum_{i=1}^s n_i e_{p_i}.
    \end{align*}
\end{itemize}

\begin{example}
    \label{ex:quasi_val_x_0_y_1}
    The regular function $g = x_0 y_1 \in \K[X]$ in the Seshadri stratification from Example~\ref{ex:strat_y_0_y_1} has the quasi-valuation $\mathcal V(g) = \tfrac12 e_X + \tfrac12 e_{01}$, independent of the choice of the total order $\geq^t$. To see this, we choose the following three parametrizations of open subsets of $\hat X = V(x_0 y_1 - x_1 y_0)$:
    \begin{align*}
        \phi_{0 \overline 0}: \K^\times \times \K^\times \times \K \to \hat X,& \quad (x, y, t) \longmapsto ((x, tx), (y, ty)),\\
        \phi_{01}: \K^\times \times \K^\times \times \K \to \hat X,& \quad (x, y, t) \longmapsto ((x, y), (tx, ty)),\\
        \phi_{1 \overline 1}: \K^\times \times \K^\times \times \K \to \hat X,& \quad (x, y, t) \longmapsto ((tx, x), (ty, y)).
    \end{align*}
    They are defined such that $\phi_q(\K^\times \times \K^\times \times \set{0})$ is equal to the intersection of the image of $\phi_q$ with the multicone $\hat X_q$ for each covering relation $q < X$ in $A$. The vanishing multiplicity of $g$ at the divisor $\hat X_q$ then agrees with the exponent of $t$ in the Laurent polynomial $g \circ \phi_q \in \K[x^{\pm 1}, y^{\pm 1}, t]$. We therefore have
    \begin{align*}
        \frac{\nu_{X, 0 \overline 0}(g)}{b_{X, 0 \overline 0}} = 1 > \frac{\nu_{X, 01}(g)}{b_{X, 01}} = \frac12 < \frac{\nu_{X, 1 \overline 1}(g)}{b_{X, 1 \overline 1}} = 1.
    \end{align*}
    By the characterization of the quasi-valuation $\mathcal V$ from equation~(\ref{eq:inductive_def_quasi_val}) it now follows $\mathcal V(g) = \tfrac12 e_X + \tfrac12 \mathcal V_{01}(g_1)$, where $g_1$ is the rational function
    \begin{align*}
        g_1 = \left(\frac{g^{\mkern2mu b_{X, 01}}}{f_{X}}\right)\bigg|_{\hat X_{01}} = \left( \frac{x_0^2 y_1^2}{y_0 y_1} \right)\bigg|_{\hat X_{01}} = x_0 x_1
    \end{align*}
    on $\hat X_{01}$. As $g_1$ is the restriction of the extremal function $f_{01}$ to $\hat X_{01}$, we have $\mathcal V_{01}(g_1) = e_{01}$.
\end{example}

The image of the quasi-valuation is denoted by $\Gamma = \set{\mathcal V(g) \in \Q^A \mid g \in R \setminus \set{0}}$. For each (not necessarily maximal) chain $C$ in $A$ the subset
\begin{align*}
    \Gamma_{C} = \set{\underline a \in \Gamma \mid \supp \underline a \subseteq C}
\end{align*}
is a finitely generated monoid. In~\cite{seshstrat} this was only shown if $C$ is a maximal chain, but it implies that $\Gamma_C$ is finitely generated as well, since its elements have non-negative entries. The set $\Gamma$ is called the \textit{fan of monoids} of the Seshadri stratification, since it is the union of all monoids $\Gamma_{C}$ and the cones in $\R^A$ generated by these monoids form a fan.

The quasi-valuation $\mathcal V: R \setminus \set{0} \to \Q^m$ induces a filtration on $R$ by the subrings
\begin{align*}
    R_{\geq \underline a} = \set{g \in R \setminus \set{0} \mid \mathcal V(g) \geq \underline a} \cup \set{0}
\end{align*}
for $\underline a \in \Gamma$. Since $\mathcal V(g)$ only has non-negative entries for all $g \in R \setminus \set{0}$, these subrings are ideals in $R$. The quotient of $R_{\geq \underline a}$ by the ideal $R_{> \underline a} = \set{g \in R \setminus \set{0} \mid \mathcal V(g) > \underline a} \cup \set{0}$ is one-dimensional for every $\underline a \in \Gamma$. They are called the \textit{leaves} of the quasi-valuation $\mathcal V$. Let
\begin{align*}
    \mathrm{gr}_{\mathcal V} R = \bigoplus_{\underline a \in \Gamma} R_{\geq \underline a} / R_{> \underline a}
\end{align*}
be the associated graded algebra. For each chain $C$ in $A$ it contains the subalgebra
\begin{align*}
    \mathrm{gr}_{\mathcal V, C} R = \bigoplus_{\underline a \in \Gamma_{C}} R_{\geq \underline a} / R_{> \underline a} \subseteq \mathrm{gr}_{\mathcal V} R,
\end{align*}
which is isomorphic to the semigroup algebra $\K[\Gamma_{C}]$ as a $\Gamma_{C}$-graded algebra. It is a finitely generated integral domain, so it gives rise to a toric variety $\Spec \mathrm{gr}_{\mathcal V, C} R$. The fact that the associated graded algebra is the union of all these subalgebras $\mathrm{gr}_{\mathcal V, C} R \cong \K[\Gamma_{C}]$ suggests that there is also a combinatorial way of describing the associated graded algebra by gluing the semigroup algebras $\K[\Gamma_{C}]$ into the \textit{fan algebra} of $\Gamma$. It is defined as the algebra
\begin{align*}
    \K[\Gamma] = \K[x_{\underline a} \mid \underline a \in \Gamma]/I(\Gamma),
\end{align*}
where $I(\Gamma)$ is the ideal generated by all elements of the form
\begin{align*}
    \begin{cases}
        x_{\underline a} x_{\underline b} - x_{\underline a + \underline b}, & \text{if there exists a chain $C$ in $A$ containing $\supp \underline a$ and $\supp \underline b$}, \\
        x_{\underline a} x_{\underline b}, & \text{else}
    \end{cases}
\end{align*}
with $\underline a, \underline b \in \Gamma$. For each chain $C$ in $A$ the fan algebra contains the subalgebra
\begin{align*}
    \bigoplus_{\underline a \in \Gamma_{C}} \K x_{\underline a} \subseteq \K[\Gamma],
\end{align*}
which is isomorphic to the semigroup algebra $\K[\Gamma_{C}]$.

All leaves of the quasi-valuation $\mathcal V$ are at most one-dimensional. Hence choosing a regular function $g_{\underline a} \in R$ with $\mathcal V(g_{\underline a}) = \underline a$ for each $\underline a \in \Gamma$ yields a basis
\begin{align*}
    \mathbb B = \set{g_{\underline a} \mid \underline a \in \Gamma}
\end{align*}
of $R$ as a vector space over $\K$ and the elements $\overline g_{\underline a}$ from a basis of the associated graded algebra $\mathrm{gr}_{\mathcal V} R$.

\begin{theorem}[{\cite[Theorem 11.1]{seshstrat}}]
    \label{thm:fan_algebra_cong_gr_V_R}
    There exist scalars $c_{\underline a} \in \K^\times$ such that the map
    \begin{align*}
        \K[\Gamma] \to \mathrm{gr}_{\mathcal V} R, \quad x_{\underline a} \mapsto c_{\underline a} \overline g_{\underline a}
    \end{align*}
    is an isomorphism of algebras.
\end{theorem}

The concepts of normal and balanced Seshadri stratifications were introduced in \cite[Sections 13, 15]{seshstrat}. They can also be used in the multiprojective case. 

\begin{definition}
    A multiprojective Seshadri stratification is called 
    \begin{enumerate}[label=(\alph{enumi})]
        \item \textbf{normal}, if $\Gamma_{\mathfrak C}$ is saturated for every maximal chain $\mathfrak C$, \ie it is equal to the intersection of the lattice $\mathcal L^{\mathfrak C}$ generated by $\Gamma_{\mathfrak C}$ with the positive orthant $\Q^{\mathfrak C}_{\geq 0}$;
        \item \textbf{balanced}, when the fan of monoids $\Gamma$ is independent of the choice of the total order $\geq^t$.
    \end{enumerate}
\end{definition}

Every normal Seshadri stratification defines a standard monomial theory on $R$ in the sense of the next proposition. When the stratification is balanced as well, then the normality and its associated standard monomial theory do not depend on the choice if the total order $\geq^t$.

An element $\underline a \in \Gamma$ is called \textit{decomposable}, if it is $0$ or it can be written in the form $\underline a = \underline a^1 + \underline a^2$ for two elements $\underline a^1, \underline a^2 \in \Gamma \setminus \set{0}$ with $\min \supp \underline a^1 \geq \max \supp \underline a^2$. Otherwise $\underline a$ is  called \textit{indecomposable}. Note that the minima and maxima exist, since the support of each element in $\Gamma$ is totally ordered. Let $\mathbb G$ be the set of all indecomposable elements in $\Gamma$. For each $\underline a \in \mathbb G$ we fix a regular function $x_{\underline a} \in R \setminus \set{0}$ with $\mathcal V(x_{\underline a}) = \underline a$ and let $\mathbb G_R = \set{x_{\underline a} \mid \underline a \in \mathbb G}$ be the set of these functions.

We assume that the stratification is normal. In this case every element $\underline a \in \Gamma$ has a unique decomposition into a sum $\underline a = \underline a^1 + \dots + \underline a^s$ of indecomposable elements $\underline a^k \in \Gamma$, such that $\min \supp \underline a^k \geq \max \supp \underline a^{k+1}$ holds for all $k = 1, \dots, s-1$. With the choice of the set $\mathbb G_R$ one can therefore associate a regular function to every element $\underline a \in \Gamma$ via
\begin{align*}
    x_{\underline a} \coloneqq x_{\underline a^1} \cdots x_{\underline a^s} \in R.
\end{align*}
A monomial in the functions in $\mathbb G_R$ is called \textit{standard}, if it is of the form $x_{\underline a}$ for some element $\underline a \in \Gamma$.

\begin{proposition}[{\cite[Proposition 15.6]{seshstrat}}]
    \label{prop:smt}
    If the stratification is normal and $\mathbb G_R$ and $x_{\underline a}$ are chosen as above, then the following statements hold:
    \begin{enumerate}[label=(\alph{enumi})]
        \item The set $\mathbb G_R$ generates $R$ as a $\K$-algebra.
        \item The set of all standard monomials in $\mathbb G_R$ is a basis of $R$ as a vector space.
        \item If $\underline a = \underline a^1 + \dots + \underline a^s$ is the unique decomposition of $\underline a$ into indecomposables, then $x_{\underline a} \coloneqq x_{\underline a^1} \cdots x_{\underline a^s}$ is a standard monomial with $\mathcal V(x_{\underline a}) = \underline a$.
        \item For each non-standard monomial $x_{\underline a^1} \cdots x_{\underline a^s}$ in $\mathbb G_R$ there exists a straightening relation
        \begin{align*}
            x_{\underline a^1} \cdots x_{\underline a^s} = \sum_{\underline b \in \Gamma} u_{\underline b} \mkern2mu x_{\underline b}
        \end{align*}
        expressing it as a linear combination of standard monomials, where $u_{\underline b} \neq 0$ only if $\underline b \geq^t \underline a^1 + \dots + \underline a^s$.
    \end{enumerate}
\end{proposition}

The above standard monomial theory is compatible with the induced stratification (see Remark~\ref{rem:induced_strat}) on $X_p$ for $p \in A$. This was shown in \cite[Theorem 15.12]{seshstrat}. By adapting this result to the multiprojective setting, we obtain the following corollary.

\begin{corollary}
    \label{cor:smt_induced_strat}
    If the Seshadri stratification on $X$ is normal and balanced, then the following statements are fulfilled for each $p \in A$:
    \begin{enumerate}[label=(\alph{enumi})]
        \item The induced stratification on $X_p$ is also normal and balanced.
        \item The fan of monoids of this stratification is equal to $\Gamma_p = \set{\underline a \in \Gamma \mid \max \supp \underline a \leq p}$.
        \item The set of indecomposable elements in $\Gamma_p$ is given by $\mathbb G_p = \mathbb G \cap \Gamma_p$ and the restriction of a function $x_{\underline a}$ with $\underline a \in \mathbb G_p$ fulfills $\mathcal V_p(x_{\underline a}\big\vert_{\hat X_p}) = \underline a$, where $\mathcal V_p$ denotes the quasi-valuation of the induced stratification.
        \item A standard monomial $x_{\underline a}$, $a \in \Gamma$, vanishes identically on $\hat X_p$, if and only if $\max \supp \underline a \leq p$. In this case, the monomial is called \textbf{\boldmath{}standard on $X_p$}.
        \item The restrictions of the standard monomials $x_{\underline a}$, $a \in \Gamma$, which are standard on $X_p$, form a basis of the multihomogeneous coordinate ring $\K[X_p]$ \wrt{} the embedding
        \begin{align*}
            X_p \longhookrightarrow \prod_{i \in I_p} \PP(V_i).
        \end{align*}
    \end{enumerate}
\end{corollary}

\begin{example}
    \label{ex:hodge_type}
    The monoid $\Gamma_{\mathfrak C}$ to a maximal chain $\mathfrak C$ always contains the set $\N_0^{\mathfrak C}$, as every extremal function $f_p$ for $p \in \mathfrak C$ has the quasi-valuation $e_p$. The stratification is called of \textbf{Hodge type}, if all its bonds $b_{p,q}$ are equal to $1$. In this case every monoid $\Gamma_{\mathfrak C}$ coincides with $\N_0^m$, since $\Gamma_{\mathfrak C}$ is contained in the lattice $L^{\mathfrak C} = \Z^{\mathfrak C}$ from equation (\ref{eq:def_L_C}). For instance, the stratification we defined in Example~\ref{ex:strat_y_1} is of Hodge type. Seshadri stratifications of Hodge type are always normal and balanced. More of their properties can be found in~\cite[Section 16]{seshstrat}.
\end{example}

\begin{example}
    \label{ex:A}
    We return to the Seshadri stratification from Example~\ref{ex:strat_y_0_y_1}. It has the following bonds:
    
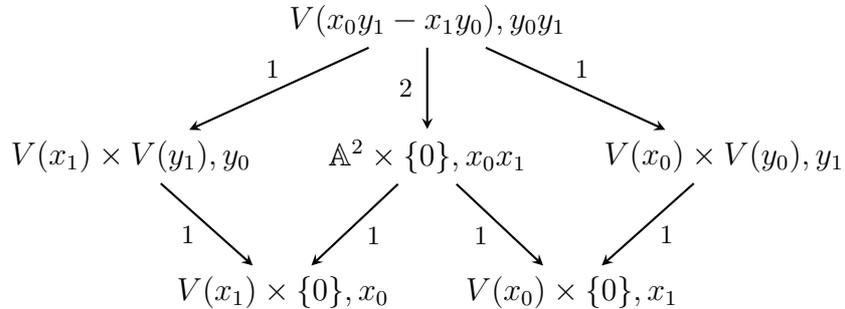
\begin{center}
        \begin{tikzpicture}
        \node (X) at (0,0) {$V(x_0 y_1 - x_1 y_0), y_0 y_1$};
        \node (1010) at (-3.9,-1.8) {$V(x_1) \times V(y_1), y_0$};
        \node (1100) at (0,-1.8) {$\A^2 \times \set{0}, x_0 x_1$};
        \node (0101) at (3.9,-1.8) {$V(x_0) \times V(y_0), y_1$};
        \node (1000) at (-1.9,-3.6) {$V(x_1) \times \set{0}, x_0$};
        \node (0100) at (1.9,-3.6) {$V(x_0) \times \set{0}, x_1$};
        
        \draw [thick, -stealth] (X) -- node[left, above, xshift=-0.2em, yshift=0em]{\footnotesize $1$}(1010);
        \draw [thick, -stealth] (X) -- node[left, xshift=-0.1em, yshift=0em]{\footnotesize $2$}(1100);
        \draw [thick, -stealth] (X) -- node[right, above, xshift=0.2em, yshift=0em]{\footnotesize $1$}(0101);
        \draw [thick, -stealth] (1010) -- node[left, below, xshift=-0.6em, yshift=0.3em]{\footnotesize $1$}(1000);
        \draw [thick, -stealth] (1100) -- node[right, below, xshift=0.6em, yshift=0.3em]{\footnotesize $1$}(1000);
        \draw [thick, -stealth] (1100) -- node[left, below, xshift=-0.6em, yshift=0.3em]{\footnotesize $1$}(0100);
        \draw [thick, -stealth] (0101) -- node[right, below, xshift=0.6em, yshift=0.3em]{\footnotesize $1$}(0100);
        \end{tikzpicture}
    \end{center}
    There are four maximal chains in $A$, which we denote by $\mathfrak C_1, \mathfrak C_2, \mathfrak C_3, \mathfrak C_4$ from left to right. In two of these chains all bonds are equal to $1$. By Example~\ref{ex:hodge_type} we get the associated monoids $\Gamma_{\mathfrak C_1} = \N_0^{\mathfrak C_1}$ and $\Gamma_{\mathfrak C_4} = \N_0^{\mathfrak C_4}$. The monoid $\Gamma_{\mathfrak C_2}$ is contained in the intersection of the lattice
    \begin{align*}
        L^{\mathfrak C_2} = \set{ a_X e_{X} + a_{01} e_{01} + a_0 e_{0} \mid a_X, a_X + a_{01} , a_X + a_{01} + a_0 \in \tfrac12 \Z } = (\tfrac12 \Z)^{\mathfrak C_2}
    \end{align*}
    with the positive orthant $\Q_{\geq 0}^{\mathfrak C_2}$. 
    
    Using Proposition~\ref{prop:valuation_lattice} one can even find a smaller lattice containing $\Gamma_{\mathfrak C_2}$. We choose the regular functions $F_2 = x_0 y_1$, $F_1 = f_{01}$ and $F_0 = f_0$. Two of those are extremal functions and we already computed the quasi-valuation of $F_2$ in Example~\ref{ex:quasi_val_x_0_y_1}. Then the matrix $B_{\mathfrak C}$ in Proposition~\ref{prop:valuation_lattice} is given by
    \begin{align*}
        B_{\mathfrak C} = \begin{pmatrix}
            2 & 0 & 0 \\
            -1 & 1 & 0 \\
            0 & 0 & 1
        \end{pmatrix}.
    \end{align*}
    It thus follows that $\Gamma_{\mathfrak C_2}$ is contained in the lattice
    \begin{align*}
        L^{\mathfrak C_2}_{\mathcal V} = \set{ a_X e_{X} + a_{01} e_{01} + a_0 e_{0} \mid 2a_X, a_{01} - a_X , a_0 \in \Z }.
    \end{align*}
    On the other hand, every element in $\mathcal L^{\mathfrak C_2} \cap \Q_{\geq 0}^{\mathfrak C_2}$ actually lies in $\Gamma_{\mathfrak C_2}$, since it can be written as a sum of the elements $e_X, e_{01}, e_0, \tfrac12 e_X + \tfrac12 e_{01} \in \Gamma_{\mathfrak C_2}$. Analogously, one can determine the monoid $\Gamma_{\mathfrak C_3}$. Summarizing our computations, we have:
    \begin{align*}
        \Gamma_{\mathfrak C_1} &= \set{ a e_X + b e_{0 \overline 0} + c e_0 \mid a, b, c \in \N_0}, \\
        \Gamma_{\mathfrak C_2} &= \set{ a e_X + b e_{0 1} + c e_0 \mid a, b \in \tfrac{1}{2}\N_0, c \in \N_0, a + b \in \N_0}, \\
        \Gamma_{\mathfrak C_3} &= \set{ a e_X + b e_{0 1} + c e_1 \mid a, b \in \tfrac{1}{2}\N_0, c \in \N_0, a + b \in \N_0}, \\
        \Gamma_{\mathfrak C_4} &= \set{ a e_X + b e_{1 \overline 1} + c e_1 \mid a, b, c \in \N_0}.
    \end{align*}
    As all these monoids are saturated, the stratification is normal. It also is balanced: Every element in $\Gamma$ is a sum of the elements $e_p$ for $p \in A$ and $\mathcal V(F_2) = \tfrac12 e_X + \tfrac12 e_{01}$ and the quasi-valuation of $F_2$ is independent of the choice of the total order $\geq^t$ (see Example~\ref{ex:quasi_val_x_0_y_1}). 
\end{example}

\section{Multidegrees and multigradings}

The $\N_0^m$-grading on the multihomogeneous coordinate ring $R = \K[X]$ corresponds to an action of the torus $T = (\K^\times)^m$ on the multicone $\hat X \subseteq V$ by scaling in each factor $V_i$. This also induces an $T$-action on the multihomogeneous coordinate ring itself: If $g$ is a function in $R$ and $\underline t \in T$, then $g^{\underline t} \coloneqq \underline t \cdot g$ is defined by $(g^{\underline t})(x) = g(\underline t^{-1} \cdot x)$ for all $x \in \hat X$.

\begin{lemma}
    \label{lem:multihom_valuation}
    $ $
    \begin{enumerate}[label=(\alph{enumi})]
      \item For all $g \in R \setminus \set{0}$ and $\underline t \in (\K^\times)^m$ it holds $\mathcal V_{\mathfrak C}(\underline t \cdot g) = \mathcal V_{\mathfrak C}(g)$.
      \item If $h = \sum_{\underline d \in \N_0^m} h_{\underline d} \in R$ is the decomposition of $h \neq 0$ into its multi\-homo\-geneous components $h_{\underline d} \in R_{\underline d}$, then
        \begin{align*}
            \mathcal V_{\mathfrak C}(h) = \operatorname{min}{\set{\mathcal V_{\mathfrak C}(h_{\underline d}) \mid \underline d \in \N_0^m \ \text{such that} \ h_{\underline d} \neq 0}}.
        \end{align*}
\end{enumerate}
\end{lemma}
\begin{proof}
    The statements can be proved analogously to \cite[Lemma 6.15]{seshstrat}. One has to replace the $\K^\times$-action by the $(\K^\times)^m$-action and use Lemma~\ref{lem:span_homog_components}.
\end{proof}

\begin{definition}
    \label{def:degree_map}
    We define the \textbf{degree map} to be the $\Q$-linear map
    \begin{align*}
        \mathrm{deg}: \Q^A \to \Q^m, \quad e_p \mapsto \deg f_p
    \end{align*}
    and call $\deg \underline a$ the \textbf{degree} of an element $\underline a \in \Q^A$.
\end{definition}

\begin{lemma}
    \label{lem:quasi_val_compatible_with_multidegree}
    If $g \in R \setminus \set{0}$ is multihomogeneous, then $\deg g = \deg \mathcal V(g)$.
\end{lemma}
\begin{proof}
    Let $\mathfrak C: p_r > \dots > p_0$ be a maximal chain in $A$ with $\mathcal V_{\mathfrak C}(g) = \mathcal V(g)$. We write $\mathcal V(g)$ in the form $a_r e_{p_r} + \dots + a_0 e_{p_0}$ with coefficients $a_i \in \Q$ and fix a positive integer $N$, such that $N \mathcal V(g) \in \Z^{\mathfrak C}$. Then we have
    \begin{align*}
        \mathcal V(g^N) = N \mathcal V(g) = \mathcal V(\prod_{i=0}^r f_p^{N a_i}).
    \end{align*}
    Suppose that $g^N$ and $f \coloneqq \prod_{i=0}^r f_{p_i}^{N a_i}$ have different multidegrees. Since the leaves of the quasi-valuation are one-dimensional, the quasi-valuation of the non-zero function $h = g^N - f$ is strictly larger than $N \mathcal V(g)$. On the other hand, $h$ consists of the two multihomogeneous components $g^N$ and $f$, so Lemma~\ref{lem:multihom_valuation} implies $\mathcal V(h) = N \mathcal V(g)$, which contradicts our assumption. Thus $g^N$ and $f$ have the same multidegree and it therefore follows:
    \begin{align*}
        \deg g = \frac{1}{N} \deg \prod_{i=0}^r f_{p_i}^{N a_i} = \sum_{i=1}^r a_i \deg f_{p_i} = \deg \mathcal V(g). \tag*{\qedhere}
    \end{align*}
\end{proof}

For every chain $C$ in $A$ the degree map $\deg: \Gamma \to \N_0^m$ defines an $\N_0^m$-grading on the monoid $\Gamma_C$ via the subsets $\Gamma_{C, \underline d}$ of all elements of degree $\underline d$. We write $\Gamma_{\underline d}$ for the elements in $\Gamma$ of degree $\underline d$. This induces an $\N_0^m$-grading on $\mathrm{gr}_{\mathcal V} R$ by the subgroups
\begin{align*}
    (\mathrm{gr}_{\mathcal V} R)_{\underline d} = \bigoplus_{\underline a \in \Gamma_{\underline d}} R_{\geq \underline a} / R_{> \underline a}.
\end{align*}
For each chain $C$ in $A$, $\mathrm{gr}_{\mathcal V, C} R$ is a graded subalgebra of $\mathrm{gr}_{\mathcal V} R$. The fan algebra $\K[\Gamma]$ also carries a grading by $\N_0^m$ induced by the degree map and there exists an isomorphism $\mathrm{gr}_{\mathcal V} R \cong \K[\Gamma]$ of $\N_0^m$-graded algebras. This follows from the construction of the isomorphism on basis elements (see Theorem~\ref{thm:fan_algebra_cong_gr_V_R}).

Let $x_{\mathfrak C} = \prod_{p \in \mathfrak C} f_p \in R$ be the product of all extremal functions along a maximal chain $\mathfrak C$ in $A$ and $I_{\mathfrak C} \subseteq \mathrm{gr}_{\mathcal V} R$ be the annihilator of the element $\overline{x_{\mathfrak C}} \in \mathrm{gr}_{\mathcal V} R$. It was shown in \cite[Corollary 10.8]{seshstrat} that there exists an isomorphism of algebras
\begin{align}
    \label{eq:I_C_is_homogeneous}
    \mathrm{gr}_{\mathcal V} R / I_{\mathfrak C} \cong \mathrm{gr}_{\mathcal V, \mathfrak C} R
\end{align}
and the intersection of all ideals $I_{\mathfrak C}$ is the minimal prime decomposition of the zero ideal in $\mathrm{gr}_{\mathcal V} R$. As the associated graded algebra $\mathrm{gr}_{\mathcal V} R$ is finitely generated and reduced, its corresponding variety $\Spec \mathrm{gr}_{\mathcal V} R$ therefore is the scheme-theoretical union of the toric varieties $\Spec \mathrm{gr}_{\mathcal V, \mathfrak C} R$, each of which is irreducible and of dimension $\dim \hat X$. Using the language of Multiproj schemes, which can be found in Appendix \ref{sec:multiproj}, we can conclude an analogous statement about the scheme $\Multiproj(\mathrm{gr}_{\mathcal V} R)$. Its irreducible components, however, are only schemes, not necessarily projective varieties. We have already looked at an example where this happens: The stratification from Example~\ref{ex:strat_y_1} is of Hodge type and the algebra $\mathrm{gr}_{\mathcal V, \mathfrak C} R \cong \K[\Gamma_{\mathfrak C}]$ associated to the maximal chain $\mathfrak C: X > 0 \overline 0 > 0$ is isomorphic to the algebra from Example~\ref{ex:non_separated_scheme} as an $\N_0^2$-graded algebra. Hence it induces a non-separated scheme.

\begin{corollary}
    \label{cor:irred_components_multiproj_gr_V}
    The scheme $\Multiproj(\mathrm{gr}_{\mathcal V} R)$ is the scheme-theoretical union of the closed, integral subschemes $\Multiproj(\mathrm{gr}_{\mathcal V, \mathfrak C} R)$, where $\mathfrak C$ runs over all maximal chains in $A$. Each of these subschemes is integral and of dimension $\dim X$.
\end{corollary}
\begin{proof}
    All the statements follow directly from Corollary 10.8 in \cite{seshstrat} in combination with the Lemmas~\ref{lem:closed_subschemes_multiproj} and \ref{lem:multiproj_dimension}. To use these lemmas we require the ideal $I_{\mathfrak C} \subseteq \mathrm{gr}_{\mathcal V, \mathfrak C} R$ to be homogeneous and prime, which holds by the isomorphism~(\ref{eq:I_C_is_homogeneous}). We also need to show that the degrees of the homogeneous elements in $\mathrm{gr}_{\mathcal V, \mathfrak C} R$ generate a sublattice of $\Z^m$ of full rank, \ie the image of the degree map $\Gamma_{\mathfrak C} \to \Z^m$ generates a group of rank $m$. It contains the degrees of all extremal functions $f_p$ for $p \in \mathfrak C$. A suitable subset of size $m$ of these degrees is linearly independent, as they can be arranged in an upper triangular matrix with non-zero diagonal (up to permutation of the rows). This can be seen via Lemma~\ref{lem:projective_covering_relation}\,\ref{itm:projective_covering_relation_b}.
\end{proof}

\section{Semi-toric degeneration}

Every Seshadri stratification on an embedded projective variety $Y \subseteq \PP(V)$ induces a degeneration of $Y$ into an union of projective toric varieties. We generalize this result using an analogous approach to the construction in~\cite[Chapter 12]{seshstrat} via Rees algebras.

Let $\mathcal J$ be the image of the map $\Gamma \to \N_0^m \times \Gamma$, $\underline a \mapsto (\deg \underline a, \underline a)$ and let $\succeq$ be the lexicographic order on $\N_0^m \times \Gamma$. For each $(\underline d, \underline a) \in \N_0^m \times \Gamma$ we define the following multihomogeneous ideals in $R$:
\begin{align*}
    \mathcal I_{\succeq (\underline d, \underline a)} = \left\langle \, g \in R \ \middle\vert \ \text{$g$ multihomogeneous and $(\deg g, \mathcal V(g)) \succeq (\underline d, \underline a)$} \, \right\rangle, \\
    \mathcal I_{\succ (\underline d, \underline a)} = \left\langle \, g \in R \ \middle\vert \ \text{$g$ multihomogeneous and $(\deg g, \mathcal V(g)) \succ (\underline d, \underline a)$} \, \right\rangle.
\end{align*}
Their quotient is given by
\begin{align*}
    \mathcal I_{\succeq (\underline d, \underline a)}/\mathcal I_{\succ (\underline d, \underline a)} = \begin{cases}
        \set{0}, & \text{if $(\underline d, \underline a) \notin \mathcal J$}, \\
        R_{\geq \underline a}/R_{> \underline a}, & \text{if $(\underline d, \underline a) \in \mathcal J$}.
    \end{cases}
\end{align*}
By fixing an isomorphism of posets $\pi: (\N_0, \geq) \to (\mathcal J, \succeq)$ we get a descending filtration
\begin{align*}
    R = \mathcal I_0 \supseteq \mathcal I_1 \supseteq \mathcal I_2 \supseteq \dots
\end{align*}
where we write $\mathcal I_j = \mathcal I_{\pi(j)}$ for $j \in \N_0$. Since each ideal $\mathcal I_j$ is multihomogeneous, the Rees algebra
\begin{align*}
    \mathcal A = \ldots \oplus R t^2 \oplus R t \oplus R \oplus \mathcal I_1 t^{-1} \oplus \mathcal I_2 t^{-2} \oplus \ldots
\end{align*}
to this filtration is an $\N_0^m$-graded subalgebra of $R[t, t^{-1}] = \bigoplus_{\underline d \in \N_0^m} R_{\underline d}[t, t^{-1}]$.

As a submodule of $R[t,t^{-1}]$, the algebra $\mathcal A$ is a torsion free module over a Dedekind domain, hence flat over $\K[t]$. Additionally the inclusion $\K[t] \hookrightarrow \mathcal A$ maps to degree $0$. These two properties imply that the induced morphism $\phi: \Spec \mathcal A \to \A^1$ is flat and $(\K^\times)^m$-equivariant (with the trivial action on $\A^1$). In particular, it induces a morphism $\psi: \Multiproj \mathcal A \to \A^1$. By Corollary 2.2.11\,(iv) in \cite{grothendieck1965elements}, $\psi$ inherits the flatness of $\phi$, because the morphism $\Spec \mathcal A \setminus V(\mathcal A_+) \to \Multiproj \mathcal A$ is surjective (set-theoretically) as a geometric quotient.

The general fiber of $\phi$ at $t \neq 0$ is isomorphic to the multicone $\hat X$, because $\mathcal A/(t-b) \cong R$ for all $b \in \A^1 \setminus \set{0}$. On the other hand we have
\begin{align*}
    \mathcal A/(t) \cong \bigoplus_{j \in \N_0} \mathcal I_j / \mathcal I_{j+1} \cong \mathrm{gr}_{\mathcal V} R,
\end{align*}
so the special fiber is isomorphic to $\Spec\mkern1mu(\mathrm{gr}_{\mathcal V} R)$.

\begin{corollary}
    The general fiber of the flat morphism $\psi: \Multiproj \mathcal A \to \A^1$ is isomorphic to $\Multiproj R \cong X$ and its special fiber at $t = 0$ is isomorphic to $\Multiproj\mkern1mu(\mathrm{gr}_{\mathcal V} R)$.
\end{corollary}

\section{Newton-Okounkov theory: Veronese submonoids}
\label{subsec:NO-complex}

In \cite{seshstrat} a Newton-Okounkov theoretical object was associated to a given Seshadri stratification. For each maximal chain $\mathfrak C$ one obtains a simplex, such that its lattice points describe the rate of growth for the dimensions of the graded components of $\mathrm{gr}_{\mathcal V, \mathfrak C} R$. These simplices fit together to form a simplicial complex $\Delta_{\mathcal V}$. The dimension of $X$ is equal to the dimension of the simplicial complex and the degree of $X \subseteq \PP(V)$ can be extracted via the volume of $\Delta_{\mathcal V}$ with respect to certain lattices.

The simplices generalize to polytopes in the multiprojective setting. However, we obtain not just one polytopal complex, but a polytopal complex $\Delta_{\mathcal V}^{(\underline d)}$ for each multidegree $\underline d \in \N_0^m$. This structure is not visible for $m = 1$ since the polytopal complexes are scaled versions of each other in this case. For most values of $\underline d$ the complex $\Delta_{\mathcal V}^{(\underline d)}$ has the same dimension as the variety $X$, but in some edge cases it can collapse to a smaller dimension.

In this first section on Newton-Okounkov theory, we define the poset $\Delta^{(\underline d)}(A)$ indexing the polytopes in the polytopal complex $\Delta_{\mathcal V}^{(\underline d)}$ and relate this poset to the geometry of a certain degenerate variety $\Proj(\mathrm{gr}_{\mathcal V}^{(\underline d)} R)$. 

For multiprojective varieties there also exists a Hilbert polynomial $H_R \in \Q[x_1, \dots, x_m]$. We refer to the Appendix~\ref{subsec:multiproj_varieties} for its properties and the definition of the multidegrees of $X$. Again, we set $r \coloneqq \dim X$. Let
\begin{align*}
    G_R = \sum_{\underline k} \frac{\deg_{\underline k}(X)}{k_1! \cdots k_m!} \, x_1^{k_1} \cdots x_m^{k_m}
\end{align*}
be the homogeneous component of highest total degree in $H_R$. We have $\deg G_R = r$. Since the leaves of the quasi-valuation $\mathcal V$ are at most one-dimensional, there exists a basis $\mathbb B$ of $R$ as a vector space over $\K$, such that $\mathbb B \to \Gamma$, $g \mapsto \mathcal V(g)$ is a bijection. In particular:
\begin{align}
    \label{eq:dimension_equality_R_gr}
    \dim R_{\underline d} = \dim \mkern2mu (\mathrm{gr}_{\mathcal V} R)_{\underline d} = \vert \Gamma_{\underline d} \vert.
\end{align}
The equality
\begin{align*}
    G_R(\underline d) = \lim_{n \to \infty} \frac{\dim R_{n \underline d}}{n^r}
\end{align*}
for all $\underline d \in \N_0^m$ suggests that we should examine the Veronese subalgebras 
\begin{align*}
    \mathrm{gr}_{\mathcal V}^{(\underline d)} R = \bigoplus_{n \in \N_0} \mkern2mu (\mathrm{gr}_{\mathcal V} R)_{n \underline d}
\end{align*}
of the associated graded algebra $\mathrm{gr}_{\mathcal V} R$. First we need to fix some notation. For each chain $C$ in $A$, the algebra $\mathrm{gr}_{\mathcal V}^{(\underline d)} R$ contains the Veronese subalgebra
\begin{align*}
    \mathrm{gr}_{\mathcal V, C}^{(\underline d)} R = \mathrm{gr}_{\mathcal V}^{(\underline d)} R \cap \mathrm{gr}_{\mathcal V, C} R
\end{align*}
of $\mathrm{gr}_{\mathcal V, C} R$. The fan of monoids contains the Veronese subfan $\Gamma^{(\underline d)} = \bigcup_{n \in \N_0} \Gamma_{n \underline d}$ and for each monoid $\Gamma_{C}$ we have the Veronese submonoid $\Gamma_{C}^{(\underline d)} = \Gamma_{C} \cap \Gamma^{(\underline d)}$. By Theorem~\ref{thm:fan_algebra_cong_gr_V_R} there are again isomorphisms of $\N_0$-graded algebras:
\begin{align*}
    \mathrm{gr}_{\mathcal V}^{(\underline d)} R \cong \K[\Gamma^{(\underline d)}] \quad \text{and} \quad \mathrm{gr}_{\mathcal V, C}^{(\underline d)} R \cong \K[\Gamma_{C}^{(\underline d)}].
\end{align*}

In general, the fan of monoids $\Gamma$ and the Veronese-fan of monoids $\Gamma^{(\underline d)}$ have different combinatorial structures. Whereas the poset of all monoids $\Gamma_C$, ordered by inclusion, is isomorphic to the poset $\Delta(A)$ of all chains in $A$, different chains can have the same Veronese monoid $\Gamma_C^{(\underline d)}$. For this reason we now define a map
\begin{align*}
    \Delta(A) \to \Delta(A), \quad C \mapsto C_{\underline d}, 
\end{align*}
so that $\Gamma_{C}^{(\underline d)} = \Gamma_{D}^{(\underline d)}$ is equivalent to $C_{\underline d} = D_{\underline d}$ for all chains $C, D \subseteq A$. The chain $C_{\underline d}$ depends on the cone
\begin{align}
    \label{eq:sigma_C}
    \sigma_{C} = \operatorname{Cone} \mkern2mu \set{\deg \underline a \mid \underline a \in \Gamma_{C}} \subseteq \R^m.
\end{align}
Regarding (rational) polyhedral cones and (lattice) polytopes, we use the language of Cox, Little and Schenck from \cite{cox2011toric}. On polyhedral cones and polytopes, we always use the standard euclidean topology. The cone in $\R^A$ spanned by $\Gamma_C$ is generated by the vectors $e_p \in \R^C$ for $p \in C$, as $\Gamma_C \subseteq \Q^C_{\geq 0}$. Therefore $\sigma_C$ is a rational polyhedral cone with respect to the lattice $\Z^m \subseteq \R^m$.

If $\underline d \notin \sigma_C$ then we set $C_{\underline d} = \varnothing$. Now assume $\underline d \in \sigma_C$. Since every polyhedral cone is the disjoint union of the relative interiors (\ie the interior in its closure) of its faces, there exists a unique face $\tau$ of $\sigma_C$ with $\underline d \in \operatorname{relint} \tau$. This is also the unique minimal face containing $\underline d$. Now each convex cone, which is generated by a finite set $S$ is generated by its edges, \ie its one dimensional faces, and every generating set of the cone contains at least one non-zero element from each edge. As $\sigma_C$ is generated by the set of all $\deg f_p$ with $p \in C$, every edge of $\sigma_C$ is of the form $\R_{\geq 0} \mkern2mu e_p$ for $p \in C$. We then define
\begin{align*}
    C_{\underline d} = \set{p \in C \mid \text{$\R_{\geq 0} \mkern2mu e_p$ is an edge of $\tau$}}.
\end{align*}
Since $\tau = \sigma_{C_{\underline d}}$, it is immediate that the image $\Delta^{(\underline d)}(A)$ of the map $\Delta(A) \to \Delta(A)$, $C \mapsto C_{\underline d}$ is equal to
\begin{align*}
    \Delta^{(\underline d)}(A) = \set{C \in \Delta(A) \mid \underline d \in \operatorname{relint} \sigma_C} \cup \set{\varnothing}.
\end{align*}

\begin{lemma}
    $ $
    \label{lem:veronese_order_complex}
    \begin{enumerate}[label=(\alph{enumi})]
        \item \label{itm:veronese_order_complex_a} For any two chains $C, D \in \Delta(A)$ it holds $\Gamma_{C}^{(\underline d)} \subseteq \Gamma_{D}^{(\underline d)}$, if and only if $\mkern1mu C_{\underline d} \subseteq D_{\underline d}$.
        \item \label{itm:veronese_order_complex_b} The map $\Delta(A) \to \Delta^{(\underline d)}(A)$, $C \mapsto C_{\underline d}$ is monotone.
        \item \label{itm:veronese_order_complex_c} The following map is an isomorphism of posets:
        \begin{align*}
            \Delta^{(\underline d)}(A) \to \set{\Gamma_{C}^{(\underline d)} \mid C \in \Delta(A)}, \quad C \longmapsto \Gamma_{C}^{(\underline d)}.
        \end{align*}
    \end{enumerate}
\end{lemma}
\begin{proof}
    For each $C \in \Delta(A)$ the monoids $\Gamma_{C}^{(\underline d)}$ and $\Gamma_{C_{\underline d}}^{(\underline d)}$ coincide. Indeed, if $\underline a$ is an element of $\Gamma_{C}^{(\underline d)}$, then $\deg \underline a \in \N \underline d$ lies in the face $\tau \subseteq \sigma_C$ defined by $C_{\underline d}$. For every $p \in C \setminus C_{\underline d}$ with $p \in \supp \underline a$ we can write the degree of $\underline a$ in the form
    \begin{align*}
        \deg \underline a = c \deg f_p + \sum_{q \in C} c_q \deg f_q
    \end{align*}
    with real numbers $c_q \geq 0$ and $c > 0$. All elements $\deg f_q$ for $q \in C$ lie in $\sigma_C$ but $\deg f_p$ is not contained in the face $\tau$. This is impossible, as $\deg \underline a \in \tau$.

    Let $C, D$ be two chains in $A$. If $C_{\underline d} \subseteq D_{\underline d}$, then we clearly have $\Gamma_{C}^{(\underline d)} \subseteq \Gamma_{D}^{(\underline d)}$. Now suppose that $\Gamma_{C}^{(\underline d)} \subseteq \Gamma_{D}^{(\underline d)}$ and fix an element $p \in C_{\underline d}$. By the definition of $C_{\underline d}$, the multidegree $\deg f_p$ lies in the relative interior of the face $\tau$ corresponding to $C_{\underline d}$. This allows us to use the following argument, which appears multiple times throughout the sections on Newton-Okounkov theory: By the properties of the relative interior, there exists an element in the intersection of the translated cone $\deg f_p + \tau \subseteq \R^m$ with the set $\N \underline d$. This holds for every convex polyhedral cone. As the cone $\sigma_C$ is generated by lattice points in $\Z^m \subseteq \R^m$, we can therefore find non-negative, rational numbers $a_q \in \Q$ and $N \in \N$, such that
    \begin{align*}
        N \underline d = \deg f_p + \sum_{q \in C_{\underline d}} a_q \deg f_q.
    \end{align*}
    By multiplying with a common denominator of all $a_q$, we can assume that these rational numbers are non-negative integers. Hence we see that
    \begin{align*}
        \mathcal V(f_p \cdot \prod_{q \in C_{\underline d}} f_q^{a_p}) = e_p + \sum_{q \in C_{\underline d}} a_q e_q \in \Gamma_{C}^{(\underline d)} = \Gamma_{D_{\underline d}}^{(\underline d)},
    \end{align*}
    which implies $C_{\underline d} \subseteq D_{\underline d}$. Finally, the parts~\ref{itm:veronese_order_complex_b} and \ref{itm:veronese_order_complex_c} follow from the first statement.
\end{proof}

\begin{lemma}
    \label{lem:veronese_submonoid_fin_gen}
    The monoid $\Gamma_{C}^{(\underline d)}$ is finitely generated for each chain $C \subseteq A$ and $\underline d \in \N_0^m$.
\end{lemma}
\begin{proof}
    Choose finitely many generators $\underline a^{(1)}, \dots, \underline a^{(s)} \in \Gamma_C$ and consider the map
    \begin{align*}
        \phi: \Z^s \to \Z^m/\Z \underline d, \quad (n_1, \dots, n_s) \longmapsto \sum_{i=1}^s n_i \deg \underline a^{(i)}.
    \end{align*}
    It is sufficient to show, that the monoid $M = \N_0^s \cap \ker \phi \subseteq \Z^s$ is finitely generated, since its image under $\N_0^s \to \mathcal L^{C}$ coincides with $\Gamma_{C}^{(\underline d)}$, where $\mathcal L^{C} \subseteq \Q^{A}$ denotes the lattice generated by $\Gamma_{C}$. The set $\R_{\geq 0}^s \cap \operatorname{span}_{\R}(\ker \phi)$ is a rational polyhedral cone \wrt{} the lattice $\ker \phi$. The intersection of this cone with the kernel of $\phi$ is exactly $M$. By Gordan's Lemma, $M$ is finitely generated.
\end{proof}

Analogous to Corollary~\ref{cor:irred_components_multiproj_gr_V}, the irreducible components of the projective variety $\Proj(\mathrm{gr}_{\mathcal V}^{(\underline d)} R)$ are determined by the maximal elements in the poset $\Delta^{(\underline d)}(A)$. Clearly, every maximal element in $\Delta^{(\underline d)}(A)$ is of the form $\mathfrak C_{\underline d}$ for a maximal chain $\mathfrak C \in \Delta(A)$, but the converse is false. There can also exist two different maximal chains $\mathfrak C$ and $\mathfrak D$ in $A$ with $\mathfrak C_{\underline d} = \mathfrak D_{\underline d}$. Fortunately, in most cases, the maximal elements in $\Delta^{(\underline d)}(A)$ are easy to describe: When $\underline d$ does not lie on the boundary of the cone $\sigma_{\mathfrak C}$, then $\mathfrak C_{\underline d}$ is maximal in $\Delta^{(\underline d)}(A)$, if and only if $\underline d \in \sigma_{\mathfrak C}$.

\begin{lemma}
    \label{lem:irreducible_components_gr_V_d}
    The projective variety $\Proj(\mathrm{gr}_{\mathcal V}^{(\underline d)} R)$ is scheme-theoretically the irredundant union of the toric subvarieties $\Proj(\mathrm{gr}_{\mathcal V, C}^{(\underline d)} R)$, where $C$ runs over all maximal elements in the poset $\Delta^{(\underline d)}(A)$.
\end{lemma}
\begin{proof}
    The proof of this statement is mostly analogous to the proof of Proposition 10.7 in \cite{seshstrat}. Recall that for each maximal chain $\mathfrak C$ in $A$ we defined the product $x_{\mathfrak C} = \prod_{p \in \mathfrak C} f_p \in R$ of all extremal functions along $\mathfrak C$ and the annihilator $I_{\mathfrak C} \subseteq \mathrm{gr}_{\mathcal V} R$ of the element $\overline{x_{\mathfrak C}} \in \mathrm{gr}_{\mathcal V} R$. In \loccit{} it was shown that the ideal $I_{\mathfrak C}$ is given by
    \begin{align*}
        I_{\mathfrak C} = \bigoplus_{\underline a \in \Gamma \setminus \Gamma_{\mathfrak C}} R_{\geq \underline a} / R_{> \underline a}.
    \end{align*}
    The intersection $I_{\mathfrak C}^{(\underline d)} = I_{\mathfrak C} \cap \mathrm{gr}_{\mathcal V}^{(\underline d)} R$ is a prime ideal in $\mathrm{gr}_{\mathcal V}^{(\underline d)} R$ and it can be written as
    \begin{align}
        \label{eq:I_C_d}
        I_{\mathfrak C}^{(\underline d)} = \bigoplus_{\underline a \in \Gamma^{(\underline d)} \setminus \Gamma_{\mathfrak C}^{(\underline d)}} R_{\geq \underline a} / R_{> \underline a}.
    \end{align}
    It follows that the intersection of the ideals $I_{\mathfrak C}^{(\underline d)}$ over all maximal chains is equal to the zero ideal. On the other hand, $I_{\mathfrak C}^{(\underline d)}$ does not depend on $\mathfrak C$ but only on the monoid $\Gamma_{\mathfrak C}^{(\underline d)}$. Hence we can choose a subset $\mathcal C$ of all maximal chains in $A$, which maps bijectively to the maximal elements in $\Delta^{(\underline d)}(A)$, such that
    \begin{align*}
        \bigcap_{\mathfrak C \in \mathcal C} I_{\mathfrak C}^{(\underline d)} = (0).
    \end{align*}
    This intersection is irredundant, since the Veronese fan of monoids $\Gamma^{(\underline d)}$ is the irredundant union of the monoids $\Gamma_{C}^{(\underline d)}$ over all maximal elements $C \in \Delta^{(\underline d)}(A)$. By (\ref{eq:I_C_d}), we have $(\mathrm{gr}_{\mathcal V}^{(\underline d)} R) / I_{\mathfrak C}^{(\underline d)} \cong \mathrm{gr}_{\mathcal V, \mathfrak C}^{(\underline d)} R$ for every $\mathfrak C \in \mathcal C$.
    
    Finally, we need to show that $I_{\mathfrak C}^{(\underline d)}$ is a minimal prime ideal in $\mathrm{gr}_{\mathcal V}^{(\underline d)} R$ for all $\mathfrak C \in \mathcal C$. If $I$ was an ideal properly contained in $I_{\mathfrak C}^{(\underline d)}$, then there exists a non-zero function $g \in R$ with $\mathcal V(g) \notin \Gamma_{\mathfrak C}^{(\underline d)}$ and $\overline g \in I_{\mathfrak C}^{(\underline d)} \setminus I$. Then we have $\overline g \cdot \overline x_{\mathfrak C} = 0$ in $\mathrm{gr}_{\mathcal V} R$. We now wish to multiply $\overline x_{\mathfrak C}$ with a suitable element $\overline h \in \mathrm{gr}_{\mathcal V, \mathfrak C} R$ such that their product lies in $\mathrm{gr}_{\mathcal V}^{(\underline d)} R$. Then $I$ cannot be prime since both $\overline g$ and $\overline x_{\mathfrak C} \overline h$ are non-zero in $\mathrm{gr}_{\mathcal V}^{(\underline d)} R / I$, but their product is zero. The multidegree of $x_{\mathfrak C}$ lies in the cone $\sigma_{\mathfrak C} = \sigma_{\mathfrak C_{\underline d}}$ and the tuple $\underline d$ is contained in its relative interior. Therefore we can find natural numbers $n_p$, $p \in \mathfrak C$, such that
    \begin{align*}
        \deg(x_{\mathfrak C} \cdot \prod_{p \in \mathfrak C} f_p^{n_p}) = \deg x_{\mathfrak C} + \sum_{p \in \mathfrak C} n_i \deg f_p \in \N \mkern1mu \underline d,
    \end{align*}
    which gives us the desired function $h = \sum_{p \in \mathfrak C} f_p^{n_p}$.
\end{proof}

\begin{example}
    \label{ex:irreducible_components_gr_V_d}
    Let $\underline d = (0,1)$ in the Seshadri stratification from Example~\ref{ex:A}. The Veronese submonoids are given by
    \begin{align*}
        \Gamma_{\mathfrak C_1}^{(\underline d)} = \N_0 e_X + \N_0 e_{0 \overline 0}, \ \; \Gamma_{\mathfrak C_2}^{(\underline d)} = \Gamma_{\mathfrak C_3}^{(\underline d)} = \N_0 e_X, \ \; \Gamma_{\mathfrak C_4}^{(\underline d)} = \N_0 e_X + \N_0 e_{1 \overline 1}.
    \end{align*}
    Hence the set $\mathcal C = \set{\mathfrak C_1, \mathfrak C_4}$ in the proof of Lemma~\ref{lem:irreducible_components_gr_V_d}, so the projective variety $\Proj(\mathrm{gr}_{\mathcal V}^{(\underline d)} R)$ is the irredundant union of the two irreducible components $\Proj(\mathrm{gr}_{\mathcal V, \mathfrak C_1}^{(\underline d)} R)$ and $\Proj(\mathrm{gr}_{\mathcal V, \mathfrak C_4}^{(\underline d)} R)$. For $\underline d = (d_1, d_2)$ with $d_1, d_2 \geq 1$, however, every maximal chain in $A$ is maximal in $\Delta^{(\underline d)}(A)$ and $\Proj(\mathrm{gr}_{\mathcal V}^{(\underline d)} R)$ consists of four irreducible components.
\end{example}

\section{Newton-Okounkov theory: Polytopal complexes}

In the same way Kaveh and Khovanskii associate a Newton-Okounkov convex body to a pair of a semigroup and an admissable rational half-space (see \cite{kaveh2012newton}), we define the set
\begin{align*}
    \Delta_{C}^{(\underline d)} = \overline{\bigcup_{n \in \N} \tfrac{1}{n} \Gamma_{C, n \underline d} } \subseteq \R^A
\end{align*}
for the monoid $\Gamma_{C}^{(\underline d)}$ and the half-space of all elements in its span of degree $\R_{\geq 0} \mkern2mu \underline d$. When we extend the degree map to $\R^A \to \R^m$, $\Delta_{C}^{(\underline d)}$ can be written in the form
\begin{align}
    \label{eq:NO_body_is_polytope}
    \Delta_{C}^{(\underline d)} = \R^C_{\geq 0} \cap \set{\underline a \in \R^C \mid \deg \underline a = \underline d}.
\end{align}
Remember that $\R^C_{\geq 0}$ is exactly the cone spanned by $\Gamma_C$. To see this equality~(\ref{eq:NO_body_is_polytope}), it suffices to show that every rational conical combination $\underline a$ of elements in $\Gamma_C$ with $\deg \underline a = \underline d$ lies in $\Delta_{C}^{(\underline d)}$. So let $\underline a \in \R^C$ be of degree $\underline d$ and of the form $\underline a = \lambda_1 \underline a^{(1)} + \dots + \lambda_s \underline a^{(s)}$ with $\lambda_i \in \Q_{\geq 0}$ and $\underline a^{(i)} \in \Gamma_C$. Since $\Gamma_C$ is a subset of $\Q^C_{\geq 0}$, there exists a natural number $N$, such that $N \underline a$ is a $\Z$-linear combination of the unit vectors $e_p \in \R^C$ with non-negative coefficients. In particular, $N \underline a$ is contained in $\Gamma_C$ and $\underline a \in \Delta_{C}^{(\underline d)}$. The other inclusion is immediate from the definition of $\Delta_{C}^{(\underline d)}$. 

The cone generated by $\Gamma_C$ is also compatible with the Veronese submonoids: 
\begin{align*}
    \operatorname{Cone} \Gamma_C^{(\underline d)} = \operatorname{Cone} \Gamma_C \cap \set{ x \in \R^C \mid \deg x \in \R \underline d}
\end{align*}
Again, one can show this equality by looking at the rational conical combinations of $\Gamma_C$ and taking the closure. In particular, the set $\Delta_{C}^{(\underline d)}$ can be described by all elements of degree $\underline d$ in $\operatorname{Cone} \Gamma_C^{(\underline d)}$.

We denote the lattice generated by $\Gamma_{C}$ by $\mathcal L^{C} \subseteq \Q^{A}$ and the lattice generated by $\Gamma_{C}^{(\underline d)}$ by $\mathcal L^{C, (\underline d)} \subseteq \Q^{A}$. Equation~(\ref{eq:NO_body_is_polytope}) implies that $\Delta_{C}^{(\underline d)}$ is a polytope and we will see shortly that its vertices are contained in the $\Q$-span of the lattice $\mathcal L^{C, (\underline d)}$. Clearly this lattice is contained in $\mathcal L^{C}$, but in general not every element $\underline a \in \mathcal L^{C}$ with $\deg \underline a \in \Z \underline d$ is contained in $\mathcal L^{C, (\underline d)}$. However, this statement is true if $C$ is an element of $\Delta^{(\underline d)}(A)$.

\begin{lemma}
    \label{lem:veronese_lattice}
    For each chain $C$ in $A$ and $\underline d \in \N_0^m$ the lattice $\mathcal L^{C, (\underline d)}$ is equal to the $\underline d$-th Veronese sublattice of $\mathcal L^{C_{\underline d}}$:
    \begin{align*}
        \mathcal L^{C, (\underline d)} = \set{\underline a \in \mathcal L^{C_{\underline d}} \mid \deg \underline a \in \Z \underline d}.
    \end{align*}
    If $\underline d \in \sigma_C$, the lattice is of rank $\vert C_{\underline d} \vert - \dim \sigma_{C_{\underline d}} + 1$. Additionally, this number is bounded from above by $r + 1 = \dim X + 1$. 
\end{lemma}
\begin{proof}
    The statement is trivial when $\underline d$ is not contained in $\sigma_C$. Now let $\underline d \in \sigma_C$. Since the lattice $\mathcal L^{C, (\underline d)}$ does not change, when we replace $C$ by $C_{\underline d}$, we can assume that $\underline d$ lies in the relative interior of $\sigma_C$.
    
    Clearly $\mathcal L^{C, (\underline d)}$ is contained in $\set{\underline a \in \mathcal L^{C} \mid \deg \underline a \in \Z \underline d}$. To show the other inclusion, let $\underline a$ be an element of $\mathcal L^{C}$ with $\deg \underline a \in \Z \underline d$. We can write $\underline a$ in the form $\underline b - \underline c$ for $\underline b, \underline c \in \Gamma_{C}$. Since $\underline d \in \operatorname{relint} \sigma_{C}$, again we can use an argument similar to the proof of Lemma~\ref{lem:veronese_order_complex}: There exists an element $\underline b' \in \Gamma_{C}$ such that $\deg(\underline b + \underline b') \in \N_0 \underline d$. Then we have $\underline a = (\underline b + \underline b') - (\underline c + \underline b')$ and since $\deg \underline a \in \Z \underline d$, the degree of $\underline c + \underline b'$ is a multiple of $\underline d$ as well. Hence $\underline a \in \mathcal L^{C, (\underline d)}$.

    Therefore the lattice $\mathcal L^{C, (\underline d)}$ coincides with the kernel of the map $\phi: \mathcal L^{C} \to \Z^m/\Z \underline d$, $\underline a \mapsto \deg \underline a + \Z \underline d$. The image of the degree map $\mathcal L^{C} \to \Z^m$ is free and its rank is equal to the dimension of the cone $\sigma_C$. This implies that the rank of $\mathcal L^{C, (\underline d)}$ is given by $\operatorname{rank}(\mathcal L^{C}) - (\dim \sigma_C - 1)$. We get the desired formula, as the lattice $\mathcal L^{C}$ is of rank $\vert C \vert$, because it contains all unit vectors $e_p$ for $p \in C$.

    Finally, the rank of the lattice $\mathcal L^{C, (\underline d)}$ is at most $r + 1$. To see this, let $\mathfrak C$ be a maximal chain containing $C$. Of course, we have
    \begin{align*}
        \mathcal L^{C, (\underline d)} \subseteq \mathcal L^{\mathfrak C, (\underline d)} \subseteq \set{\underline a \in \mathcal L^{\mathfrak C} \mid \deg \underline a \in \Z \underline d}.
    \end{align*}
    By a similar argument, the lattice on the right hand side is of rank $\vert \mathfrak C \vert - \dim \sigma_{\mathfrak C} + 1$. We have seen in the proof of Corollary~\ref{cor:irred_components_multiproj_gr_V}, that the degrees of the extremal functions along $\mathfrak C$ span a group of rank $m$, so the dimension of $\sigma_{\mathfrak C}$ is equal to $m$ and it follows
    \begin{align*}
        \operatorname{rank} \mathcal L^{C, (\underline d)} \leq \vert \mathfrak C \vert - \dim \sigma_{\mathfrak C} + 1 = \dim \hat X - m + 1 = \dim X + 1. \tag*{\qedhere}
    \end{align*}
\end{proof}

\begin{example}
    \label{ex:multisimplices}
    The polytope $\Delta_{C}^{(\underline d)}$ of a chain $C$ is of a particularly nice form, namely a product of simplices, when the support of $\deg f_p$ is a one-element set for each $p \in A$. In this case, we get a partition of the poset $A$ into the subsets
    \begin{align*}
        A_i = \set{p \in A \mid \deg f_p \in \N e_i}.
    \end{align*}
    We fix a chain $C \in \Delta(A)$ and a degree $\underline d = (d_1, \dots, d_m) \in \N_0^m$. For each $i \in [m]$ we have the subchain $C_i = C \cap A_i$ and by equation~(\ref{eq:NO_body_is_polytope}), $\Delta_{C}^{(\underline d)}$ is equal to the direct product of the polytopes
    \begin{align*}
         P_i = \set{x \in \R^{A_i}_{\geq 0} \mid \supp x \subseteq C_i, (\deg x)_i = d_i}
    \end{align*}
    where $(\deg x)_i$ denotes the $i$-th component of $\deg x$. We write $\vert \underline c \vert = c_1 + \dots + c_m$ for the total degree of a tuple $\underline c \in \N_0^m$. We show that each of these polytopes is given by $d_i \Delta_{C_i} \subseteq \R^{A_i}$ with
    \begin{align*}
        \Delta_{C_i} = \begin{cases}
            \mkern2mu \operatorname{Conv} \mkern-2mu \Set{\frac{1}{\vert \deg f_p \vert} e_p \ \middle\vert \ p \in C_i}, & \text{if $C_i \neq \varnothing$}, \\
            \set{0}, & \text{if $C_i = \varnothing$ and $d_i = 0$}, \\
            \mkern9mu \varnothing, & \text{if $C_i = \varnothing$ and $d_i > 0$}.
        \end{cases}
    \end{align*}
    For $C_i = \varnothing$, we clearly have $P_i = d_i \Delta_{C_i}$. If $C_i \neq \varnothing$, then each point $x \in P_i$ can be written in the form
    \begin{align*}
        x = \sum_{p \in C_i} \lambda_p \cdot \frac{d_i}{\vert \deg f_p \vert} e_p
    \end{align*}
    with coefficients $\lambda_p \in \R_{\geq 0}$. Since $d_i = \deg x = \sum_{p \in C_i} \lambda_p d_i$, this is a convex combination of elements in $d_i \Delta_{C_i}$. The other inclusion $d_i \Delta_{C_i} \subseteq P_i$ is immediate from the definition of the two polytopes.
\end{example}

Now let us return to the general case. We are ready to describe the face lattice of the polytope $\Delta_{C}^{(\underline d)}$, \ie the poset
\begin{align*}
    L(\Delta_{C}^{(\underline d)}) = \set{F \mid \text{$F$ face of $\Delta_{C}^{(\underline d)}$}}
\end{align*}
ordered by inclusion. It is well known that $L(\Delta_{C}^{(\underline d)})$ is a graded lattice of length $\dim \Delta_{C}^{(\underline d)} + 1$.

\begin{proposition}
    \label{prop:face_lattice}
    The following statements hold for each maximal chain $\mathfrak C$ in $A$:
    \begin{enumerate}[label=(\alph{enumi})]
        \item \label{itm:face_lattice_a} For any two subchains $C, D \subseteq \mathfrak C$ we have $\Delta_{C}^{(\underline d)} \subseteq \Delta_{D}^{(\underline d)}$, if and only if $C_{\underline d} \subseteq D_{\underline d}$.
        \item The map
        \begin{align*}
            F_{\mathfrak C}^{(\underline d)}: \set{C \in \Delta^{(\underline d)}(A) \mid C \subseteq \mathfrak C } \rightarrow L(\Delta_{\mathfrak C}^{(\underline d)}), \quad C \longmapsto \Delta_{C}^{(\underline d)}
        \end{align*}
        is an isomorphism of posets.
        \item \label{itm:face_lattice_c} For all $C \subseteq \mathfrak C$ the face $\Delta_{C}^{(\underline d)}$ is of dimension $\vert C_{\underline d} \vert - \dim \sigma_{C_{\underline d}}$ and its vertices lie in the $\Q$-span of the lattice $\mathcal L^{C, (\underline d)}$.
    \end{enumerate}
\end{proposition}
\begin{proof}
    For each $p \in \mathfrak C$ the set $\Delta_{\mathfrak C \setminus \set{p}}^{(\underline d)} = \Delta_{\mathfrak C}^{(\underline d)} \cap \set{x \in \R^{\mathfrak C} \mid p \notin \supp x}$ is a face of $\Delta_{\mathfrak C}^{(\underline d)}$. As an intersection of these faces, $\Delta_{C}^{(\underline d)}$ is a face as well for each $C \subseteq \mathfrak C$. By Lemma~\ref{lem:veronese_order_complex} it coincides with the face $\Delta_{C_{\underline d}}^{(\underline d)}$.

    \begin{enumerate}[label=(\alph{enumi})]
        \item For any two subsets $C, D \subseteq \mathfrak C$ the polytopes $\Delta_{C}^{(\underline d)}$ and $\Delta_{D}^{(\underline d)}$ agree if $C_{\underline d} = D_{\underline d}$. Conversely, suppose that $\Delta_{C}^{(\underline d)} \subseteq \Delta_{D}^{(\underline d)}$. We have seen in the proof of Lemma~\ref{lem:veronese_order_complex} that the monoid $\Gamma_C^{(\underline d)}$ contains an element $\underline a^{(p)}$ with $p \in \supp \underline a^{(p)}$ for every $p \in C_{\underline d}$, which implies $C_{\underline d} \subseteq D_{\underline d}$.
        \item First, we prove that every face of $\Delta_{\mathfrak C}^{(\underline d)}$ is induced by a subchain $C \subseteq \mathfrak C$. We set $D \coloneqq \mathfrak C_{\underline d}$. By part~\ref{itm:face_lattice_a}, we have $\Delta_{D}^{(\underline d)} = \Delta_{\mathfrak C}^{(\underline d)}$ and $\Delta_{D \setminus \set{p}}^{(\underline d)}$ is a facet of $\Delta_{D}^{(\underline d)}$ for all $p \in D$, \ie a face of codimension $1$. Let $F$ be any facet of $\Delta_{D}^{(\underline d)}$. Clearly it holds
        \begin{align*}
            F \subseteq \partial \Delta_{D}^{(\underline d)} \subseteq \bigcup_{p \in D} \mkern2mu \set{x \in \R^{D}_{\geq 0} \mid p \notin \supp x}.
        \end{align*}
        If $K$ is any convex set in $\R^{D}_{\geq 0}$ and $x, y \in K$, then there is a point $z \in K$ on the line through $x$ and $y$ with $\supp z = \supp x \cup \supp y$. This implies that there must exist an element $p \in D$ with
        \begin{align*}
            F \subseteq \Delta_{D}^{(\underline d)} \cap \set{x \in \R^{D}_{\geq 0} \mid p \notin \supp x} = \Delta_{D \setminus \set{p}}^{(\underline d)}.
        \end{align*}
        and since $\Delta_{D \setminus \set{p}}^{(\underline d)}$ is a facet of $\Delta_{D}^{(\underline d)}$, it follows $F = \Delta_{D \setminus \set{p}}^{(\underline d)}$. Because every face of a polytope is an intersection of facets, each face of $\Delta_{\mathfrak C}^{(\underline d)}$ is equal to $\Delta_{C}^{(\underline d)}$ for some $C \subseteq \mathfrak C$. Hence the map $F_{\mathfrak C}^{(\underline d)}$ is surjective. By~\ref{itm:face_lattice_a}, it is injective as well.
        \item Fix a subchain $C \subseteq \mathfrak C$. We know that the polytope $\Delta_{C}^{(\underline d)}$ is given by the intersection of a cone with an affine hyperplane of codimension one:
        \begin{align*}
            \Delta_{C}^{(\underline d)} = \operatorname{Cone} \Gamma_C^{(\underline d)} \cap \set{x \in \operatorname{span}_\R(\Gamma_C^{(\underline d)}) \mid \deg x = \underline d}.
        \end{align*}
        In particular, Lemma~\ref{lem:veronese_lattice} implies:
        \begin{align*}
            \dim \Delta_C^{(\underline d)} = \dim \operatorname{Cone} \Gamma_C^{(\underline d)} - 1 = \operatorname{rank} \mathcal L^{C, (\underline d)} - 1 = \vert C_{\underline d} \vert - \dim \sigma_{C_{\underline d}}.
        \end{align*}
        The vertices of $\Delta_{C}^{(\underline d)}$ are of the form $\Delta_D^{(\underline d)}$ for a subchain $D \subseteq C$. By the definition of these polytopes, we have $\Delta_D^{(\underline d)} = \set{\tfrac{1}{n} \underline a}$ for each element $\underline a \in \Gamma_D^{(\underline d)}$ of degree $n \underline d$. Hence the vertices of $\Delta_{C}^{(\underline d)}$ lie in the rational span of the lattice $\mathcal L^{C, (\underline d)}$.\hfill\qedhere
    \end{enumerate}
\end{proof}

As the face lattice of every polytope is a graded poset, it follows that $\Delta^{(\underline d)}(A)$ is the union of graded posets. In general, not all maximal chains in $\Delta^{(\underline d)}$ have the same length, but there still exists a rank function $r: \Delta^{(\underline d)}(A) \to \N_0$, where the rank of a chain $C \in \Delta^{(\underline d)}(A)$ is given by
\begin{align*}
    r(C) = \operatorname{rank} \mathcal L^{C, (\underline d)} = \vert C \vert - \dim \sigma_C + 1.
\end{align*}

\begin{definition}
    A \textbf{polytopal complex} in a finite-dimensional real vector space $W$ is a set $\mathcal K$ of polytopes in $W$ that satisfies the following properties:
    \begin{enumerate}[label=(\alph{enumi})]
        \item If $P \in \mathcal K$ and $Q$ is a (possibly empty) face of $P$, then $Q \in \mathcal K$;
        \item the intersection of any two polytopes $P, Q \in \mathcal K$ is a face of both $P$ and $Q$.
    \end{enumerate}
\end{definition}

\begin{definition}
    Let $\underline d \in \N_0^m$. We define the \textbf{Newton-Okounkov polytopal complex} of the Veronese subalgebra $R^{(\underline d)} \subseteq R$ as the union
    \begin{align*}
        \Delta_{\mathcal V}^{(\underline d)} = \bigcup_{\mathfrak C} \mkern2mu \Delta_{\mathfrak C}^{(\underline d)} \subseteq \R^A
    \end{align*}
    running over all maximal chains $\mathfrak C$ in $A$.
\end{definition}

By the description of the polytopes from equation~(\ref{eq:NO_body_is_polytope}) we have $\Delta_{C}^{(\underline d)} \cap \Delta_{D}^{(\underline d)} = \Delta_{C \cap D}^{(\underline d)}$ for all chains $C, D \subseteq A$. Therefore the set
\begin{align*}
    \mathcal K_{\mathcal V} = \set{\Delta_{C}^{(\underline d)} \mid C \in \Delta^{(\underline d)}(A)}
\end{align*}
is a polytopal complex. Technically, $\Delta_{\mathcal V}^{(\underline d)}$ is not a polytopal complex, but rather the geometric realization of the polytopal complex $\mathcal K_{\mathcal V}$.

\begin{remark}
    The Newton-Okounkov polytopal complex $\Delta_{\mathcal V}^{(\underline d)}$ can also be interpreted as a slice of the fan of cones:
    \begin{align*}
        \Delta_{\mathcal V}^{(\underline d)} = \Big( \bigcup_{\mathfrak C \subseteq A \atop \text{max.\,chain}} \mkern-5mu \operatorname{Cone} \Gamma_{\mathfrak C} \mkern3mu \Big) \cap \Set{x \in \R^A \mid \deg x = \underline d}.
    \end{align*}
    In the literature, the cone generated by $\Gamma_{\mathfrak C}$ is known as the \textit{global Newton-Okounkov body} of the algebra $\mathrm{gr}_{\mathcal V, \mathfrak C} R$ as it captures the behaviour of the Newton-Okounkov bodies of all its Veronese subalgebras. These global bodies were examined in \cite{cid2021multigraded} and \cite{lazarsfeld2009convex}.
\end{remark}

\section{Newton-Okounkov theory: Rational structures}

Let $\mathfrak C$ be a maximal chain in $A$. The semigroup 
\begin{align*}
    \widetilde{\Gamma}_{\mathfrak C}^{(\underline d)} = \mathcal L^{\mathfrak C, (\underline d)} \cap \operatorname{Cone} \Gamma_{\mathfrak C}^{(\underline d)} = \mathcal L^{\mathfrak C, (\underline d)} \cap \operatorname{Cone} \Gamma_{\mathfrak C}
\end{align*}
is called the \textit{saturation} of $\Gamma_{\mathfrak C}^{(\underline d)}$, as it is equal to the monoid of all $\underline a \in \mathcal L^{\mathfrak C, (\underline d)}$, such that there exists a natural number $k$ with $k \underline a \in \Gamma_{\mathfrak C}^{(\underline d)}$. Gordan's Lemma implies that the saturation is finitely generated. By definition, its elements are given by the lattice points in the scaled polytopes $n \Delta_{\mathfrak C}^{(\underline d)}$:
\begin{align}
    \label{eq:regularization_vs_polytope}
    (\widetilde{\Gamma}_{\mathfrak C}^{(\underline d)})_n = \set{\underline a \in \widetilde{\Gamma}_{\mathfrak C}^{(\underline d)} \mid \deg \underline a = n} = n \Delta_{\mathfrak C}^{(\underline d)} \cap \mathcal L^{\mathfrak C, (\underline d)}.
\end{align}
This links our problem of describing the leading function $G_R$ of the Hilbert polynomial to Ehrhart theory. The growth rate of an Ehrhart polynomial is determined by the dimension and the volume of the polytope. But as $\Delta_{\mathfrak C}^{(\underline d)}$ is not full-dimensional in the span of the lattice $\mathcal L^{\mathfrak C, (\underline d)}$, we first need to find a suitable rational structure.

\begin{definition}
    Let $P$ be a polytope in a real vector space $\R^d$. An \textbf{integral structure} (respectively \textbf{rational structure}) on $P$ is an affine embedding $\iota: P \hookrightarrow \R^{\dim P}$ together with a collection of subsets $P(n) \subseteq P$ for all $n \in \N$, such that the following conditions are fulfilled:
    \begin{enumerate}[label=(\alph{enumi})]
        \item The vertices of $\iota(P)$ have integral (respectively rational) coordinates;
        \item for each $n \in \N$ it holds
        \begin{align*}
            \iota(P(n)) = \set{ x \in \iota(P) \mid nx \in \Z^{\dim P}}.
        \end{align*}
    \end{enumerate}
\end{definition}

Having a rational structure on a given polytope $P$ allows the use of Ehrhart theory, even if $P$ is not full-dimensional in its ambient space: The cardinality
\begin{align*}
    \vert P(n) \vert = \vert \mkern2mu \iota(P) \cap \tfrac{1}{n} \Z^{\dim P} \vert = \vert \mkern2mu n \mkern1mu \iota(P) \cap \Z^{\dim P} \vert
\end{align*}
then is a quasi-polynomial in $n$ of degree $\dim P$ and its leading coefficient is a constant equal to the standard Euclidean volume of $\iota(P)$.

The main obstacle for constructing a rational structure for the polytope $\Delta_{\mathfrak C}^{(\underline d)}$ are the degrees appearing in the lattice $\mathcal L^{\mathfrak C, (\underline d)}$. To proceed, we need to show that $\mathcal L^{\mathfrak C, (\underline d)}$ is not empty in degree $1$. Unfortunately, this statement can actually be wrong in certain edge cases, as we show in the example below. However, when $\underline d$ lies in the relative interior of $\sigma_{\mathfrak C}$, then $\mathcal L^{\mathfrak C, (\underline d)}$ has elements of degree $1$: By Lemma~\ref{lem:veronese_lattice}, the lattice $\mathcal L^{\mathfrak C, (\underline d)}$ is the $\underline d$-th Veronese sublattice of $\mathcal L^{\mathfrak C}$, hence we only need to check whether $\mathcal L^{\mathfrak C}$ has an element of degree $\underline d$. This serves as the motivation for the next two lemmas.

\begin{example}
    \label{ex:lattice_not_in_deg_1}
    Consider the maximal chain $\mathfrak C \coloneqq \mathfrak C_2: X > 01 > 0$ in the Seshadri stratification from Example~\ref{ex:strat_y_0_y_1} and the degree $\underline d = (0,1) \in \N_0^2$. By Lemma~\ref{lem:veronese_lattice}, the lattice $\mathcal L^{\mathfrak C, (\underline d)}$ is generated by the monoid to the subchain $\mathfrak C_{\underline d} = \set{X}$. It follows from our computations in Example~\ref{ex:A} that every element $\underline a \in \Gamma$ with $\supp \underline a \subseteq \set{X}$ is a multiple of $e_X$. Hence $\mathcal L^{\mathfrak C, (\underline d)}$ is generated by $e_X$ and this element is of degree $2 \underline d$.
\end{example}

\begin{lemma}
    \label{lem:support_in_C}
    For each non-zero rational function $g \in \K(\hat X)$ and every maximal chain $\mathfrak C$ there exists a regular function $h \in \K[\hat X]$ with $\supp \mathcal V(gh^k) \subseteq \mathfrak C$ for all $k \in \N$.
\end{lemma}
\begin{proof}
    Let $p_r > \dots > p_0$ be the elements of the chain $\mathfrak C$. For each covering relation $q < p$ in $A$ we have the discrete valuation $\nu_{p,q}: \K(\hat X_p) \setminus \set{0} \to \Z$ of the prime divisor $\hat X_q \subseteq \hat X_p$ and the bond $b_{p,q} = \nu_{p,q}(f_p) \in \N$.
    
    The function $h$ can be constructed inductively over the length of the poset $A$. The statement of this lemma is trivial, when the length is zero. Otherwise, let $B$ be the set of all $q \in A \setminus \set{p_{r-1}}$ which are covered by $p_r$. For each element $q \in B$ we choose a regular function $h_q \in \K[\hat X]$, such that $h_q$ is the zero function on $\hat X_q$ and does not vanish identically on $\hat X_{p_{r-1}}$. Since $\nu_{p_r, q}(h_q) \geq 1$, we can choose natural numbers $n_q$, $q \in B$, fulfilling the following inequalities:
    \begin{align*}
         \frac{ n_q \mkern2mu \nu_{p_r, q}(h_q) }{ b_{p_r, q} } \geq \frac{\nu_{p_r, p_{r-1}}(g)}{b_{p_r, p_{r-1}}} - \frac{ \nu_{p_r, q}(g) }{ b_{p_r, q} }.
    \end{align*}
    We now define the regular function
    \begin{align*}
        h = \prod_{q \in B} h_q^{n_q} \in \K[\hat X].
    \end{align*}
    As all functions $h_q$ do not vanish identically on $\hat X_{p_{r-1}}$, we get $\nu_{p_r, p_{r-1}}(gh) = \nu_{p_r, p_{r-1}}(g)$. The choice of the number $n_q$ implies
    \begin{align*}
        \frac{ \nu_{p_r, q}(g h) }{ b_{p_r, q} } &= \frac{1}{ b_{p_r, q} } \bigg( \nu_{p_r, q}(g) + \sum_{p \in B} n_p \nu_{p_r, q}(h_p) \bigg) \geq \frac{1}{ b_{p_r, q} } \bigg( \nu_{p_r, q}(g) + n_q \nu_{p_r, q}(h_q) \bigg) \\
        &\geq \frac{\nu_{p_r, p_{r-1}}(g)}{b_{p_r, p_{r-1}}} = \frac{\nu_{p_r, p_{r-1}}(gh)}{b_{p_r, p_{r-1}}}.
    \end{align*}
    By the construction of the quasi-valuation, its values can be computed inductively using the induced Seshadri stratification on $X_{p_{r-1}}$ with the underlying poset $\set{q \in A \mid q \leq p_{r-1}}$. Let $\mathcal V_{p_{r-1}}$ denote its quasi-valuation. Then we have
    \begin{align*}
        \mathcal V(gh) = \frac{\nu_{p_r, p_{r-1}}(g)}{b_{p_r, p_{r-1}}} \mkern3mu e_{p_r} + \frac{\mathcal V_{p_{r-1}}(g_1)}{b_{p_r, p_{r-1}}},
    \end{align*}
    where $g_1$ is the rational function
    \begin{align*}
        g_1 = \frac{(gh)^{ b_{p_r, p_{r-1}} }}{ f_p^{\nu_{p_r, p_{r-1}}(gh)} } \bigg\vert_{\hat X_{p_{r-1}}}.
    \end{align*}
    Here we used the alternative description of the quasi-valuation from Remark 6.5 in \cite{seshstrat}. 
    
    By induction, there exists a non-zero function $h_1 \in \K[\hat X_{p_{r-1}}]$ with $\supp \mathcal V_{p_{r-1}}(g_1 h_1^k) \subseteq \mathfrak C \setminus \set{p_r}$ for all $k \in \N$. We choose any lift $\overline h_1$ of $h_1$ in $\K[\hat X]$. Note that we still have
    \begin{align*}
        \frac{ \nu_{p_r, q}(g h^k h_1^k) }{ b_{p_r, q} } \geq \frac{\nu_{p_r, p_{r-1}}(g)}{b_{p_r, p_{r-1}}} = \frac{\nu_{p_r, p_{r-1}}(gh^k h_1^k)}{b_{p_r, p_{r-1}}}.
    \end{align*}
    The quasi-valuation of $g h^k h_1^k$ is equal to
    \begin{align*}
        \mathcal V(gh^k h_1^k) = \frac{\nu_{p_r, p_{r-1}}(g)}{b_{p_r, p_{r-1}}} \mkern3mu e_{p_r} + \frac{\mathcal V_{p_{r-1}}(\widetilde g_1)}{b_{p_r, p_{r-1}}},
    \end{align*}
    for the regular function
    \begin{align*}
        \widetilde g_1 = \frac{(gh^k h_1^k)^{ b_{p_r, p_{r-1}} }}{ f_p^{\nu_{p_r, p_{r-1}}(g h^k h_1^k)} } \bigg\vert_{\hat X_{p_{r-1}}} = g_1 \cdot h_1^{ k \mkern2mu b_{p_r, p_{r-1}} }.
    \end{align*}
    In particular, $\supp \mathcal V(g h^k h_1^k)$ is contained in $\mathfrak C$ for every $k \in \N$.
\end{proof}

\begin{lemma}
    \label{lem:L_C_all_degrees}
    The degree map $\mathcal L^{\mathfrak C} \to \Z^m$ is surjective for each maximal chain $\mathfrak C$ in $A$.
\end{lemma}
\begin{proof}
    Again, we prove this statement via induction over the length of $A$. When it is zero, then $A$ only consists of one element, $m = 1$ and $\hat X$ is a line. Any linear function $g$ on the multicone has the property $\deg \mathcal V(g) = 1$, which implies the surjectivity.

    Now suppose that the length of $A$ is non-zero. Let $p_r > \dots > p_0$ be the elements of the chain $\mathfrak C$. If the index set $I_{p_{r-1}}$ is equal to $[m]$, then $\mathcal L^{\mathfrak C} \to \Z^m$ is surjective by induction. Otherwise $[m] \setminus I_{p_{r-1}}$ contains exactly one element, which we denote by $j$. By Lemma~\ref{lem:properties_A_mathcal_I} the multicone $\hat X_{p_{r-1}}$ is equal to
    \begin{align*}
        \hat X_{p_{r-1}} = \set{ (v_1, \dots, v_m) \in \hat X \mid v_i \in V_i \mkern5mu \forall i \in [m], v_j = 0}.
    \end{align*}
    Using the projection map $\hat X \twoheadrightarrow \hat X_{p_{r-1}}$ we view $\K[\hat X_{p_{r-1}}]$ as the graded subring
    \begin{align*}
        \bigoplus_{\underline d \in \N_0^m \atop d_j = 0} \K[\hat X]_{\underline d} \subseteq \K[\hat X].
    \end{align*}
    We choose any non-zero linear function $\ell \in V_j^*$. By induction, each tuple $\underline d \in \Z^m$ with $d_j = 0$ lies in the image of $\mathcal L^{\mathfrak C} \to \Z^m$. Our goal is to construct a function $g \in \K[\hat X_{p_{r-1}}]$ with $\supp \mathcal V(g \ell) \subseteq \mathfrak C$. Then, by Lemma~\ref{lem:multihom_valuation}, we can assume that $g$ was multihomogeneous of degree $\underline d \in \N_0^m$ with $d_j = 0$ and it follows from Lemma~\ref{lem:quasi_val_compatible_with_multidegree} that the $j$-th component of $\deg \mathcal V(g \ell)$ is equal to $1$. As $\mathcal V(g \ell) \in \Gamma_{\mathfrak C}$, the map $\mathcal L^{\mathfrak C} \to \Z^m$ must be surjective.
    
    Let $B$ be the set of all $q \in A \setminus \set{p_{r-1}}$ covered by $p_r$. For every $q \in B$ the intersection
    \begin{align*}
        \hat X_{q} \cap \hat X_{p_{r-1}}  = \set{ (v_1, \dots, v_m) \in \hat X_q \mid v_i \in V_i \mkern5mu \forall i \in [m], v_j = 0}
    \end{align*}
    is a proper subvariety of $\hat X_{p_{r-1}}$, otherwise this would imply $p_{r-1} \leq q$. Hence we can choose a non-zero regular function $h_q \in \K[\hat X_{p_{r-1}}]$, which restricts to the zero function on $\hat X_{q} \cap \hat X_{p_{r-1}}$. Seen as a function on $\hat X$, $h_q$ vanishes on the whole multicone $\hat X_q$. Similar to the proof of Lemma~\ref{lem:support_in_C} we choose $n_q \in \N$ with
    \begin{align*}
         \frac{n_{q}}{ b_{p_r, q} } \, \nu_{p_r, q}(h_q) \geq \frac{\nu_{p_r, p_{r-1}}(\ell)}{b_{p_r, p_{r-1}}}
    \end{align*}
    and define the regular function
    \begin{align*}
        g = \prod_{q \in B} h_q^{n_q} \in \K[\hat X_{p_{r-1}}].
    \end{align*}
    By the choice of the number $n_q$ we get
    \begin{align*}
        \frac{ \nu_{p_r, q}(g \ell) }{ b_{p_r, q} } \geq \frac{n_{q}}{ b_{p_r, q} } \, \nu_{p_r, q}(h_q) \geq \frac{\nu_{p_r, p_{r-1}}(\ell)}{b_{p_r, p_{r-1}}} = \frac{\nu_{p_r, p_{r-1}}(g \ell)}{b_{p_r, p_{r-1}}},
    \end{align*}
    so $p_{r-1}$ lies in every maximal chain $\mathfrak D \subseteq A$ with $\supp \mathcal V(g \ell) \subseteq \mathfrak D$. But in general $\supp \mathcal V(g \ell)$ is not contained in $\mathfrak C$. To achieve this, we need to multiply $g \ell$ by another suitable function, which we get from Lemma~\ref{lem:support_in_C}: There exists a regular function $h \in \K[\hat X_{p_{r-1}}]$ with $\supp \mathcal V_{p_{r-1}}(g_1 h^k) \subseteq \mathfrak C \setminus \set{p_r}$ for all $k \in \N$, where $g_1$ is the rational function
    \begin{align*}
        g_1 = \frac{(g \ell)^{ b_{p_r, p_{r-1}} }}{ f_p^{\nu_{p_r, p_{r-1}}(g \ell)} } \bigg\vert_{\hat X_{p_{r-1}}}.
    \end{align*}
    Then we have
    \begin{align*}
        \mathcal V(gh \ell) = \frac{\nu_{p_r, p_{r-1}}(g \ell)}{b_{p_r, p_{r-1}}} \mkern3mu e_{p_r} + \frac{\mathcal V_{p_{r-1}}(g_1 h^{ b_{p_r, p_{r-1}}} )}{b_{p_r, p_{r-1}}},
    \end{align*}
    so the support of $\mathcal V_{p_{r-1}}(gh \ell)$ is contained in $\mathfrak C$. This completes the proof.
\end{proof}

With the help of the last lemma, we can now construct a rational structure on the polytope $\Delta_{\mathfrak C}^{(\underline d)}$. The equation $(\widetilde{\Gamma}_{\mathfrak C}^{(\underline d)})_n = \mathcal L^{\mathfrak C}_{n \underline d} \cap \operatorname{Cone} \Gamma_{\mathfrak C}$ suggests that these structures should be compatible for different $\underline d$ in the following sense.

\begin{proposition}
    \label{prop:global_rational_structure}
    For each maximal chain $\mathfrak C$ in $A$ there exists a linear map $\mathrm{pr}: \R^{\mathfrak C} \to \R^r$ with the following property: For each $\underline d \in \operatorname{relint} \sigma_{\mathfrak C}$ the subsets
    \begin{align*}
        \Delta_{\mathfrak C}^{(\underline d)}(n) = \set{ \tfrac{1}{n} \underline a \mid \underline a \in (\widetilde{\Gamma}_{\mathfrak C}^{(\underline d)})_n } \subseteq \Delta_{\mathfrak C}^{(\underline d)}
    \end{align*}
    for $n \in \N$ form a rational structure on $\Delta_{\mathfrak C}^{(\underline d)}$ together with the map $\Delta_{\mathfrak C}^{(\underline d)} \hookrightarrow \R^{\mathfrak C} \to \R^r$.
\end{proposition}
\begin{proof}
    By Lemma~\ref{lem:L_C_all_degrees}, we can choose elements $\underline b^{(1)}, \dots, \underline b^{(m)} \in \mathcal L^{\mathfrak C}$ with $\deg \underline b^{(i)} = e_i$. They define a group homomorphism
    \begin{align*}
        \mathcal L^{\mathfrak C} \to \mathcal L_0^{\mathfrak C}, \quad \underline a \mapsto \underline a - \sum_{i=1}^m c_i \mkern2mu \underline b^{(i)} \quad \text{for $\deg \underline a = (c_1, \dots, c_m)$}.
    \end{align*}
    Note that the lattice $\mathcal L_0^{\mathfrak C}$ of degree $0$ elements in $\mathcal L^{\mathfrak C}$ is of rank $\vert \mathfrak C \vert - m = r$. We extend the above map to an $\R$-linear map $\mathrm{pr}: \R^{\mathfrak C} \to U_0$, where $U_0$ is the real span of $\mathcal L_0^{\mathfrak C} \cong \Z^r$ in $\R^{\mathfrak C}$. For each $\underline d \in \operatorname{relint} \sigma_{\mathfrak C}$ the affine subspace $U_{\underline d} \coloneqq \set{ x \in \R^{\mathfrak C} \mid \deg x = \underline d}$ contains the polytope $\Delta_{\mathfrak C}^{(\underline d)}$ and both are of dimension $r$. By construction, $\mathrm{pr}$ induces a bijection $\mathcal L_{\underline d}^{\mathfrak C} \to \mathcal L_0^{\mathfrak C}$, $\underline a \mapsto \mathrm{pr}(\underline a)$, so the composition $\iota_{\underline d}: \Delta_{\mathfrak C}^{(\underline d)} \hookrightarrow \R^{\mathfrak C} \to U_0$ is an affine embedding. Furthermore it maps the set
    \begin{align*}
        \Delta_{\mathfrak C}^{(\underline d)}(1) = (\widetilde{\Gamma}_{\mathfrak C}^{(\underline d)})_1 = \Delta_{\mathfrak C}^{(\underline d)} \cap \mathcal L^{\mathfrak C}
    \end{align*}
    bijectively to $\iota_{\underline d}(\Delta_{\mathfrak C}^{(\underline d)}) \cap \mathcal L_0^{\mathfrak C}$. Since $n \Delta_{\mathfrak C}^{(\underline d)} = \Delta_{\mathfrak C}^{(n \underline d)}$ for all $n \in \N$, it follows
    \begin{align*}
        \iota_{\underline d}(\Delta_{\mathfrak C}^{(\underline d)}(n)) = \iota_{\underline d} \big( \tfrac{1}{n} (\Delta_{\mathfrak C}^{(n \underline d)} \cap \mathcal L^{\mathfrak C}) \big) = \tfrac{1}{n} \iota_{\underline d}(\Delta_{\mathfrak C}^{(n \underline d)}) \cap \mathcal L_0^{\mathfrak C} = \iota_{\underline d}(\Delta_{\mathfrak C}^{(\underline d)}) \cap \tfrac{1}{n} \mathcal L_0^{\mathfrak C}.
    \end{align*}
    Lastly, we have seen in Proposition~\ref{prop:face_lattice} that the vertices of $\Delta_{\mathfrak C}^{(\underline d)}$ lie in the $\Q$-span of $\mathcal L^{\mathfrak C}$. As the map $\mathrm{pr}$ is compatible with the lattices, $\iota_{\underline d}$ defines a rational structure on $\Delta_{\mathfrak C}^{(\underline d)}$.
\end{proof}

\section{Newton-Okounkov theory: The leading term of the Hilbert polynomial}

\begin{theorem}
    \label{thm:G_R_formula}
    For each maximal chain $\mathfrak C$ in $A$ we fix a map $\mathrm{pr}_{\mathfrak C}: \R^{\mathfrak C} \to \R^r$ as in Proposition~\ref{prop:global_rational_structure}. If $\underline d \in \N_0^m$ does not lie on the boundary of $\sigma_{\mathfrak C}$ for any maximal chain $\mathfrak C$, then it holds
    \begin{align*}
        G_R(\underline d) = \sum_{\mathfrak C} \operatorname{vol}(\mathrm{pr}_{\mathfrak C}(\Delta_{\mathfrak C}^{(\underline d)})),
    \end{align*}
    where the sum runs over all maximal chains in $A$.
\end{theorem}
\begin{proof}
    As $R$ is finitely generated in total degree $1$, its Veronese subalgebra $R^{(\underline d)}$ is finitely generated in degree one, so the Hilbert quasi-polynomial of $R^{(\underline d)}$ is a polynomial. By (\ref{eq:dimension_equality_R_gr}) it coincides with the Hilbert quasi-polynomial $H^{(\underline d)}$ of $\mathrm{gr}_{\mathcal V}^{(\underline d)} R$. We have seen in Lemma~\ref{lem:irreducible_components_gr_V_d} that the associated projective variety of this degenerated algebra is the irredundant union of its irreducible components $\Proj(\mathrm{gr}_{\mathcal V, C}^{(\underline d)} R)$, where $C$ runs over the set $\mathcal C$ of all maximal elements in $\Delta^{(\underline d)}(A)$. Since $\underline d$ does not lie on the boundary of $\sigma_{\mathfrak C}$, we know that $\mathcal C$ consists exactly of the maximal chains $\mathfrak C$ in $A$ with $\underline d \in \sigma_{\mathfrak C}$.
    
    By Lemma~\ref{lem:veronese_lattice} the component $\Proj(\mathrm{gr}_{\mathcal V, \mathfrak C}^{(\underline d)} R)$ for $\mathfrak C \in \mathcal C$ is a projective toric variety of dimension $\vert \mathfrak C \vert - \dim \sigma_{\mathfrak C} = r$. In particular, 
    \begin{align*}
        G_R(\underline d) = \lim_{n \to \infty} \frac{\vert \Gamma_{n \underline d} \vert}{n^r}
    \end{align*}
    computes the coefficient $a$ of the monomial $x^r$ in $H^{(\underline d)} \in \Q[x]$. It is zero, when $\mathcal C$ is empty, otherwise it is the leading coefficient of $H^{(\underline d)}$. 
    
    For $\mathfrak C \in \mathcal C$ let $H_{\mathfrak C}^{(\underline d)}$ denote the Hilbert quasi-polynomial of $\mathrm{gr}_{\mathcal V, \mathfrak C}^{(\underline d)} R \cong \K[\Gamma_{\mathfrak C}^{(\underline d)}]$. Then $a$ is given by the sum of the leading terms $a_{\mathfrak C}$ of all quasi-polynomials $H_{\mathfrak C}^{(\underline d)}$. Using the arguments from the Lemmas 9.9 and 9.10 in \cite{seshstrat}, one can prove that $a_{\mathfrak C}$ is the leading term of the Hilbert quasi-polynomial of $\K[\widetilde{\Gamma}_{\mathfrak C}^{(\underline d)}]$, induced by the saturated monoid. By Proposition~\ref{prop:global_rational_structure} this quasi-polynomial is an Ehrhart quasi-polynomial, so $a_{\mathfrak C}$ is constant and equal to the volume of the embedded polytope $\mathrm{pr}_{\mathfrak C}(\Delta_{\mathfrak C}^{(\underline d)})$. This completes the proof, as $\Delta_{\mathfrak C}^{(\underline d)}$ is the empty polytope for all maximal chains $\mathfrak C \subseteq A$ not in $\mathcal C$.
\end{proof}

\begin{example}
    In the Seshadri stratification of Hodge type from Example~\ref{ex:strat_y_1} there are the four maximal chains $\mathfrak C_1, \dots, \mathfrak C_4$ from left to right.
    
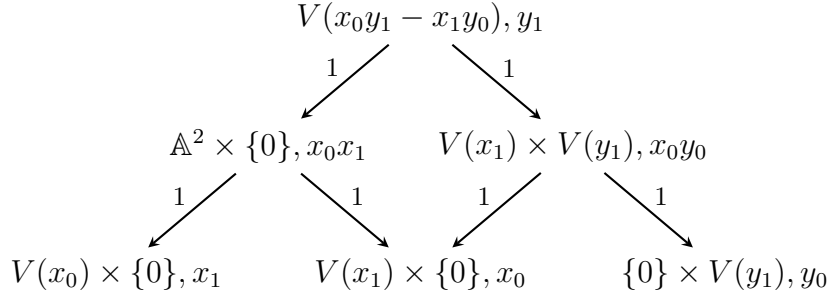
\begin{center}
        \begin{tikzpicture}
        \node (X) at (0,0) {$V(x_0 y_1 - x_1 y_0), y_1$};
        \node (1100) at (-2,-1.7) {$\A^2 \times \set{0}, x_0 x_1$};
        \node (1010) at (2,-1.7) {$V(x_1) \times V(y_1), x_0 y_0$};
        \node (0100) at (-4,-3.4) {$V(x_0) \times \set{0}, x_1$};
        \node (1000) at (0,-3.4) {$V(x_1) \times \set{0}, x_0$};
        \node (0010) at (4,-3.4) {$\set{0} \times V(y_1), y_0$};
        
        \draw [thick, -stealth] (X) -- node[right, above, xshift=0.4em, yshift=-0.2em]{\footnotesize $1$}(1010);
        \draw [thick, -stealth] (X) -- node[left, above, xshift=-0.4em, yshift=-0.2em]{\footnotesize $1$}(1100);
        \draw [thick, -stealth] (1010) -- node[left, above, xshift=-0.4em, yshift=-0.2em]{\footnotesize $1$}(1000);
        \draw [thick, -stealth] (1010) -- node[right, above, xshift=0.4em, yshift=-0.2em]{\footnotesize $1$}(0010);
        \draw [thick, -stealth] (1100) -- node[right, above, xshift=0.4em, yshift=-0.2em]{\footnotesize $1$}(1000);
        \draw [thick, -stealth] (1100) -- node[left, above, xshift=-0.4em, yshift=-0.2em]{\footnotesize $1$}(0100);
        \end{tikzpicture}
    \end{center}
    For a tuple $\underline d \in \N_0^2$ in the interior of $\sigma_{\mathfrak C_1} = \R_{\geq 0}^2$ the vertices of the polytope $\Delta_{\mathfrak C_1}^{(\underline d)}$ are given by $d_1 e_0 + d_2 e_X$ and $\tfrac12 d_1 e_{01} + d_2 e_X$. By fixing the element $\underline b^{(1)} = e_0$ of degree $(1,0)$ and $\underline b^{(2)} = e_X$ of degree $(0,1)$ we get an integral structure on $\Delta_{\mathfrak C_1}^{(\underline d)}$ via the proof of Proposition~\ref{prop:global_rational_structure}. It identifies the vertices with the points $0$ and $d_1 (\tfrac12 e_{01} - e_0)$ in the lattice $\mathcal L_0^{\mathfrak C_1} = \Z \cdot (e_{01} - 2 e_0)$. The volume of the resulting polytope in the linear span of this lattice is equal to $\tfrac12 d_1$.

    In the same way one can compute an integral structure on $\Delta_{\mathfrak C_1}^{(\underline d)}$ with volume $\tfrac12 d_1$. For the third maximal chain $\mathfrak C_3$ we again have $\sigma_{\mathfrak C_3} = \R_{\geq 0}^2$. We fix the elements $\underline b^{(1)} = e_1$ of degree $(1,0)$ and $\underline b^{(2)} = e_X$ of degree $(0,1)$ for the integral structure. For $\underline d \in \N_0^2$ in the interior of this cone, one needs to distinguish between two cases. If $d_1 \geq d_2$, then the polytope $\Delta_{\mathfrak C_3}^{(\underline d)}$ has the vertices $d_1 e_1 + d_2 e_X$ and $(d_1 - d_2) e_1 + d_2 e_{0 \overline 0}$, which correspond to the points $0$ and $d_2 (- e_1 + e_{0 \overline 0} - e_X)$ in $\mathcal L_0^{\mathfrak C_1} = \Z \cdot (- e_1 + e_{0 \overline 0} - e_X)$. Hence we get the volume $d_2$. Analogously, we have the volume $d_1$ in the case $d_2 \geq d_1$.

    The cone of the last maximal chain $\mathfrak C_4$ is spanned by $(1,1)$ and $(0,1)$. For every $\underline d \in \N_0^2$ not contained in this cone, the polytope $\Delta_{\mathfrak C_4}^{(\underline d)}$ is empty. Otherwise $\Delta_{\mathfrak C_4}^{(\underline d)}$ has the vertices $d_1 e_{0 \overline 0} + (d_2 - d_1) e_{\overline 0}$ and $d_1 e_{0 \overline 0} + (d_2 - d_1) e_X$. Via the elements $\underline b^{(1)} = e_{0 \overline 0} - e_{\overline 0}$ and $\underline b^{(2)} = e_{\overline 0}$ we get the volume $d_2 - d_1$.

    With these volumes, we can now compute the leading term of the Hilbert polynomial:
    \begin{align*}
        G_R(\underline d) = \left. \begin{cases}
            \tfrac12 d_1 + \tfrac12 d_1 + d_2, & \text{for $d_1 > d_2 > 0$} \\
            \tfrac12 d_1 + \tfrac12 d_1 + d_1 + (d_2 - d_1), & \text{for $d_2 > d_1 > 0$}
        \end{cases} \right\} = d_1 + d_2.
    \end{align*}
    The multidegrees of $X$ are therefore given by $\deg_{(1,0)}(X) = \deg_{(0,1)}(X) = 1$. Indeed, the multiprojective coordinate ring $R = \K[x_0, x_1, y_0, y_1]/(x_0 y_1 - x_1 y_0)$ has a basis consisting of all monomials $x_0^a x_1^b y_0^c y_1^d$ with $a,b,c,d \in \N_0$ and $bc = 0$. Hence the graded component $R_{\underline d}$ is of dimension $(d_1 + 1) + (d_2 + 1) - 1 = d_1 + d_2 + 1$. This is already a polynomial in $\underline d$ and its leading term agrees with the function $G_R$ we computed above. 
\end{example}

\section{Seshadri stratifications of LS-type}
\label{sec:LS-type}

For suitable choices of extremal functions, the polytopes $\Delta_{\mathfrak C}^{(\underline d)}$ are products of simplices for all $\underline d$, e.\,g. when the support of $\deg f_p$ is a one-element set for each $p \in A$ (see Example~\ref{ex:multisimplices}) or when $I_p = I_q \in \mathcal I$ is equivalent to $\deg f_p = \deg f_q$ for all $p, q \in A$ (see below). One might ask if there exists a rational structure as in Proposition~\ref{prop:global_rational_structure}, that is compatible with this decomposition into simplices, \ie the map $\mathrm{pr}_{\mathfrak C}: \R^{\mathfrak C} \to \R^r$ is a product of rational structures, one for each simplex. In general, this idea is too naive: It already fails for the stratification we examined in Example~\ref{ex:lattice_not_in_deg_1}, since the lattice $\mathcal L^{\mathfrak C}$ to the maximal chain $\mathfrak C: X > 01 > 0$ does not decompose into the product $\mathcal L^{\set{0, \mkern2mu 01}} \times \mathcal L^{\set{X}}$. However, when all the monoids $\Gamma_{\mathfrak C}$ are so-called LS-monoids, then such a decomposition does exist and one can compute the volumes of the polytopes explicitly via the bonds in the stratification.

\begin{definition}
    Let $C$ be a chain of covering relations in $A$, \ie it consists of elements $p_s > \dots > p_0$ such that $p_i$ covers $p_{i-1}$ for each $i = 1, \dots, s$. The \textbf{Lakshmibai-Seshadri-lattice} (short: LS-lattice) associated to $C$ is the lattice
    \begin{align*}
        \mathrm{LS}_C = \Set{ \sum_{i=0}^s a_i e_{p_i} \in \Q^C \ \middle\vert \ b_{p_i, p_{i-1}} (a_i + \dots + a_s) \in \Z \ \forall i \in [s], a_0 + \dots + a_s \in \Z}.
    \end{align*}
    and its intersection $\mathrm{LS}_C^+ = \mathrm{LS}_C \cap \Q^C_{\geq 0}$ with the positive orthant is called the \textbf{LS-monoid} to the chain $C$.
\end{definition}

Every LS-lattice is generated by its LS-monoid, since one can shift each element in $\mathrm{LS}_C$ into the positive orthant via the vectors $e_{p_i} \in \mathrm{LS}_C$. LS-lattices are also compatible with subchains: If $D \subseteq C$ are two chains of covering relations, then we have $\mathrm{LS}_C \cap \Q^D = \mathrm{LS}_D$. This has the following consequences: If the monoid $\Gamma_{\mathfrak C}$ is an LS-monoid for every maximal chain $\mathfrak C$, then the Seshadri stratification is normal. The set $\mathbb G \subseteq \Gamma$ of all indecomposable elements is finite and for every $\underline u = \sum_{p \in A} u_p e_p \in \mathbb G$ the coefficients $u_p$ add up to $1$ (this follows from the proof of Lemma 3.3 in~\cite{seshstratnormal}).

\begin{definition}
    We call a Seshadri stratification on $X \subseteq \prod_{i=1}^m \PP(V_i)$ \textbf{of LS-type}, if the following conditions are fulfilled:
    \begin{enumerate}[label=(\alph{enumi})]
        \item \label{itm:LS_a} Each component of the multidegree $\deg f_p \in \N_0^m$ is at most $1$ for all $p \in A$;
        \item \label{itm:LS_b} if $I_p = I_q$ for any two elements $p, q \in A$, then $\deg f_p = \deg f_q$;
        \item \label{itm:LS_c} the fan of monoids $\Gamma$ is equal to the union $\bigcup_{\mathfrak C} \mathrm{LS}_{\mathfrak C}^+$ over all maximal chains $\mathfrak C$ in $A$.
    \end{enumerate}
\end{definition}

The next remark implies that this definition generalizes the notion of a Seshadri stratification of LS-type from~\cite[Definition 2.6]{seshstratnormal}. For $m = 1$ both definitions agree.

\begin{remark}
    For every stratification of LS-type, the monoid $\Gamma_{\mathfrak C}$ agrees with the LS-monoid $\mathrm{LS}_{\mathfrak C}^+$ for each maximal chain in $A$: We clearly have $\mathrm{LS}_{\mathfrak C}^+ \subseteq \Gamma \cap \Q^{\mathfrak C} = \Gamma_{\mathfrak C}$. For the reverse inclusion, let $p_r > \dots > p_0$ be the elements in $\mathfrak C$ and $b_j = b_{p_j, p_{j-1}}$ be the bond to the covering relation $p_j > p_{j-1}$ for all $j = 0, \dots, r$. By the definition of an LS-lattice, each element $\underline a^{(j)} = \tfrac{1}{b_j} e_{p_j} - \tfrac{1}{b_j} e_{p_{j-1}}$ is contained in $\mathrm{LS}_{\mathfrak C} \subseteq \mathcal L^{\mathfrak C}$. Hence one can find rational functions $F_r, \dots, F_1 \in \K(\hat X) \setminus \set{0}$ with $\mathcal V(F_j) = \underline a^{(j)}$ for all $j = 1, \dots, r$. We can now use Proposition~\ref{prop:valuation_lattice}: The matrix $B_{\mathfrak C}$ is given by
    \begin{align*}
        B_{\mathfrak C} = \begin{pmatrix}
            b_{r}^{-1} & 0 & \cdots & \cdots & 0 \\
            - b_{r}^{-1} & b_{r-1}^{-1} & \ddots & & \vdots \\
            0 & - b_{r-1}^{-1} & \ddots & \ddots & \vdots \\
            \vdots & \ddots & \ddots & b_{1}^{-1} & 0 \\
            0 & 0 & \cdots & - b_{1}^{-1} & b_{0}^{-1}
        \end{pmatrix}^{-1} = 
        \begin{pmatrix}
            b_r & 0 & \cdots & \cdots & 0 \\
            b_{r-1} & b_{r-1} & \ddots & & \vdots \\
            \vdots & \vdots & \ddots & \ddots & \vdots \\
            b_{1} & b_{1} & \cdots & b_1 & 0 \\
            b_{0} & b_0 & \cdots & b_0 & b_0
        \end{pmatrix}
    \end{align*}
    Therefore $\mathcal L^{\mathfrak C}$ is contained in $\mathrm{LS}_{\mathfrak C}$ and by intersecting with $\Q^{\mathfrak C}_{\geq 0}$ we get $\Gamma_{\mathfrak C} \subseteq \mathrm{LS}_{\mathfrak C}^+$.
\end{remark}

We encounter our first LS-type stratification in Section~\ref{sec:type_A}. More examples will be provided in a future article by stratifying Schubert-varieties. Both of our running examples~\ref{ex:strat_y_1} and \ref{ex:strat_y_0_y_1} do not fulfill the condition~\ref{itm:LS_a} in the above definition. But their fans of monoids are still of LS-type, \ie the condition~\ref{itm:LS_c} is satisfied. This is clearly the case for Hodge-type stratifications and we computed the fan of monoids of the second stratification in Example~\ref{ex:A}. One might try to modify these stratifications by raising the extremal functions to suitable powers. In this way, condition~\ref{itm:LS_b} can be fulfilled, but unfortunately the resulting fans of monoids (which are scaled by the inverses of these powers) might not be of LS-type anymore. 

Until the end of this section we fix a Seshadri stratification of LS-type on a multiprojective variety $X$. The poset $\mathcal I$ defines a partition of $A$ into the subsets
\begin{align*}
    A_I = \set{p \in A \mid I_p = I}
\end{align*}
for $I \in \mathcal I$. By definition, all extremal functions in $A_I$ have the same multidegree. We see in the next lemma that this degree is always given by
\begin{align*}
    e_I = \sum_{i \in \underline I} e_i \in \N_0^m
\end{align*}
\label{txt:def_underline_I}for a subset $\underline I \subseteq I$ characterized by the covering relations in the index poset $\mathcal I$: If $I$ is a minimal element in $\mathcal I$ then it holds $\underline I = I$, otherwise $\underline I$ is the union of all sets $I \setminus J$, where $J \subsetneq I$ is a covering relation in $\mathcal I$.

\begin{example}
    If $\mathcal I$ is totally ordered, one can assume \wwlog{} that it consists of the sets $[i]$ for all $i \in [m]$. In this case, we have $\underline{[i]} = \set{i}$.

    To give another example, consider the poset $\mathcal I$ with the elements $I = \set{2}$, $J = \set{1,2}$, $K = \set{2,3}$ and $L = [3]$. Here $\underline I = \set{2}$, $\underline J = \set{1}$, $\underline K = \set{3}$ and $\underline L = \set{1,3}$.
\end{example}

\begin{lemma}
    For all $p \in A$ it holds $\deg f_p = e_{I_p}$.
\end{lemma}
\begin{proof}
    We fix an element $I \in \mathcal I$ and let $\underline d = (d_1, \dots, d_m)$ be the multidegree of any extremal function $f_p$ for $p \in A_I$. For all $i \in \underline I$ there exists a covering relation $q < p$ in $A$ with $I_p \setminus I_q = \set{i}$ and we have $d_i \neq 0$ by Lemma~\ref{lem:projective_covering_relation}\,\ref{itm:projective_covering_relation_b}.
    
    Conversely, let $d_i = 1$ for some $i \in I$ and let $p$ be any element of $A_I$. Then the subvariety
    \begin{align*}
        Y = \set{ (v_1, \dots, v_m) \in \hat X_p \mid v_j \in V_j \mkern6mu \forall j \in [m], v_i = 0}
    \end{align*}
    of $\hat X_p$ is irreducible, contained in the vanishing set of $f_p$ and the codimension of $Y$ in $\hat X_p$ is at least one. If $\operatorname{codim}_Y(\hat X_p) = 1$, then $i \in \underline I$. Otherwise there exists an element $q \in A_I$ with $q < p$ and $Y \subseteq \hat X_q$ and we can proceed by induction over the codimension of $Y$.
\end{proof}

Fix a maximal chain $\mathfrak C$ in $A$ with associated maximal chain $I_1 \subsetneq \dots \subsetneq I_m = [m]$ in $\mathcal I$. It defines a decomposition of $\mathfrak C$ into the subchains $\mathfrak C_j = \set{p \in \mathfrak C \mid I_p = I_j}$. The covering relation $\min \mathfrak C_j > \max \mathfrak C_{j-1}$ has bond $1$ by Lemma~\ref{lem:projective_covering_relation}\,\ref{itm:projective_covering_relation_b}. It follows from the definition of LS-lattices, that they decompose into a product of sublattices along covering relations with bond $1$:
\begin{align*}
    \mathrm{LS}_{\mathfrak C} = \mathrm{LS}_{\mathfrak C_1} \times \dots \times \mathrm{LS}_{\mathfrak C_m} \subseteq \Q^{\mathfrak C}.
\end{align*}
Therefore $\mathcal L^{\mathfrak C}$ is equal to the product of the sublattices $\mathcal L^{\mathfrak C_j} \subseteq \Q^{\mathfrak C_j}$ generated by $\Gamma_{\mathfrak C_j}$. Of course, this is compatible with the monoids as well: $\Gamma_{\mathfrak C} = \Gamma_{\mathfrak C_1} \times \dots \times \Gamma_{\mathfrak C_m}$. We define the $m \times m$-matrix $M_{\mathfrak C}$ with entries in $\Z$, such that its $j$-th column consists of the degree vector $e_{I_j} \in \N_0^m$. Its follows from Lemma~\ref{lem:projective_covering_relation}\,\ref{itm:projective_covering_relation_b} that this matrix is invertible over $\Z$, so its inverse gives rise to a group isomorphism
\begin{align*}
    \phi^{\mathfrak C}: \Z^m \to \Z^m, \quad \underline d \mapsto M_{\mathfrak C}^{-1} \underline d
\end{align*}
identifying $\sigma_{\mathfrak C} \cap \N_0^m$ with $\N_0^m$. For $j \in [m]$ let $\phi_j^{\mathfrak C}: \Z^m \to \Z$ be its projection onto the $j$-th component. This allows us to show
\begin{align}
    \label{eq:LS_type_veronese_lattice}
    \mathcal L^{\mathfrak C, (\underline d)} = \set{\underline a \in \mathcal L^{\mathfrak C} \mid \deg \underline a \in \Z \underline d}
\end{align}
for all $\underline d \in \sigma_{\mathfrak C}$. Each element $\underline a \in \mathcal L^{\mathfrak C}$ with $\deg \underline a \in \Z \underline d$ can be written as $\underline a = \underline b - \underline c$ for $\underline b, \underline c \in \Gamma_{\mathfrak C}$. Using the isomorphism $\phi^{\mathfrak C}$ one can find an element $\underline a' \in \sum_{p \in \mathfrak C} \N_0 e_p$, such that $\deg (\underline b + \underline a') \in \Z \underline d$. Then we have $\underline b + \underline a', \underline c + \underline a' \in \Gamma_{\mathfrak C}^{(\underline d)}$, hence $\underline a \in \mathcal L^{\mathfrak C, (\underline d)}$.

Let $r$ be the dimension of $X$. The Newton-Okounkov polytopes of a stratification of LS-type decompose into products of simplices: For each maximal chain $\mathfrak C$ and $\underline d \in \sigma_{\mathfrak C}$ we can write the polytope $\Delta_{\mathfrak C}^{(\underline d)}$ in the form
\begin{align*}
    \Delta_{\mathfrak C}^{(\underline d)} &= \R_{\geq 0}^{\mathfrak C} \cap \set{ x \in \R^{\mathfrak C} \mid \deg x = \underline d} \\
    &= \prod_{j=1}^m \R_{\geq 0}^{\mathfrak C_j}  \cap \set{ x \in \R^{\mathfrak C_j} \mid \deg x = \phi_j^{\mathfrak C}(\underline d) \mkern1mu e_{I_j} } = \prod_{j=1}^m \Delta_{\mathfrak C_j}^{(\phi_j^{\mathfrak C}(\underline d) \mkern1mu e_{I_j})}.
\end{align*}
Hence $\Delta_{\mathfrak C}^{(\underline d)}$ is a multisimplex, since we have
\begin{align*}
    \Delta_{\mathfrak C_j}^{(\phi_j^{\mathfrak C}(\underline d) \mkern1mu e_{I_j})} = \phi_j^{\mathfrak C}(\underline d) \Delta_{\mathfrak C_j},
\end{align*}
where $\Delta_{\mathfrak C_j}$ is the convex hull of all vectors $e_p$ for $p \in \mathfrak C_j$. For fixed $j \in [m]$ let $p_s > \dots > p_0$ be the elements of the subchain $\mathfrak C_j$ and $b_{k,k-1}$ be the bond of the covering relation $p_k > p_{k-1}$ in $A$ for $k = 1, \dots, s$. We define the linear map
\begin{align*}
    \mathrm{pr}_{\mathfrak C_j}: \R^{\mathfrak C_j} \to \R^s, \quad e_{p_i} \mapsto \begin{cases}
        0, & \text{if $i=0$}, \\
        \sum_{k=1}^i b_{k,k-1} e_k, & \text{if $i \geq 1$}.
    \end{cases}
\end{align*}

\begin{proposition}
    For each $k \in \N$ the map $\mathrm{pr}_{\mathfrak C_j}$ and the sets
    \begin{align*}
        (k \Delta_{\mathfrak C_j})(n) = \set{ \tfrac{1}{n} \underline a \mid \underline a \in \Gamma_{\mathfrak C_j, nk}}
    \end{align*}
    form an integral structure on the scaled polytope $k \Delta_{\mathfrak C_j} \subseteq \R^{\mathfrak C_j}$.
\end{proposition}
\begin{proof}
    The lattice $\mathcal L^{\mathfrak C_j}$ is graded by $\Z \cong \Z e_{I_j}$. Analogous to the proof of Proposition~\ref{prop:global_rational_structure} any element $\underline b \in \mathcal L^{\mathfrak C_j}$ of degree $1$ defines a linear map $\overline{\mathrm{pr}}_{\mathfrak C}: \R^{\mathfrak C} \to U_0$ sending an element $\underline a \in \mathcal L^{\mathfrak C_j}$ of degree $d \in \Z$ to $\underline a - d \underline b$, where $U_0$ is the linear span of the lattice $\mathcal L^{\mathfrak C_j}_0$ of degree zero elements in $\mathcal L^{\mathfrak C_j}$. The restriction of this map to any affine subspace
    \begin{align*}
        U_d = \set{ x \in \R^{\mathfrak C_j} \mid \deg x = d}
    \end{align*}
    for $d \in \Z$ is bijective and identifies $d \underline b + \mathcal L^{\mathfrak C_j}_0 = \mathcal L^{\mathfrak C_j} \cap U_d$ with $\mathcal L^{\mathfrak C_j}_0$. For every $n \in \N$ the polytope $kn \Delta_{\mathfrak C_j}$ is contained in $U_{kn}$, hence $\overline{\mathrm{pr}}_{\mathfrak C_j}$ maps the subset $(k \Delta_{\mathfrak C_j})(1)$ onto $\overline{\mathrm{pr}}_{\mathfrak C_j}(k \Delta_{\mathfrak C_j}) \cap \mathcal L^{\mathfrak C_j}_0$. It follows
    \begin{align*}
        \overline{\mathrm{pr}}_{\mathfrak C_j} \big( (k \Delta_{\mathfrak C_j})(n) \big) &= \overline{\mathrm{pr}}_{\mathfrak C_j} \big( \tfrac{1}{n} (kn \Delta_{\mathfrak C_j})(1) \big) = \tfrac{1}{n} \big( \overline{\mathrm{pr}}_{\mathfrak C_j}(kn \Delta_{\mathfrak C_j}) \cap \mathcal L^{\mathfrak C_j}_0 \big) \\
        &= \overline{\mathrm{pr}}_{\mathfrak C_j}(k \Delta_{\mathfrak C_j}) \cap \tfrac{1}{n} \mathcal L^{\mathfrak C_j}_0.
    \end{align*}
    As all vertices of $k \Delta_{\mathfrak C_j}$ are contained in the lattice $\mathcal L^{\mathfrak C_j}$, the map $\overline{\mathrm{pr}}_{\mathfrak C_j}$ defines an integral structure on this polytope.
    
    For our purposes we choose $\underline b = e_{p_0}$, so that the composition of $\overline{\mathrm{pr}}_{\mathfrak C_j}$ with
    \begin{align*}
        \psi: U_0 \hookrightarrow \R^{\mathfrak C_j} \xrightarrow{\mathrm{pr}_{\mathfrak C_j}} \R^s
    \end{align*}
    coincides with $\mathrm{pr}_{\mathfrak C_j}$. By the definition of $\mathrm{pr}_{\mathfrak C_j}$ and the defining conditions of an LS-lattice, the map $\psi$ restricts to a group homomorphism $\overline\psi: \mathcal L^{\mathfrak C_j}_0 \to \Z^s$. To finish the proof, we need to show that $\overline\psi$ is an isomorphism. Its image is equal to the set of all elements $\mathrm{pr}_{\mathfrak C_j}(\underline a)$ with $\underline a \in \mathcal L^{\mathfrak C_j}$. The lattice $\mathcal L^{\mathfrak C_j}$ contains the elements
    \begin{align*}
        \underline a^{(i)} = \frac{1}{b_{i, i-1}} e_{p_i} - \frac{1}{b_{i, i-1}} e_{p_{i-1}}
    \end{align*}
    for $i = 1, \dots, s$. The image of $\underline a^{(i)}$ under the map $\overline\psi$ is of the form $e_i + \sum_{k=1}^{i-1} \Z e_k$, so these images form a basis of $\Z^s$ and $\overline\psi$ is surjective. This also implies that its kernel has rank zero.
\end{proof}

The proposition immediately has the consequence that the product map 
\begin{align*}
    \mathrm{pr}_{\mathfrak C} = \mathrm{pr}_{\mathfrak C_1} \times \dots \times \mathrm{pr}_{\mathfrak C_m}: \R^{\mathfrak C} \to \R^r
\end{align*}
forms an integral structure on the polytope $\Delta_{\mathfrak C}^{(\underline d)}$ for each $\underline d \in \operatorname{relint} \sigma_{\mathfrak C}$ together with the subsets
\begin{align*}
    \Delta_{\mathfrak C}^{(\underline d)}(n) = \prod_{j=1}^m \, \set{ \tfrac{1}{n} \underline a \mid \underline a \in \Gamma_{\mathfrak C_j, n \phi_j^{\mathfrak C}(\underline d) }} = \set{ \tfrac{1}{n} \underline a \mid \underline a \in \Gamma_{\mathfrak C, n \underline d}}.
\end{align*}
Since $\widetilde{\Gamma}_{\mathfrak C}^{(\underline d)} = \mathcal L^{\mathfrak C, (\underline d)} \cap \operatorname{Cone} \Gamma_{\mathfrak C} = \Gamma_{\mathfrak C}^{(\underline d)}$ follows from equation~(\ref{eq:LS_type_veronese_lattice}), the Veronese monoid $\Gamma_{\mathfrak C}^{(\underline d)}$ is saturated. Therefore the map $\mathrm{pr}_{\mathfrak C}$ meets the requirements from Proposition~\ref{prop:global_rational_structure}. 

The volume of $\mathrm{pr}_{\mathfrak C_j}(\Delta_{\mathfrak C_j})$ is equal to the product $\prod_{k=1}^s b_{k, k-1}$ of all bonds in the subchain $\mathfrak C_j$ divided by the factorial of $\vert \mathfrak C_j \vert - 1$. Let
\begin{align*}
    b_{\mathfrak C} = \prod_{i=1}^r b_{i, i-1}
\end{align*}
denote the product of all bonds in $\mathfrak C$. Since all bonds connecting the chains $\mathfrak C_j$ are equal to $1$, we get
\begin{align}
    \label{eq:multisimplex_volume}
    \mathrm{vol} \big( \mathrm{pr}_{\mathfrak C}(\Delta_{\mathfrak C}^{(\underline d)}) \big) = b_{\mathfrak C} \cdot \prod_{j=1}^m \frac{\phi_j^{\mathfrak C}(\underline d)^{\vert \mathfrak C_j \vert - 1}}{(\vert \mathfrak C_j \vert - 1)!}
\end{align}
for every $\underline d \in \operatorname{relint} \sigma_{\mathfrak C}$. As the polytope $\Delta_{\mathfrak C}^{(\underline d)}$ is empty for $\underline d \notin \sigma_{\mathfrak C}$, one needs to take care which maximal chains to consider when computing the leading term $G_R$ of the Hilbert polynomial via Theorem~\ref{thm:G_R_formula}.

The coefficients of $G_R$ contain the multidegrees of $X$. One can compute them explicitly in the case when the poset $\mathcal I$ is totally ordered. \WWlog{} we can then rearrange the numbering of the projective spaces $\PP(V_i)$ such that the following situation applies.

\begin{corollary}
    \label{cor:multidegree_formula}
    Suppose that the poset $\mathcal I$ consists only of the sets $[i]$ for $i \in [m]$. Then the multidegree of the variety $X \subseteq \prod_{i=1}^m \PP(V_i)$ to a tuple $\underline k \in \N_0^m$ with $k_1 + \dots + k_m = \dim X$ is given by
    \begin{align*}
        \deg_{\underline k}(X) = \sum_{\mathfrak C} \mkern2mu b_{\mathfrak C} \mkern1mu,
    \end{align*}
    where the sum runs over all maximal chains $\mathfrak C$ in $A$, which contain exactly $k_i + 1$ elements from $A_i = \set{p \in A \mid I_p = [i]}$ for each $i = 1, \dots, m$ and $b_{\mathfrak C}$ denotes the product of all bonds in $\mathfrak C$.
\end{corollary}
\begin{proof}
    For any maximal chain $\mathfrak C$ in $A$ the matrix $M_{\mathfrak C}$ we defined earlier is the identity matrix and $\sigma_{\mathfrak C}$ coincides with the positive orthant $\R_{\geq 0}^m$. For all $\underline d \in \N_0^m$ we have $\phi_j^{\mathfrak C}(\underline d) = d_j$ for all $\underline d \in \N_0^m$. Using equation~(\ref{eq:multisimplex_volume}) we therefore get
    \begin{align*}
        G_R(\underline d) = \frac{1}{k_1! \cdots k_m!} \sum_{\mathfrak C} b_{\mathfrak C} \mkern3mu d_1^{\mkern2mu \vert \mathfrak C_1 \vert - 1} \mkern-2mu \cdots \mkern2mu d_m^{\mkern2mu \vert \mathfrak C_m \vert - 1}.
    \end{align*}
    This implies the claimed formula, since the coefficient of the monomial $d_1^{k_1} \cdots d_m^{k_m}$ is equal to the multidegree $\deg_{\underline k}(X)$ divided by $k_1! \cdots k_m!$.
\end{proof}

\section{Multiprojective stratifications on flag varieties in type \texttt{A}}
\label{sec:type_A}

In this section we construct a multiprojective Seshadri stratification on every (partial) flag variety $G/Q$ in Dynkin type \texttt{A}. This strati\-fication is normal and balanced and the resulting standard monomial theory (as of Proposition~\ref{prop:smt}) is the classical Hodge-Young theory (see~\cite{hodge1943some} and \cite{hodge1994methods}) of products of Pl\"ucker coordinates indexed by semistandard Young-tableaux.

We fix the simple algebraic group $G = \mathrm{SL}_n(\K)$ over an algebraically closed field $\K$ of characteristic zero, the torus $T$ of diagonal matrices in $G$ and the Borel subgroup $B$ of all upper triangular matrices with determinant $1$, which contains $T$. The Weyl group $W = N_G(T)/C_G(T)$ can be identified with the symmetric group $S_n$, since $C_G(T) = T$ and the normalizer of $T$ consists of the matrices which have exactly one non-zero entry in every row and each column. Let $\varepsilon_i: T \to \K^\times$ be the character of $T$, where $\varepsilon_i(t)$ is equal to the $i$-th entry on the diagonal of $t \in T$. The root system $\Phi$ of $G$ is given by all characters $\varepsilon_i - \varepsilon_j$ for $i \neq j$ in $[n]$ and the choice of the Borel subgroup corresponds to the set $\Phi^+$ of positive roots and the set $\Delta$ of the simple roots $\alpha_i = \varepsilon_i - \varepsilon_{i+1}$ for $i \in [n-1]$. Let $\Lambda$ denote the weight lattice of the root system $\Phi$ and $\Lambda^+$ be the monoid of all dominant weights. To each $i \in [n-1]$ there is the associated fundamental weight $\omega_i = \varepsilon_1 + \dots + \varepsilon_i \in \Lambda^+$ and the maximal parabolic subgroup $P_i = B W_{P_i} B$, where $W_{P_i} \subseteq W$ is the stabilizer of $\omega_i$. It is generated by the simple reflections $s_\alpha$ for $\alpha \in \Delta \setminus \set{\alpha_i}$.

Every weight $\lambda \in \Lambda$ can be uniquely written in the form $\lambda = c_1 \varepsilon_1 + \dots + c_{n-1} \varepsilon_{n-1}$ with coefficients $c_i \in \Z$. Then $\lambda$ is a dominant weight, if and only if $c_1 \geq \dots \geq c_{n-1} \geq 0$. In this way, each dominant weight $\lambda$ corresponds to a partition $p(\lambda) = (c_1, \dots, c_{n-1})$ of $n-1$ parts (which are potentially zero). The partition $p(\lambda)$ is usually visualized via a Young diagram (we use the English notation) having exactly $c_i$ boxes in its $i$-th row. For each $i = 1, \dots, n-1$ it contains exactly $\langle \lambda, \alpha_i^\vee \rangle$ columns of length $i$.

\begin{definition}
    \label{def:YT}
    For each $\lambda \in \Lambda^+$ and its corresponding partition $p(\lambda)$ we define:
    \begin{enumerate}[label=(\alph{enumi})]
        \item The set $\mathrm{YT}(\lambda)$ of all Young-tableaux of shape $p(\lambda)$ with entries in $[n]$;
        \item The subset $\mathrm{SSYT}(\lambda) \subseteq \mathrm{YT}(\lambda)$ of all \textit{semistandard} Young-tableaux $T \in \mathrm{YT}(\lambda)$, \ie the entries of $T$ increase weakly along each row (from left to right) and strictly along each column (from top to bottom).
    \end{enumerate}
\end{definition}

\subsection{Maximal parabolic subgroups}

The Grassmann varieties $G/P_{i} \cong \mathrm{Gr}_i(\K^n)$ for $i = 1, \dots, n-1$ can be embedded into the projectivized fundamental representation $\PP(V(\omega_{i})) \cong \PP(\bigwedge^{\! i} \K^n)$ via the usual Pl\"ucker embedding: 
\begin{align*}
    G/P_{i} \hookrightarrow \PP(\hbox{$\bigwedge\nolimits^{i} \K^n$}), \quad g P_i \longmapsto [g \cdot (e_1 \wedge \dots \wedge e_i)].
\end{align*}
This representation is minuscule, that is to say the Weyl group acts transitively on the set of its weights. Hence all weight spaces are one-dimensional and the weights are in bijection to the elements of the Bruhat poset $W/W_{P_i} \cong S_n / (S_i \times S_{n-i})$. They correspond to subsets $J \subseteq[n]$ of size $i$ and can also be identified with semistandard Young-tableaux in $\mathrm{SSYT}(\omega_i)$. The weight space in $V(\omega_{i})$ of weight $\theta \in W/W_{P_{i}}$ is generated by the vector $e_\theta = e_{j_1} \wedge \dots \wedge e_{j_i} \in \bigwedge^{\! i} \K^n$, where $j_1 < \dots < j_i$ are the elements of the subset $J \subseteq [n]$ corresponding to $\theta$. The dual basis vectors $p_\theta \in V(\omega_{i})^*$ are known as \textit{Pl\"ucker coordinates}. It is well known that Pl\"ucker coordinates fulfill the conditions~\ref{itm:seshadri_strat_b} and \ref{itm:seshadri_strat_c} on a Seshadri stratification (see \cite[1.2.10, 1.4.11]{seshadri2016introduction}). In fact, the Grassmann varieties were one of the motivating examples for the development of Seshadri stratifications (\cite{littelmann2017grassmannian}).

\begin{proposition}
    \label{prop:strat_grassmannian}
    There exists a Seshadri stratification on $G/P_i$ with underlying poset $W/W_{P_i}$, where the strata $X_\theta$ are given by the Schubert varieties in $G/P_i$ and the extremal functions $f_\theta = p_\theta$ by Pl\"ucker coordinates.
\end{proposition}

\subsection{Weyl group cosets}

To construct multiprojective stratifications, we need to work with different parabolic subgroups at the same time. We therefore introduce some notation. All statements in this subsection are well known and can be found in any classical text book about Coxeter groups, for example in~\cite{bjorner2006combinatorics}.

To every inclusion $Q \subseteq Q'$ of two parabolic subgroups in $G$ containing $B$ there is the monotone surjection\label{eq:def_pi_Q}
\begin{align*}
    \pi_{Q,Q'} \!: \: W/W_Q \twoheadrightarrow W/W_{Q'}, \ \sigma W_Q \mapsto \sigma W_{Q'},
\end{align*}
where we typically write $\pi_{Q'}$ instead, if the source is clear. Every element $\theta \in W/W_{Q'}$ has a unique minimal preimage $\min_Q(\theta)$ and a unique maximal preimage $\max_Q(\theta)$ in $W/W_Q$ under $\pi_{Q'}$. The corresponding two maps\label{eq:def_min_Q_max_Q}
\begin{align*}
    \mathrm{min}_Q \!: W/W_{Q'} &\hookrightarrow W/W_{Q}, \ \theta \mapsto \mathrm{min}_Q(\theta) \quad \text{and}\\
    \mathrm{max}_Q \!: W/W_{Q'} &\hookrightarrow W/W_{Q}, \ \theta \mapsto \mathrm{max}_Q(\theta),
\end{align*}
are isomorphisms of posets onto their image. 

\begin{definition}
    We say an element $\sigma W_Q \in W/W_Q$ is \textbf{\boldmath{}$Q'$-minimal}/\textbf{\boldmath{}$Q'$-maximal}, if it lies in the image of $\min_Q$ or $\max_Q$ respectively. 
\end{definition}

The Weyl group of $G = \mathrm{SL}_n(\K)$ is isomorphic to the symmetric group $S_n$. We usually write a permutation $\sigma: [n] \to [n]$ in $S_n$ in the one-line notation $\sigma = \sigma(1) \cdots \sigma(n)$. For $i = 1, \dots, n-1$ we write an element $\sigma W_{P_i}$ of $W/W_{P_i} \cong S_n / (S_i \times S_{n-i})$ in the form $\sigma(1) \cdots \sigma(i)$.

The Bruhat order on $W/W_{P_i}$ can be characterized via this one-line notation: For a tuple $\underline j = (j_1, \dots, j_i)$ of natural numbers, we write $\underline j^\leq$ for the permuted tuple with weakly increasing entries (from left to right). For all $\phi = \phi(1) \cdots \phi(i), \theta = \theta(1) \cdots \theta(i) \in W/W_{P_i}$ we then have
\begin{align*}
    \phi \leq \theta \quad \Longleftrightarrow \quad (\phi(1), \dots, \phi(i))^\leq \leq (\theta(1), \dots, \theta(i))^\leq,
\end{align*}
where the tuples on the right hand side are compared component-wise. The one-line notation can also be used to describe the Bruhat order on $W$, as $\sigma \leq \tau \in W$ is equivalent to $\pi_{P_i}(\sigma) \leq \pi_{P_i}(\tau)$ for all $i = 1, \dots, n-1$.

The one-line notation is very advantageous in conjunction with the lifting maps $\min_Q$ and $\max_Q$. For example, let $Q = P_1 \cap P_3$ for $n = 4$ and $\theta = 134 \in W/W_{P_3}$ (we omit the brackets of the tuple notation), then the unique maximal lift $\max_Q(\theta)$ in $W/W_Q$ is given by the projection of $\max_B(\theta) \in W$ to $W/W_Q$ and the element $\max_B(\theta)$ is equal to $4312$. More precisely, the image of all $P_3$-maximal elements in $W/W_Q$ under the map $\min_B: W/W_Q \to W$ are all permutations $\sigma = \sigma(1) \cdots \sigma(4)$ with $\sigma(1) > \sigma(2)$, $\sigma(1) > \sigma(3)$ and $\sigma(2) < \sigma(3)$.

\subsection{The underlying poset of the stratification}

We now go over to arbitrary parabolic subgroups by fixing the partial flag variety $X = G/Q$ to a parabolic subgroup
\begin{align*}
    Q = P_{k_1} \cap \dots \cap P_{k_m}
\end{align*}
with strictly ascending indices $1 \leq k_1 < \dots < k_m \leq n - 1$. Every parabolic subgroup containing $B$ can be uniquely written in this way. To reduce the number of indices, we write $\pi_i$ instead of $\pi_{P_{k_i}}$ and $W_i$ instead of $W_{P_{k_i}}$. Let $R$ denote the multihomogeneous coordinate ring of $G/Q$ with respect to the Pl\"ucker embedding
\begin{align}
    \label{eq:G_Q_pluecker_embedding}
    G/Q \hookrightarrow \prod_{i=1}^m G/P_{k_i} \hookrightarrow \prod_{i=1}^m \PP(V(\omega_{k_i})).
\end{align}
It is well known that this ring $R$ contains a lot of information about the representation theory of $G$. It carries the structure of a $G$-representation and the graded component $R_{\underline d} \subseteq R$ of degree $\underline d \in \N_0^m$ is isomorphic to the dual representation $V(\mu)^*$ to the dominant weight $\mu = d_1 \omega_{k_1} + \dots + d_m \omega_{k_m} \in \Lambda^+$. 

We view the Pl\"ucker coordinates in $V(\omega_{k_i})^*$ as elements of $R$ via the pullback along the projection $\prod_{j=1}^m V(\omega_{k_j}) \twoheadrightarrow V(\omega_{k_i})$. For each element $(\theta, i)$ in the disjoint union $\ulW = \coprod_{i=1}^m W/W_i \times \set{i}$ we therefore have an associated Pl\"ucker coordinate $p_{(\theta, i)} \in R$ and $R$ is clearly generated by these functions as a $\K$-algebra. Hence the monomials/products of Pl\"ucker coordinates form a generating system of $R$ as a vector space. Each of these monomials is either called \textit{standard} or \textit{non-standard} and the set of standard monomials is a basis of $R$. Using only combinatorial methods one can determine whether a monomial is standard. This is typically called a \textit{standard monomial theory}. For $G/Q$ it was shown in~\cite[Chapter 2]{seshadri2016introduction} that this basis is given by semistandard Young-tableaux in the following sense. Let $p_{\underline\theta} = p_{(\theta_1, i_1)} \cdots p_{(\theta_\ell, i_\ell)}$ be a product of Pl\"ucker coordinates with $i_1 \geq \dots \geq i_\ell$. Since each element $\theta_j \in W/W_{i_j}$ can be interpreted as a tableau in $\mathrm{SSYT}(\omega_{i_j})$, the product $p_{\underline\theta}$ corresponds to the Young-tableau
\begin{align*}
    (\theta_1, \dots, \theta_\ell) \in \mathrm{YT}(\omega_{i_1} + \dots + \omega_{i_\ell}),
\end{align*}
such that its $j$-th column is given by $\theta_j$.

\begin{theorem}[{\cite[Proposition 2.3.1, Theorem 2.6.1]{seshadri2016introduction}}]
    \label{thm:smt_basis_G_Q}
    The multihomogeneous coordinate ring $R = \K[G/Q]$ has a basis consisting of the products $p_{(\theta_1, i_1)} \cdots p_{(\theta_\ell, i_\ell)}$ of Pl\"ucker coordinates with $i_1 \geq \dots \geq i_\ell$, such that the corresponding Young-tableau $(\theta_1, \dots, \theta_\ell)$ is semistandard.
\end{theorem}

It is our goal to construct a multiprojective stratification on $G/Q$, such that the associated fan of monoids is in bijection to semistandard tableaux with columns in $\ulW$. In order to construct such a stratification, we first need a suitable candidate for the underlying poset. Notice that the entries of every Young-tableau, which only contains columns from the set $\ulW$, are already strictly increasing along each column, by definition. Therefore semistandardness can be seen as a local property: Such a tableau is semistandard, if and only if every two consecutive columns are semistandard (as a tableau of just $2$ columns). This induces a partial order on the set $\ulW$.

\begin{definition}
    \label{def:ulWlambda_type_A}
    We define a relation $\geq$ on the set $\ulW = \coprod_{i=1}^m W/W_i \times \set{i}$ via
    \begin{equation}
    \label{eq:def_ulW_type_A}
        (\theta, i) \geq (\phi, j) \quad :\Longleftrightarrow \quad i \leq j \ \ \text{and} \ \ \mathrm{max}_Q(\theta) \geq \mathrm{min}_Q(\phi).
    \end{equation}
    for all $(\theta, i), (\phi, j) \in \ulW$.
\end{definition}

With the interpretation of the elements of $\ulW$ as Young-tableaux, we show in Corollary~\ref{cor:partial_order_via_tableaux} that the relation $\geq$ can be written as
\begin{align}
    \label{eq:relation_via_tableaux}
    \vcenter{\hbox{\begin{ytableau} a_1 \\ \vdots \\ a_{k_i} \end{ytableau}}} \, \geq \, \vcenter{\hbox{\begin{ytableau} b_1 \\ \vdots \\ \vdots \\ b_{k_j} \end{ytableau}}} \quad \Longleftrightarrow \quad i \leq j \quad \text{and} \quad \vcenter{\hbox{\begin{ytableau} b_1 & a_1 \\ \vdots & \vdots \\ \vdots & a_{k_i} \\ b_{k_j} \end{ytableau}}} \quad \text{is semistandard}.
\end{align} 
In particular, this implies that the relation $\geq$ is a partial order. However, as we do not show this characterization of the relation right now, we carefully avoid using the transitivity of $\geq$ (the reflexivity and antisymmetry are immediate from the definition).

\ytableausetup{smalltableaux}
\begin{figure}
\centering
\begin{subfigure}{.5\textwidth}
    \begin{center}
        \begin{tikzpicture}[scale=0.65]
        \node (p3) at (0,0) {$\begin{ytableau} 3 \end{ytableau}$};
        \node (p2) at (0,-1.5)  {$\begin{ytableau} 2 \end{ytableau}$};
        \node (p1) at (-1.5,-3) {$\begin{ytableau} 1 \end{ytableau}$};
        \node (p23) at (1.5,-3) {$\begin{ytableau} 2 \\ 3 \end{ytableau}$};
        \node (p13) at (0,-4.5) {$\begin{ytableau} 1 \\ 3 \end{ytableau}$};
        \node (p12) at (0,-6.5) {$\begin{ytableau} 1 \\ 2 \end{ytableau}$};
        
        \draw [ thick, shorten <=-2pt, shorten >=-2pt, -stealth] (p3) -- (p2);
        \draw [thick, shorten <=-2pt, shorten >=-2pt, -stealth] (p2) -- (p1);
        \draw [thick, shorten <=-2pt, shorten >=-2pt, -stealth] (p2) -- (p23);
        \draw [thick, shorten <=-2pt, shorten >=-2pt, -stealth] (p1) -- (p13);
        \draw [thick, shorten <=-2pt, shorten >=-2pt, -stealth] (p23) -- (p13);
        \draw [thick, shorten <=-2pt, shorten >=-2pt, -stealth] (p13) -- (p12);
        \end{tikzpicture}
    \end{center}
    \caption{Type $\texttt{A}_2$}
    \label{fig:hasse_a_2}
\end{subfigure}%
\begin{subfigure}{.5\textwidth}
    \begin{center}
        \begin{tikzpicture}[scale=0.70]
        \node (p4) at (0,0) {$\begin{ytableau} 4 \end{ytableau}$};
        \node (p3) at (1.5,0) {$\begin{ytableau} 3 \end{ytableau}$};
        \node (p2) at (3,0)  {$\begin{ytableau} 2 \end{ytableau}$};
        \node (p1) at (4.5,0) {$\begin{ytableau} 1 \end{ytableau}$};
        \node (p34) at (1.5,-1.75) {$\begin{ytableau} 3 \\ 4 \end{ytableau}$};
        \node (p24) at (3,-1.75) {$\begin{ytableau} 2 \\ 4 \end{ytableau}$};
        \node (p14) at (4.5,-1.75) {$\begin{ytableau} 1 \\ 4 \end{ytableau}$};
        \node (p23) at (3,-3.8) {$\begin{ytableau} 2 \\ 3 \end{ytableau}$};
        \node (p13) at (4.5,-3.8) {$\begin{ytableau} 1 \\ 3 \end{ytableau}$};
        \node (p12) at (6,-3.8) {$\begin{ytableau} 1 \\ 2 \end{ytableau}$};
        \node (p234) at (3,-6.1) {$\begin{ytableau} 2 \\ 3 \\ 4 \end{ytableau}$};
        \node (p134) at (4.5,-6.1) {$\begin{ytableau} 1 \\ 3 \\ 4 \end{ytableau}$};
        \node (p124) at (6,-6.1) {$\begin{ytableau} 1 \\ 2 \\ 4 \end{ytableau}$};
        \node (p123) at (7.5,-6.1) {$\begin{ytableau} 1 \\ 2 \\ 3 \end{ytableau}$};
        
        \draw [thick, shorten <=-2pt, shorten >=-2pt, -stealth] (p4) -- (p3);
        \draw [thick, shorten <=-2pt, shorten >=-2pt, -stealth] (p3) -- (p2);
        \draw [thick, shorten <=-2pt, shorten >=-2pt, -stealth] (p2) -- (p1);
        \draw [thick, shorten <=-2pt, shorten >=-2pt, -stealth] (p3) -- (p34);
        \draw [thick, shorten <=-2pt, shorten >=-2pt, -stealth] (p2) -- (p24);
        \draw [thick, shorten <=-2pt, shorten >=-2pt, -stealth] (p1) -- (p14);
        \draw [thick, shorten <=-2pt, shorten >=-2pt, -stealth] (p34) -- (p24);
        \draw [thick, shorten <=-2pt, shorten >=-2pt, -stealth] (p24) -- (p23);
        \draw [thick, shorten <=-2pt, shorten >=-2pt, -stealth] (p24) -- (p14);
        \draw [thick, shorten <=-2pt, shorten >=-2pt, -stealth] (p23) -- (p13);
        \draw [thick, shorten <=-2pt, shorten >=-2pt, -stealth] (p14) -- (p13);
        \draw [thick, shorten <=-2pt, shorten >=-2pt, -stealth] (p13) -- (p12);
        \draw [thick, shorten <=-2pt, shorten >=-2pt, -stealth] (p23) -- (p234);
        \draw [thick, shorten <=-2pt, shorten >=-2pt, -stealth] (p13) -- (p134);
        \draw [thick, shorten <=-2pt, shorten >=-2pt, -stealth] (p12) -- (p124);
        \draw [thick, shorten <=-2pt, shorten >=-2pt, -stealth] (p234) -- (p134);
        \draw [thick, shorten <=-2pt, shorten >=-2pt, -stealth] (p134) -- (p124);
        \draw [thick, shorten <=-2pt, shorten >=-2pt, -stealth] (p124) -- (p123);
        \end{tikzpicture}
    \end{center}
    \caption{Type $\texttt{A}_3$}
    \label{fig:hasse_a_3}
\end{subfigure}
\caption{Hasse-diagrams of $\ulW$ for $Q = B$}
\label{fig:hasse_ulWlambda_type_A}
\end{figure}
\ytableausetup{nosmalltableaux}

\subsection{The stratification on a partial flag variety}

Using the underlying poset $\ulW$ we now define a multiprojective Seshadri stratification of the partial flag variety $X = G/Q$ with respect to the Pl\"ucker embedding~(\ref{eq:G_Q_pluecker_embedding}). Let $I_{(\theta, i)} = \set{i, \dots, m} \subseteq [m]$ be the index set of an element $(\theta, i) \in \underline W$. Then the projection of $G/Q$ via $\pi_{\set{i, \dots, m}}$ can be interpreted as the flag variety $G/Q_i$ to the parabolic subgroup
\begin{align*}
    Q_i \coloneqq P_{k_i} \cap \dots \cap P_{k_m},
\end{align*}
since the following diagram commutes:
\begin{equation*}
    \begin{tikzcd}
    G/Q \arrow[d, two heads] \arrow[r, hook] & \prod_{j=1}^m G/P_{k_j} \arrow[r, hook] \arrow[d, two heads] & \prod_{j=1}^m \PP(V(\omega_{k_j})) \arrow[d, two heads] \\
    G/Q_i \arrow[r, hook] & \prod_{j=i}^m G/P_{k_j} \arrow[r, hook] & \prod_{j=i}^m \PP(V(\omega_{k_j}))
    \end{tikzcd}
\end{equation*}
Furthermore, to each element $(\theta, i) \in \underline W$ we associate:
\begin{itemize}
    \item The Schubert variety $X_{\max_{Q_i}(\theta)} \subseteq G/Q_i$ as the stratum $X_{(\theta, i)}$,
    \item the extremal function $f_{(\theta, i)} = p_{(\theta, i)}$. 
\end{itemize}
Note that, if $Q$ is a maximal parabolic subgroup, then $G/Q$ is a Grassmann variety and we already know that these definitions give rise to a Seshadri stratification, namely the stratification from Proposition~\ref{prop:strat_grassmannian} of all Schubert varieties and Pl\"ucker coordinates.

We write $\hat X_{(\theta, i)}$ for the multicone of $X_{\max_{Q_i}(\theta)}$, viewed as a subvariety of $\prod_{i=1}^m V(\omega_{k_i})$. It coincides with the intersection
\begin{align*}
    \hat X_{\widetilde\theta}^{(i)} \coloneqq \hat X_{\widetilde\theta} \cap \set{(v_1, \dots, v_m) \in \prod_{j=1}^m V(\omega_{k_j}) \mid v_1 = \dots = v_{i-1} = 0}, 
\end{align*}
where $\hat X_{\widetilde\theta}$ is the multicone of the Schubert variety $X_{\widetilde\theta} \subseteq G/Q$ to the element $\widetilde\theta = \min_Q \circ \max_{Q_i}(\theta)$. 

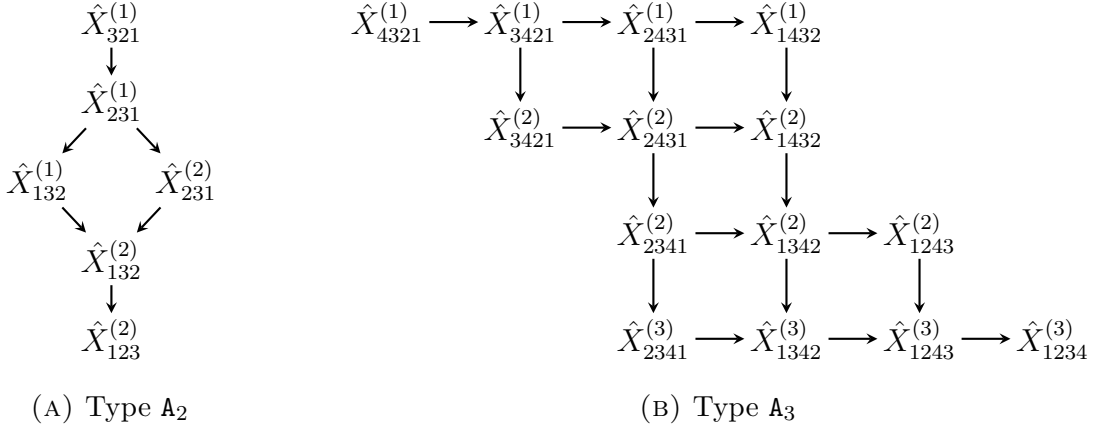
\begin{figure}
\centering
\begin{subfigure}{.3\textwidth}
    \begin{center}
        \begin{tikzpicture}[scale=0.7]
        \node (p3) at (0,0) {$\hat X_{321}^{(1)}$};
        \node (p2) at (0,-1.5)  {$\hat X_{231}^{(1)}$};
        \node (p1) at (-1.4,-3) {$\hat X_{132}^{(1)}$};
        \node (p23) at (1.4,-3) {$\hat X_{231}^{(2)}$};
        \node (p13) at (0,-4.5) {$\hat X_{132}^{(2)}$};
        \node (p12) at (0,-6) {$\hat X_{123}^{(2)}$};
        
        \draw [thick, shorten <=-2pt, shorten >=-2pt, -stealth] (p3) -- (p2);
        \draw [thick, shorten <=-2pt, shorten >=-2pt, -stealth] (p2) -- (p1);
        \draw [thick, shorten <=-2pt, shorten >=-2pt, -stealth] (p2) -- (p23);
        \draw [thick, shorten <=-2pt, shorten >=-2pt, -stealth] (p1) -- (p13);
        \draw [thick, shorten <=-2pt, shorten >=-2pt, -stealth] (p23) -- (p13);
        \draw [thick, shorten <=-2pt, shorten >=-2pt, -stealth] (p13) -- (p12);
        \end{tikzpicture}
    \end{center}
    \caption{Type $\texttt{A}_2$}
    \label{fig:strat_a_2}
\end{subfigure}%
\begin{subfigure}{.7\textwidth}
    \begin{center}
        \begin{tikzpicture}[scale=0.7]
        \node (p4) at (0,0) {$\hat X_{4321}^{(1)}$};
        \node (p3) at (2.5,0) {$\hat X_{3421}^{(1)}$};
        \node (p2) at (5,0)  {$\hat X_{2431}^{(1)}$};
        \node (p1) at (7.5,0) {$\hat X_{1432}^{(1)}$};
        \node (p34) at (2.5,-2) {$\hat X_{3421}^{(2)}$};
        \node (p24) at (5,-2) {$\hat X_{2431}^{(2)}$};
        \node (p14) at (7.5,-2) {$\hat X_{1432}^{(2)}$};
        \node (p23) at (5,-4) {$\hat X_{2341}^{(2)}$};
        \node (p13) at (7.5,-4) {$\hat X_{1342}^{(2)}$};
        \node (p12) at (10,-4) {$\hat X_{1243}^{(2)}$};
        \node (p234) at (5,-6) {$\hat X_{2341}^{(3)}$};
        \node (p134) at (7.5,-6) {$\hat X_{1342}^{(3)}$};
        \node (p124) at (10,-6) {$\hat X_{1243}^{(3)}$};
        \node (p123) at (12.5,-6) {$\hat X_{1234}^{(3)}$};
        
        \draw [thick, shorten <=-2pt, shorten >=-2pt, -stealth] (p4) -- (p3);
        \draw [thick, shorten <=-2pt, shorten >=-2pt, -stealth] (p3) -- (p2);
        \draw [thick, shorten <=-2pt, shorten >=-2pt, -stealth] (p2) -- (p1);
        \draw [thick, shorten <=-2pt, shorten >=-2pt, -stealth] (p3) -- (p34);
        \draw [thick, shorten <=-2pt, shorten >=-2pt, -stealth] (p2) -- (p24);
        \draw [thick, shorten <=-2pt, shorten >=-2pt, -stealth] (p1) -- (p14);
        \draw [thick, shorten <=-2pt, shorten >=-2pt, -stealth] (p34) -- (p24);
        \draw [thick, shorten <=-2pt, shorten >=-2pt, -stealth] (p24) -- (p23);
        \draw [thick, shorten <=-2pt, shorten >=-2pt, -stealth] (p24) -- (p14);
        \draw [thick, shorten <=-2pt, shorten >=-2pt, -stealth] (p23) -- (p13);
        \draw [thick, shorten <=-2pt, shorten >=-2pt, -stealth] (p14) -- (p13);
        \draw [thick, shorten <=-2pt, shorten >=-2pt, -stealth] (p13) -- (p12);
        \draw [thick, shorten <=-2pt, shorten >=-2pt, -stealth] (p23) -- (p234);
        \draw [thick, shorten <=-2pt, shorten >=-2pt, -stealth] (p13) -- (p134);
        \draw [thick, shorten <=-2pt, shorten >=-2pt, -stealth] (p12) -- (p124);
        \draw [thick, shorten <=-2pt, shorten >=-2pt, -stealth] (p234) -- (p134);
        \draw [thick, shorten <=-2pt, shorten >=-2pt, -stealth] (p134) -- (p124);
        \draw [thick, shorten <=-2pt, shorten >=-2pt, -stealth] (p124) -- (p123);
        \end{tikzpicture}
    \end{center}
    \caption{Type $\texttt{A}_3$}
    \label{fig:strat_a_3}
\end{subfigure}
\caption{Stratifications of $G/B$}
\label{fig:strat_type_A}
\end{figure}
    
\begin{theorem}
    \label{thm:stratification_type_A}
    The varieties $X_{(\theta, i)}$ for $(\theta, i) \in \ulW$ together with the extremal functions $f_{(\theta, i)}$ form a Seshadri stratification on $X = G/Q \hookrightarrow \prod_{j=1}^m \PP(V(\omega_{k_j}))$.
\end{theorem}

Before we are able to prove this theorem, we need to show, that $\geq$ is indeed a partial order on $\ulW$ and establish a good understanding of this poset and its covering relations. The key ingredient is the following innocent looking lemma. Although it can be shown more abstractly, we stick to a prove using methods of type \texttt{A} for simplicity. A more general statement can be found in \cite[Lemma 12.4]{lakshmibaiGP4}.

\begin{lemma}
\label{lem:key_lem_type_A}
    If $\theta \in W/W_{Q_i}$ is $P_{k_i}$-maximal, then $\pi_{Q_j}(\theta)$ is $P_{k_j}$-maximal for all $j \geq i$. 
\end{lemma}
\begin{proof}
    The $P_{k_i}$-maximality of $\theta$ is equivalent to the $P_{k_i}$-maximality of its maximal representative $\sigma = \max_B(\theta)$ in $W$. We write $\sigma = \sigma(1) \cdots \sigma(n)$ and $\tau \coloneqq \max_B \circ \pi_{Q_j}(\theta) = \tau(1) \cdots \tau(n)$ in one-line notation. The parabolic subgroups $P_{k_1}, \dots, P_{k_m}$ partition the set $[n]$ into $m+1$ subsets $I_s = \set{k_s + 1, \dots, k_{s+1}}$ for $s = 0, \dots, m$, where we set $k_0 = 0$ and $k_{m+1} = n$. Since $Q_j = \bigcup_{s=j}^m P_{k_s}$, the one-line notations of $\sigma$ and $\tau$ agree up to permutation in the blocks $I_0 \cup \dots \cup I_j$ and in each $I_s$ for $s > j$. As $\tau$ is the maximal representative of $\pi_{Q_j}(\theta)$, we have $\tau s_{\alpha_\ell} < \tau$ for all $\ell \in [n] \setminus \set{k_j, \dots, k_m}$, hence the numbers $\tau(r)$ are strictly decreasing in $I_0 \cup \dots \cup I_j$ and in all $I_s$ for $s > j$. Additionally the $P_{k_i}$-maximality of $\sigma$ implies, that the numbers $\sigma(r)$ are strictly decreasing in $I_0 \cup \dots \cup I_i$ and in $I_{i+1} \cup \dots \cup I_m$. By combining these observations, we get $\sigma(r) = \tau(r)$ for all $j+1 \leq r \leq n$ and the numbers $\tau(r)$ are strictly decreasing in $I_{j+1} \cup \dots \cup I_m$. But this just means, that $\tau$ is $P_{k_j}$-maximal.
\end{proof}

\begin{lemma}
\label{lem:poset_prop}
    For all $(\theta, i), (\phi, j) \in \ulW$ the relation $(\theta, i) \geq (\phi, j)$ holds, if and only if $i \leq j$ and any one of the following equivalent statements is fulfilled:
    \begin{enumerate}[label=(\alph{enumi})]
        \item \label{itm:poset_i} $\pi_j \circ \max_{Q_i}(\theta) \geq \phi$;
        \item \label{itm:poset_ii} there exists a parabolic subgroup $Q \subseteq Q' \subseteq P_{k_i} \cap P_{k_j}$ and lifts $\overline\theta$, $\overline\phi \in W/W_{Q'}$ of $\theta$ and $\phi$ respectively, such that $\overline\theta \geq \overline\phi$ in $W/W_{Q'}$;
        \item \label{itm:poset_iii} $\min_Q \circ \max_{Q_i}(\theta) \geq \min_Q \circ \max_{Q_j}(\phi)$ in $W/W_Q$.
    \end{enumerate}
\end{lemma}
\begin{proof}
    By the definition of the partial order on $\ulW$, it suffices to show the following: For any two elements $(\theta, i), (\phi, j) \in \ulW$ with $i \leq j$ the inequality $\max_Q(\theta) \geq \min_Q(\phi)$ is equivalent to each of the three conditions \ref{itm:poset_i}, \ref{itm:poset_ii} and \ref{itm:poset_iii}. We assume the relation $i \leq j$ for the inclusion $Q_j \subseteq Q_i$.
    
    Let $Q'$ be a parabolic subgroup contained in $P_{k_i} \cap P_{k_j}$ containing $Q$. Clearly condition~\ref{itm:poset_ii} is equivalent to $\mathrm{max}_Q(\theta) \geq \mathrm{min}_Q(\phi)$ and \ref{itm:poset_ii} follows from \ref{itm:poset_iii}. Furthermore $\mathrm{max}_Q(\theta) \geq \mathrm{min}_Q(\phi)$ implies \ref{itm:poset_i}, since $\pi_j \circ \max_{Q_i}(\theta) = \pi_j \circ \mathrm{max}_Q(\theta) \geq \pi_j \circ \mathrm{min}_Q(\phi) = \phi$
    
    It remains to show, that \ref{itm:poset_iii} follows from \ref{itm:poset_i}. We write $\widetilde\theta = \min_Q \circ \max_{Q_i}(\theta)$ and $\widetilde\phi = \min_Q \circ \max_{Q_j}(\phi)$. Since both elements are $Q_i$-minimal, it is enough to prove the inequality $\widetilde\theta \geq \widetilde\phi$ in $W/W_{Q_i}$. As the element $\max_{Q_i}(\theta)$ is $P_{k_i}$-maximal, its projection to $W/W_{Q_j}$ is $P_{k_j}$-maximal by Lemma~\ref{lem:key_lem_type_A}, hence we have the equality $\pi_{Q_j} \circ \operatorname{max}_{Q_i}(\theta) = \operatorname{max}_{Q_j} \circ \pi_{j} \circ \operatorname{max}_{Q_i}(\theta)$. But this implies
    \begin{align*}
        \operatorname{max}_{Q_i}(\theta) &\geq \operatorname{min}_{Q_i} \circ \pi_{Q_j} \circ \operatorname{max}_{Q_i}(\theta) = \operatorname{min}_{Q_i} \circ \operatorname{max}_{Q_j} \circ \pi_{j} \circ \operatorname{max}_{Q_i}(\theta) \\
        &\geq \operatorname{min}_{Q_i} \circ \operatorname{max}_{Q_j}(\phi),
    \end{align*}
    where we used condition~\ref{itm:poset_i} for the last inequality. This completes the proof.
\end{proof}

\begin{corollary}
    \label{cor:partial_order_via_tableaux}
    The characterization (\ref{eq:relation_via_tableaux}) of the relation on $\ulW$ is fulfilled. In particular, the relation is a partial order.
\end{corollary}
\begin{proof}
    Let $(\theta, i), (\phi, j)$ be two elements in $\ulW$ written as tableaux with one column and entries $a_1, \dots, a_{k_i}$ and $b_1, \dots, b_{k_j}$ respectively. The first $k_i$ numbers in the one-line notation of $\widetilde\theta \coloneqq \min_B \circ \max_{Q_i}(\theta) = \theta_1 \cdots \theta_n$ are strictly increasing and therefore $\theta_s = a_s$ for all $s = 1, \dots, k_i$. The last $n-k_i$ numbers are strictly decreasing. The analogous statement holds for the one-line notation of $\widetilde\phi \coloneqq \min_B \circ \max_{Q_j}(\phi) = \phi_1 \cdots \phi_n$. 
    
    If $i \leq j$ and $b_s \leq a_s$ holds for all $s = 1, \dots, k_i$, then we have $b_1 \cdots b_{k_j} = \pi_j(\widetilde\phi) \leq \pi_j(\widetilde\theta) = a_1 \cdots a_{k_i} \theta_{k_i + 1} \cdots \theta_{k_j}$, because the last $k_j - k_i$ numbers are the largest numbers missing in $a_1 \cdots a_{k_i}$. This implies $(\theta, i) \geq (\phi, j)$ by Lemma~\ref{lem:poset_prop}\,\ref{itm:poset_i}. Conversely, if $(\theta, i) \geq (\phi, j)$, then $i \leq j$ and $\widetilde\theta \geq \phi$. Hence we have $b_s \leq a_s$ for all $s = 1, \dots, k_i$, as the first $k_i$ numbers in their one-line notation are increasingly ordered.
\end{proof}

We are now able to fully understand the covering relations of $\ulW$. If $(\theta, i)$ covers $(\phi, j)$, then we either have $i = j$ and $\theta > \phi$ is a covering relation in $W/W_i$ or we have $i < j$. In the second case we have $(\theta, i) > (\pi_{r} \circ \max_Q(\theta), r) > (\phi, j)$ for each $i < r < j$, so it follows $j = i+1$. Additionally, the lifts $\widetilde\theta = \min_Q \circ \max_{Q_i}(\theta)$ and $\widetilde\phi = \min_Q \circ \max_{Q_j}(\phi)$ agree, which we can show in the quotient $W/W_{Q_i} = W/W_{P_{k_i}} \cap W/W_{Q_j}$, as the natural map $W/W_{Q_i} \hookrightarrow W/W_{P_{k_i}} \cap W/W_{Q_j}$ is an isomorphism of posets onto its image.

Since $\widetilde\theta \geq \widetilde\phi$ holds by the previous lemma, we have
\begin{align*}
    &(\theta, i) = (\pi_{i}(\widetilde\theta), i) \geq (\pi_{i}(\widetilde\phi), i) > (\phi, j) \quad \text{and} \\
    &(\theta, i) > (\pi_{j}(\widetilde\theta), j) \geq (\pi_{j}(\widetilde\phi), j) = (\phi, j),
\end{align*}
hence $\widetilde\theta$ and $\widetilde\phi$ are equal in $W/W_{P_{k_i}}$ and in $W/W_{P_{k_j}}$. This yields
\begin{align*}
    \pi_{Q_j}(\widetilde\phi) = \operatorname{max}_{Q_j}(\phi) = \operatorname{max}_{Q_j} \circ \pi_{j}(\widetilde\theta) = \operatorname{max}_{Q_j} \circ \pi_{P_{k_j}} \circ \operatorname{max}_{Q_i}(\theta).
\end{align*}
But by Lemma~\ref{lem:key_lem_type_A} the element $\pi_{Q_j} \circ \max_{Q_i}(\theta)$ is $P_{k_j}$-maximal, so the right hand side is equal to $\pi_{Q_j} \circ \operatorname{max}_{Q_i}(\theta) = \pi_{Q_j}(\widetilde\theta)$. Therefore $\widetilde\theta = \widetilde\phi$.

In particular, if $(\theta, i) > (\phi, j)$ is covering relation in $\ulW$, then $\hat X_{(\phi, j)}$ is of codimension one in $\hat X_{(\theta, i)}$. Therefore the condition~\ref{itm:seshadri_strat_a} on a Seshadri stratification is fulfilled and the relation $(\theta, i) \geq (\phi, j)$ implies $\hat X_{(\phi, j)} \subseteq \hat X_{(\theta, i)}$. Conversely if $\hat X_{(\phi, j)} \subseteq \hat X_{(\theta, i)}$, then the Schubert variety $X_{\max_{Q_j}(\phi)} \subseteq G/Q_j$ is contained in $X_{\pi_{Q_j} \circ \max_{Q_i}(\theta)}$. Hence $\max_{Q_j}(\phi) \leq \pi_{Q_j} \circ \max_{Q_i}(\theta)$, which implies $(\phi, j) \leq (\theta, i)$ by Lemma~\ref{lem:poset_prop}\,\ref{itm:poset_i}.

\begin{lemma}
    \label{lem:prop_strata_type_A}
    Let $(\theta, i) \in \ulW$ and $\widetilde\theta = \min_Q \circ \max_{Q_i}(\theta)$. Then the following equality holds for all $i < j \leq m$:
    \begin{align*}
        \set{(v_1, \dots, v_m) \in \hat X_{(\theta, i)} \mid v_i = \dots = v_{j-1} = 0} = \hat X_{(\pi_{j}(\widetilde{\theta}), j)}.
    \end{align*}
\end{lemma}
\begin{proof}
    Let $v = (v_1, \dots, v_m) \in \hat X_{(\theta, i)}$ with $v_i = \dots = v_{j-1} = 0$. We choose a non-zero vector $w_r \in V(\omega_{k_r})$ for all $r = i, \dots, m$ with $v_r \in \K w_r$, such that $([w_i], \dots, [w_m])$ is an element of the Schubert variety
    \begin{align*}
        X_{(\theta, i)} = X_{\max_{Q_i}(\theta)} \subseteq \prod_{r=i}^m \PP(V(\omega_{k_r})).
    \end{align*}
    Because of the following commutative diagram, $([w_j], \dots, [w_m])$ lies in the Schubert variety to the element $\pi_{Q_j} \circ \max_{Q_i}(\theta) \in W/W_{Q_j}$:
    \begin{equation*}
    \begin{tikzcd}
    X_{\pi_{Q_i}(\theta)} \arrow[r, hook] \arrow[d, two heads] & \prod_{r=i}^m \PP(V(\omega_{k_r})) \arrow[d, two heads] &  \\
    X_{\pi_{Q_j}(\theta)} \arrow[r, hook] & \prod_{r=j}^m \PP(V(\omega_{k_r}))
    \end{tikzcd}
    \end{equation*}
    But by Lemma~\ref{lem:poset_prop} we have $\pi_{Q_j} \circ \operatorname{max}_{Q_i}(\theta) = \operatorname{max}_{Q_j} \circ \pi_j \circ \operatorname{max}_{Q_i}(\theta) = \operatorname{max}_{Q_j} \circ \pi_{j}(\widetilde{\theta})$, so $v$ is contained in the multicone $\hat X_{(\pi_{j}(\widetilde{\theta}), j)}$. 
    
    Conversely, every element of the multicone $\hat X_{(\pi_{j}(\widetilde{\theta}), j)}$ lies in the stratum $\hat X_{(\theta, i)}$ because of the surjectivity of the map $X_{\max_{Q_i}(\theta)} \twoheadrightarrow X_{\pi_{Q_j} \circ \max_{Q_i}(\theta)} = X_{\pi_{Q_j}(\widetilde\theta)}$.
\end{proof}

\begin{proof}[Proof of Theorem~\ref{thm:stratification_type_A}]
    It is well known, that Schubert-varieties are smooth in codimension one (see e.\,g. \cite[Corollary 3.5]{seshstratandschubvar}). Their multicones $\hat X_{(\theta, i)} \subseteq \prod_{j=1}^m V(\omega_{k_j})$ are closed, irreducible subvarieties and they are smooth in codimension one as well by Corollary~\ref{cor:smooth_in_codim_1}.

    We already proved condition~\ref{itm:seshadri_strat_a} and the equivalence of $(\phi, j) \leq (\theta, i)$ and the inclusion $\hat X_{(\phi, j)} \subseteq \hat X_{(\theta, i)}$ of their multicones. Next, we show \ref{itm:seshadri_strat_b}. Let $(\phi, j) \nleq (\theta, i)$ in $\ulW$. We need to prove, that the Pl\"ucker coordinate $p_{(\phi,j)}$ vanishes identically on $\hat X_{(\theta, i)}$. This is trivial, if $j < i$. Now we assume $j \geq i$ and set $\kappa = \pi_j \circ \max_{Q_i}(\theta) \in W/W_j$. Note that the affine cone $\hat X_{\kappa} \subseteq V(\omega_{k_j})$ of the Schubert variety $X_\kappa \subseteq G/P_{k_j}$ coincides with the projection of $\hat X_{(\theta, i)}$ to $V(\omega_{k_j})$. Now Schubert varieties and Pl\"ucker coordinates form a Seshadri stratification on $G/P_{k_j}$ by Proposition~\ref{prop:strat_grassmannian}, hence~\ref{itm:seshadri_strat_b} is fulfilled in this case. Therefore the function $p_{(\phi, i)}$ vanishes on $\hat X_{(\theta, i)}$, if and only if $\phi \nleq \kappa = \pi_j \circ \max_{Q_i}(\theta)$. By Lemma~\ref{lem:poset_prop}\,\ref{itm:poset_i} this is equivalent to $(\phi, j) \nleq (\theta, i)$.

    Lastly, we prove~\ref{itm:seshadri_strat_c}. We fix an element $(\theta, i) \in \ulW$. The function $p_{(\theta,i)}$ vanishes on all multicones $\hat X_{(\phi,j)}$ for $(\phi, j) < (\theta, i)$. This is clearly true for $j > i$, otherwise it follows from~\ref{itm:seshadri_strat_c} for the stratification on $G/P_{k_i}$. 
    
    Conversely, let $v = (v_1, \dots, v_m) \in \hat X_{(\theta, i)}$ such that $p_{(\theta,i)}(v) = 0$. In the case of $v_i = 0$, the element $v$ is contained in the multicone $\hat X_{(\phi, j)}$ for $j = i+1$ and $\phi = \pi_j \circ \max_{Q_i}(\theta)$ by Lemma~\ref{lem:prop_strata_type_A} and we have $(\phi, j) < (\theta, i)$. For $v_i \neq 0$, its projective class $[v_i]$ can be viewed as an element of the Schubert variety $X_\theta \subseteq G/P_{k_i}$. Again, using the Seshadri stratification on $G/P_{k_i}$ we see that $[v_i]$ is contained in the Schubert variety $X_\phi$ to an element $\phi < \theta$ in $W/W_i$. For each $r = i, \dots, m$ we choose a non-zero vector $w_r \in V(\omega_{k_r})$ with $v_r \in \K w_r$ and $w \coloneqq ([w_i], \dots, [w_m]) \in X_{(\theta, i)}$. Then $w$ lies in a Schubert cell $C_\sigma \subseteq G/Q_i$ for a unique element $\sigma \in W/W_{Q_i}$. It satisfies $\pi_i(\sigma) \leq \phi < \theta$, so $\sigma \leq \max_{Q_i} \circ \pi_i(\sigma) < \max_{Q_i}(\theta)$. Hence $v$ is contained in $\hat X_{(\pi_i(\sigma), i)}$, which completes the proof.
\end{proof}

\subsection{The fan of monoids and standard monomial theory}

Let $\mathrm{SSYT}_Q$ be the set of all semistandard Young-tableaux with entries in $[n]$, where only columns of length $k_1, \dots, k_m$ may appear. Equivalently, this is the union of the sets $\mathrm{YT}(\mu)$ over all $\mu \in \N_0 \mkern1mu \omega_{k_1} + \dots + \N_0 \mkern1mu \omega_{k_m}$. A Young-tableau is contained in this union, if and only if all columns come from elements in $\ulW$.

\begin{corollary}
    \label{cor:smt_type_A}
    $ $
    \begin{enumerate}[label=(\alph{enumi})]
        \item \label{itm:smt_type_A_a} The following map is a bijection:
        \begin{align*}
            \mathrm{SSYT}_Q \to \Gamma, \quad ((\theta_1, i_1), \dots, (\theta_\ell, i_\ell)) \mapsto e_{(\theta_1, i_1)} + \dots + e_{(\theta_\ell, i_\ell)}.
        \end{align*}
        \item The Seshadri stratification on $G/Q$ is normal and balanced.
        \item \label{itm:smt_type_A_c} The set $\mathbb G$ of all indecomposable elements in $\Gamma$ coincides with the set
        \begin{align*}
            \Gamma(1) = \set{\underline a \in \Gamma \mid \vert \mkern-1mu \deg \underline a \mkern2mu \vert = 1} = \set{e_{(\theta, i)} \mid (\theta, i) \in \ulW}
        \end{align*}
        of all elements of total degree $1$ in $\Gamma$.
        \item \label{itm:smt_type_A_d} Let $\mathbb G_R$ be the set of all Pl\"ucker coordinates $p_{(\theta, i)}$ for $(\theta, i) \in \ulW$. Then the standard monomial basis from Proposition~\ref{prop:smt} agrees with the basis from Theorem~\ref{thm:smt_basis_G_Q}.
    \end{enumerate}
\end{corollary}
\begin{proof}
    \begin{enumerate}[label=(\alph{enumi})]
        \item For a tableau $T \in \mathrm{SSYT}_Q$ with columns $(\theta_1, i_1), \dots, (\theta_\ell, i_\ell)$ consider the regular function $f_T = p_{(\theta_1, i_1)} \cdots p_{(\theta_\ell, i_\ell)}$. Since $T$ is semistandard, it follows from the equivalence~(\ref{eq:relation_via_tableaux}) that there exists a maximal chain in $\ulW$ containing all elements $(\theta_1, i_1), \dots, (\theta_\ell, i_\ell)$. Hence $f_T$ has the quasi-valuation $\mathcal V(f_T) = e_{(\theta_1, i_1)} + \dots + e_{(\theta_\ell, i_\ell)}$, so the map $\mathrm{SSYT}_Q \to \Gamma$ is well-defined. The injectivity is already contained in the definition of this map, since one can reconstruct the semistandard tableau from the coefficients of the vectors $e_p$, $p \in \ulW$. We also know from Theorem~\ref{thm:smt_basis_G_Q} that the functions $f_T$ for $T \in \mathrm{SSYT}_Q$ form a basis of $R$. Therefore the map $\mathrm{SSYT}_Q \to \Gamma$ is surjective as well.
        \item The quasi-valuation of extremal functions does not depend on the choice of the total order $\geq^t$ on the poset $\ulW$ and every element in $\Gamma$ is the quasi-valuation of a product of extremal functions in a common maximal chain. Hence the stratification is balanced. By part~\ref{itm:smt_type_A_a}, the monoid $\Gamma_{\mathfrak C}$ to a maximal chain $\mathfrak C$ in $\ulW$ coincides with $\N_0^{\mathfrak C}$, which clearly is saturated. So the stratification is also normal.
        \item This statement is a consequence of part~\ref{itm:smt_type_A_a}.
        \item Let $\underline a \in \Gamma$ and $T$ be its corresponding tableau in $\mathrm{SSYT}_Q$. Then part~\ref{itm:smt_type_A_d} follow from the fact, that the unique decomposition $\underline a = \underline a^1 + \dots + \underline a^s$ into indecomposables with $\min \supp \underline a^k \geq \max \supp \underline a^{k+1}$ for all $k = 1, \dots, s-1$ is given by the columns of $T$, where $\underline a^1$ corresponds to the rightmost and $\underline a^s$ to the leftmost column. \hfill\qedhere
    \end{enumerate}
\end{proof}

Comparing this result to the fan of monoids of Hodge-type stratifications (see Example~\ref{ex:hodge_type}) suggests that our stratification on $G/Q$ should also be of Hodge type. Fortunately, we can skip the computation of the bonds with the following argument: The bond of every minimal element $p \in \underline W$ is determined by the degree of the extremal function $f_p$, hence the bond is $1$. Let $p$ be any non-minimal element in $\underline W$ and $q_1, \dots, q_s \in \underline W$ be the elements covered by $p$. \WWlog{} we can assume $b_1 \geq \dots \geq b_s$, where we set $b_i \coloneqq b_{p, q_i}$ for all $i = 1, \dots, s$. Now choose any regular function $g \in R$ which restricts to a non-zero function $\overline g$ on $\hat X_p$ with vanishing multiplicity $\nu_{p, q_1}(\overline g) = 1$. Such a function is called a \textit{uniformizer} in the local ring $\mathcal O_{\hat X_p, \hat X_{q_1}}$. We then have the following inequalities:
\begin{align*}
    \frac{\nu_{p,q_i}(g)}{b_i} \geq \frac{1}{b_i} \geq \frac{1}{b_1} = \frac{\nu_{p,q_1}(g)}{b_1}.
\end{align*}
By the definition of the quasi-valuation, the coefficient of the basis vector $e_p$ in $\mathcal V(g)$ is therefore equal to $1/b_1$. But $\mathcal V(g)$ is contained in the fan of monoids $\Gamma$ of our stratification on $G/Q$, hence $\mathcal V(g)$ has integer coefficients by Corollary~\ref{cor:smt_type_A}. This shows $1 = b_1 = \dots = b_s$.

In particular, the stratification on $G/Q$ is of LS-type and the index poset $\mathcal I$ is totally ordered. This allows the computation of the multidegrees of $G/Q$ via Corollary~\ref{cor:multidegree_formula}.

\begin{example}
    The multidegrees of $G/B$ in the Dynkin types $\texttt{A}_2$ and $\texttt{A}_3$ can be read off the Hasse-diagrams from Figure~\ref{fig:hasse_ulWlambda_type_A}. More precisely, for a tuple $\underline j \in \N_0^m$ with $j_1 + \dots + j_m = \dim G/B$ the multidegree $\deg_{\underline j}(G/B)$ is given by the number of maximal chains in $\underline W$ containing exactly $j_i + 1$ many elements of $W/W_{i}$. For $n = 2$ this yields $\deg_{(2,1)}(G/B) = 1$ and $\deg_{(2,1)}(G/B) = 1$ for $n = 2$ and for $n = 3$ we get
    \begin{align*}
        \deg_{(1,2,3)}(G/B) = 1, \quad &\deg_{(1,3,2)}(G/B) = 2, \quad \deg_{(1,4,1)}(G/B) = 2, \\
        \deg_{(2,2,2)}(G/B) = 2, \quad &\deg_{(2,1,3)}(G/B) = 1, \quad \deg_{(2,3,1)}(G/B) = 2, \\
        \deg_{(3,1,2)}(G/B) = 1, \quad &\deg_{(3,2,1)}(G/B) = 1.
    \end{align*}
\end{example}

\subsection{Schubert varieties and generalizations}

The definition of the multiprojective Seshadri stratification on $G/Q$ leads to several questions on possible generalizations. For instance, our construction was based on a few choices, \ie the numbers $k_1, \dots, k_m$, the embedding of $G/Q$ into a product of projective spaces and the index poset $\mathcal I$ we used for the stratification. We will answer the following questions in a separate paper. The rest of this section teases a few of the results and also shows some of the occurring problems.
\begin{itemize}
    \item Can we also choose the numbers $k_1, \dots, k_m$ in descending order $n-1 \geq k_1 > \dots > k_m \geq 1$ or in any order?
    \item Are there multiprojective stratifications using other embeddings of $G/Q$ into a product of projectivized irreducible representations?
    \item Is it possible to also define Seshadri stratifications on Schubert varieties \wrt{} the Plücker embedding or any other embedding?
    \item Can the construction even be generalized to Schubert varieties in arbitrary Dynkin types?
    \item If we can define any of the above stratifications, is it balanced and normal? Which tableau model corresponds to the associated fan of monoids?
\end{itemize}

Our stratification on $G/Q$ can actually be directly generalized to all other Dynkin types with the following caveat: The choice of the numbers $k_1, \dots, k_m$ corresponds to simple roots $\alpha_1, \dots, \alpha_m$ in the Dynkin diagram of our group. One obtains a well-defined Seshadri stratification, if and only if these roots are ordered along the Dynkin diagram, \ie there exists a path, which visits them in the order $\alpha_1, \dots, \alpha_m$. The path may also visit other roots along the way, but no root may be visited twice. Therefore we cannot define a stratification on $G/B$ in the types $\texttt{D}_n$ for $n \geq 4$ and $\texttt{E}$.

Since the Dynkin diagram in type $\texttt{A}$ is a line, this result shows that we can, in fact, choose the reverse ordering $n-1 \geq k_1 > \dots > k_m \geq 1$. The resulting stratification is defined in the exact same way. One can even use the same proofs, except for Lemma~\ref{lem:key_lem_type_A}, which is to be expected, as this Lemma heavily uses the order of the numbers $k_i$ (or the associated simple roots). Actually, this Lemma is exactly the reason why the simple roots need to be ordered along the Dynkin diagram. 

The stratification using the reversed ordering is also of Hodge type and the tableau model corresponding to the fan of monoids $\Gamma$ belongs to -- what we call -- \textit{Anti-Young-tableaux}. In contrast to usual Young-tableaux, the boxes in the underlying Anti-Young-diagrams of anti-tableaux are aligned at the bottom and right instead of top and left. The notion of semistandard tableaux does not change: An anti-tableau is called \textit{semistandard}, if the entries are weakly increasing along the rows and strictly increasing along the columns. Again, the underlying poset of the anti-stratification is designed, such that an anti-tableau is semistandard, if and only if its columns are contained in a common maximal chain.

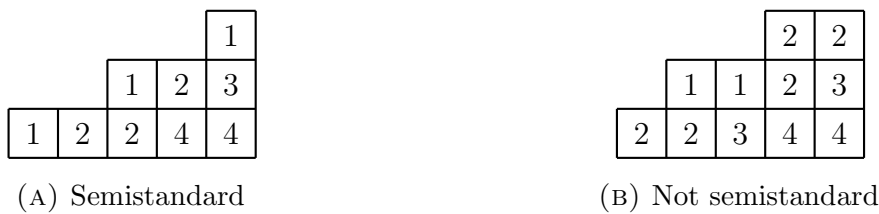
\begin{figure}
\centering
\begin{subfigure}{.5\textwidth}
    \begin{center}
        \begin{tikzpicture}[scale=0.65]
            \node at (-4.5,0.5) {$1$};
            \node at (-3.5,0.5) {$2$};
            \node at (-2.5,0.5) {$2$};
            \node at (-1.5,0.5) {$4$};
            \node at (-0.5,0.5) {$4$};
            \node at (-2.5,1.5) {$1$};
            \node at (-1.5,1.5) {$2$};
            \node at (-0.5,1.5) {$3$};
            \node at (-0.5,2.5) {$1$};
    
            \draw [thick] (0,0) -- (-5,0);
            \draw [thick] (0,0) -- (0,3);
            \draw [thick] (0,1) -- (-5,1);
            \draw [thick] (0,2) -- (-3,2);
            \draw [thick] (0,3) -- (-1,3);
            \draw [thick] (-1,0) -- (-1,3);
            \draw [thick] (-2,0) -- (-2,2);
            \draw [thick] (-3,0) -- (-3,2);
            \draw [thick] (-4,0) -- (-4,1);
            \draw [thick] (-5,0) -- (-5,1);
        \end{tikzpicture}
    \end{center}
    \caption{Semistandard}
    \label{fig:anti-ss}
\end{subfigure}%
\begin{subfigure}{.5\textwidth}
    \begin{center}
        \begin{tikzpicture}[scale=0.65]
            \node at (-4.5,0.5) {$2$};
            \node at (-3.5,0.5) {$2$};
            \node at (-2.5,0.5) {$3$};
            \node at (-1.5,0.5) {$4$};
            \node at (-0.5,0.5) {$4$};
            \node at (-3.5,1.5) {$1$};
            \node at (-2.5,1.5) {$1$};
            \node at (-1.5,1.5) {$2$};
            \node at (-0.5,1.5) {$3$};
            \node at (-1.5,2.5) {$2$};
            \node at (-0.5,2.5) {$2$};
    
            \draw [thick] (0,0) -- (-5,0);
            \draw [thick] (0,0) -- (0,3);
            \draw [thick] (0,1) -- (-5,1);
            \draw [thick] (0,2) -- (-4,2);
            \draw [thick] (0,3) -- (-2,3);
            \draw [thick] (-1,0) -- (-1,3);
            \draw [thick] (-2,0) -- (-2,3);
            \draw [thick] (-3,0) -- (-3,2);
            \draw [thick] (-4,0) -- (-4,2);
            \draw [thick] (-5,0) -- (-5,1);
        \end{tikzpicture}
    \end{center}
    \caption{Not semistandard}
    \label{fig:anti-no-ss}
\end{subfigure}
\caption{Anti-Young-tableaux for $n = 4$}
\label{fig:anti-tableaux}
\end{figure}

\ytableausetup{smalltableaux}
\begin{example}
    \label{ex:antitableaux}
    If we choose the reversed ordering $k_1 = 2$ and $k_2 = 1$ for $G/B$ in Dynkin type $\texttt{A}_2$, we obtain the following poset and stratification.
    \begin{center}
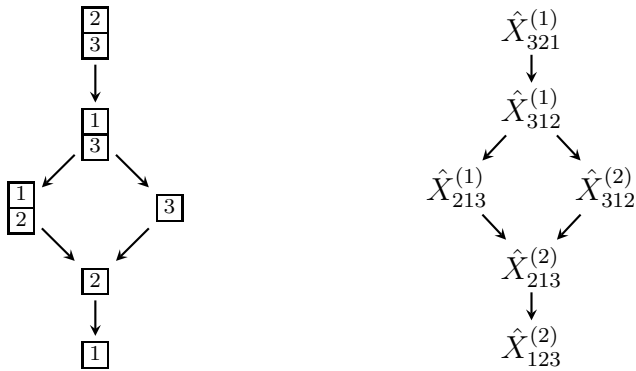
    \begin{minipage}{.35\textwidth}
        \begin{center}
            \begin{tikzpicture}[scale=0.65]
            \node (p23) at (0,0.57) {$\begin{ytableau} 2 \\ 3 \end{ytableau}$};
            \node (p13) at (0,-1.5)  {$\begin{ytableau} 1 \\ 3 \end{ytableau}$};
            \node (p12) at (-1.5,-3) {$\begin{ytableau} 1 \\ 2 \end{ytableau}$};
            \node (p3) at (1.5,-3) {$\begin{ytableau} 3 \end{ytableau}$};
            \node (p2) at (0,-4.5) {$\begin{ytableau} 2 \end{ytableau}$};
            \node (p1) at (0,-6) {$\begin{ytableau} 1 \end{ytableau}$};
            
            \draw [thick, shorten <=-2pt, shorten >=-2pt, -stealth] (p3) -- (p2);
            \draw [thick, shorten <=-2pt, shorten >=-2pt, -stealth] (p2) -- (p1);
            \draw [thick, shorten <=-2pt, shorten >=-2pt, -stealth] (p13) -- (p3);
            \draw [thick, shorten <=-2pt, shorten >=-2pt, -stealth] (p12) -- (p2);
            \draw [thick, shorten <=-2pt, shorten >=-2pt, -stealth] (p23) -- (p13);
            \draw [thick, shorten <=-2pt, shorten >=-2pt, -stealth] (p13) -- (p12);
            \end{tikzpicture}
        \end{center}
    \end{minipage}
    \begin{minipage}{.35\textwidth}
        \begin{center}
            \begin{tikzpicture}[scale=0.7]
            \node (p3) at (0,0) {$\hat X_{321}^{(1)}$};
            \node (p2) at (0,-1.5)  {$\hat X_{312}^{(1)}$};
            \node (p1) at (-1.4,-3) {$\hat X_{213}^{(1)}$};
            \node (p23) at (1.4,-3) {$\hat X_{312}^{(2)}$};
            \node (p13) at (0,-4.5) {$\hat X_{213}^{(2)}$};
            \node (p12) at (0,-6) {$\hat X_{123}^{(2)}$};
            
            \draw [thick, shorten <=-2pt, shorten >=-2pt, -stealth] (p3) -- (p2);
            \draw [thick, shorten <=-2pt, shorten >=-2pt, -stealth] (p2) -- (p1);
            \draw [thick, shorten <=-2pt, shorten >=-2pt, -stealth] (p2) -- (p23);
            \draw [thick, shorten <=-2pt, shorten >=-2pt, -stealth] (p1) -- (p13);
            \draw [thick, shorten <=-2pt, shorten >=-2pt, -stealth] (p23) -- (p13);
            \draw [thick, shorten <=-2pt, shorten >=-2pt, -stealth] (p13) -- (p12);
            \end{tikzpicture}
        \end{center}
    \end{minipage}
    \end{center}
\end{example}
\ytableausetup{nosmalltableaux}

It is not possible to directly generalize our stratifications on $G/Q$ to Schubert varieties, since Pl\"ucker coordinates have the wrong vanishing sets. As an example, let us take the Schubert variety $X_\tau \subseteq \mathrm{SL}_3(\K)/B$ for $\tau = 312$. The Pl\"ucker coordinate $p_{(3,1)}$ on the multicone $\hat X_{312} \subseteq V(\omega_1) \times V(\omega_2)$ vanishes on the two subvarieties $\hat X_{213}$ and $\hat X_{132}$, which are both of codimension one. Therefore $213$ and $132$ should be covered by $312$ in the underlying poset of the stratification. Analogously, $p_{(2,1)}$ vanishes on $\hat X_{123} \subseteq \hat X_{213}$, so $213$ covers $123$. But both $123$ and $132$ should have the same associated extremal function $p_{(1,1)}$, which is impossible due to condition~\ref{itm:seshadri_strat_b} on a Seshadri stratification. 

This behaviour fits with the problems of the standard monomial theory of Young-tableaux: Let us fix the standard monomial basis $\mathbb B$ of $\K[G/Q]$ indexed by the set $\mathrm{SSYT}_Q$. For each Weyl group coset $\tau \in W/W_Q$ one can consider the set of all basis vectors in $\mathbb B$ which do not vanish identically on $\hat X_\tau \subseteq \hat X$ and ask whether this set is a basis of the multihomogeneous coordinate ring $\K[X_\tau]$ with respect to the embedding $X_\tau \hookrightarrow \prod_{i=1}^m \PP(V(\omega_{k_i}))$. In general, this statement is false (see Example~\ref{ex:no_induced_smt_basis}). However, for all strata, \ie the Schubert varieties $X_{\max_{Q_i}(\theta)}$ for $(\theta, i) \in \underline W$, we do obtain a standard monomial theory in this way via Corollary~\ref{cor:smt_induced_strat}. This explains why not every Schubert subvariety of $G/Q$ needs to appear as a stratum.

\begin{remark}
    One can prove the following: For a fixed index $i \in [m]$ the set of lifts
    \begin{align*}
        \Set{\mkern2mu \operatorname{min}_Q \circ \operatorname{max}_{Q_i}(\theta) \in W/W_Q \mid (\theta, i) \in \ulW \mkern2mu }
    \end{align*}
    coincides with the set of all elements in $W/W_Q$, which are $Q_i$-minimal and $Q^i$-maximal for the parabolic subgroup $Q^i = \bigcap_{j=1}^i P_{k_i}$. We will show this statement in a future article.
\end{remark}

\begin{example}
    \label{ex:no_induced_smt_basis}
    Again, we consider the Schubert variety $X_{312}$ in $G/B$ of Dynkin type $\texttt{A}_2$. Notice that this Schubert variety occurs as a stratum in the Seshadri stratification from Example~\ref{ex:antitableaux}, which is of Hodge type and therefore normal and balanced. Hence we can use Corollary~\ref{cor:smt_induced_strat} to obtain a standard monomial basis on $X_{312}$: As before, the Plücker coordinates are indexed by $\underline W$ and the Plücker monomials $p_{(1,2)}^a \mkern2mu p_{(2,2)}^b \mkern2mu p_{(3,2)}^c \mkern2mu p_{(12,1)}^d \mkern2mu p_{(13,1)}^e$ with $a,b,c,d,e \in \N_0$ and $cd = 0$ form a basis of the homogeneous coordinate ring $\K[\hat X_{312}]$. In particular, the graded component of degree $(1,1)$ is of dimension $5$.

    On the other hand, in the standard monomial basis of $\K[G/B]$ from Corollary~\ref{cor:smt_type_A} there are $6$ standard monomials of degree $(1,1)$, which do not vanish identically on $X_{312}$:
    \begin{align*}
        p_{(1,2)} p_{(12,1)}, \ p_{(1,2)} p_{(13,1)}, \ p_{(2,2)} p_{(12,1)}, \ p_{(2,2)} p_{(13,1)}, \ p_{(3,2)} p_{(12,1)}, \ p_{(3,2)} p_{(13,1)}.
    \end{align*}
    Hence the restrictions of these function cannot form a basis of $\K[X_{312}]_{(1,1)}$.
\end{example}

\appendix
\section{Multiproj-schemes}
\label{sec:multiproj}

\subsection{The Multiproj-construction}

Let $R = \bigoplus_{\underline d \in \Z^m} R_{\underline d}$ be a (commutative) ring, which is graded by the group $\Z^m$. An element $r \in R$ is called \textit{(multi-$\mkern-1mu$)homogeneous}, if it is contained in a subgroup $R_{\underline d}$. In this case $\deg r \coloneqq \underline d$ is its \textit{(multi-$\mkern-1mu$)degree}. Ideals generated by homogeneous elements are called \textit{(multi-$\mkern-1mu$)homogeneous ideals}.

There is a similar construction to the Proj-construction for $\N_0$-graded rings, which associates a scheme $\Multiproj R$ to the multigraded ring $R$. In general, this scheme does not have all the nice properties of the usual Proj-scheme, for example it need not be projective or separated. To introduce these schemes, we follow the construction from Brenner and Schr\"oer in \cite{brenner2003ample}. The grading on $R$ corresponds to an action of the $m$-torus $\Spec \Z[t_1^{\pm 1}, \dots, t_m^{\pm 1}]$ on $\Spec R$. There exists a quotient $\operatorname{Quot}(R)$ of $\Spec R$ with respect to this action in the category of ringed spaces. However, this quotient is not a quotient in the category of schemes in general. 

An element $f \in R$ is called \textbf{relevant}, if it is homogeneous and the degrees of the homogeneous elements $g \in R$, which divide some power $f^k$ for $k \in \N_0$, generate a subgroup of $\Z^m$ of finite index. For every relevant $f \in R$ the morphism $\Spec R_f \to \Spec R_{(f)}$ is a geometric quotient (in the sense of GIT), where $R_{(f)}$ denotes the subring of $R_f$ of all elements of multidegree zero. Therefore we have an open subset $D_+(f) \subseteq \operatorname{Quot}(R)$ isomorphic to $\Spec R_{(f)}$. The Multiproj-scheme of $R$ is then defined as the locally ringed space
\begin{align}
    \label{eq:def_multiproj_as_quotient}
    \Multiproj R = \bigcup_{f \in R \atop \text{relevant}} D_+(f) \subseteq \operatorname{Quot}(R).
\end{align}
Let $R_+$ be the ideal in $R$ generated by all relevant elements. It is called the \textbf{irrelevant ideal}. The induced morphism
\begin{align*}
    \Spec R \setminus V(R_+) \to \Multiproj R
\end{align*}
is then a geometric quotient with respect to the torus action.

There is also another way of realizing the scheme $\Multiproj R$, which directly generalizes the usual Proj-construction. We denote this scheme by $\Multiproj R$ as well. Set-theoretically it is given by
\begin{align*}
    \Multiproj R = \set{ P \subseteq R \mid \text{$P$ multihomogeneous prime ideal}, R_+ \nsubseteq P}
\end{align*}
and the closed subsets are those of the form
\begin{align*}
    V(I) = \set{P \in \Multiproj R \mid P \supseteq I},
\end{align*}
where $I$ is a multihomogeneous ideal in $R$. For each $P \in \Multiproj R$ let $R_{(P)}$ denote the subring of homogeneous elements of multidegree $0$ in the localization $R_P$. For an open subset $U \subseteq \Multiproj R$ we define the ring $\mathcal O(U)$ of functions 
\begin{align*}
    f: U \to \coprod_{P \in U} R_{(P)},
\end{align*}
which are locally given by a quotient of elements in $R$: To each $P \in U$ there exists an open neighborhood $V$ of $P$ in $U$ and multihomogeneous elements $r,s \in R$ of the same multidegree, such that $f(P) = \frac{r}{s} \in R_{(P)}$.

The topological space $\Multiproj R$ together with the sheaf $\mathcal O$ forms a locally ringed space and the stalk at a point $P \in \Multiproj R$ is canonically isomorphic to the local ring $R_{(P)}$. For every relevant element $f \in R$ we have an isomorphism between the open subset $D_+(f) = \set{P \in \Multiproj R \mid f \notin P}$ and $\Spec R_{(f)}$, topologically given by
\begin{align*}
    \chi_f: D_+(f) \to \Spec R_{(f)}, \quad P \mapsto \langle \phi_f(P) \rangle \cap R_{(f)},
\end{align*}
where $\phi_f$ denotes the natural map $R \to R_f$ and $\langle \phi_f(P) \rangle$ is the ideal generated by $\phi_f(P)$. This can be seen as follows: First of all, one can use $\phi_f$ to construct an isomorphism $\Multiproj R_f \to D_+(f)$ of locally ringed spaces. On the other hand, the inclusion $\iota_f: R_{(f)} \hookrightarrow R_f$ induces an isomorphism $\Multiproj R_f \to \Spec R_{(f)}$. We can write the inclusion $\iota_f$ as the composition of the embedding $R_{(f)} \hookrightarrow R_{(f)}[t_1^{\pm 1}, \dots, t_m^{\pm 1}]$ and a ring homomorphism $R_{(f)}[t_1^{\pm 1}, \dots, t_m^{\pm 1}] \to R_f$ sending $t_i$ to a homogeneous unit in $R_f$ of multidegree $e_i$ (which exists, since $f$ is relevant). Both maps are graded ring homomorphisms, when we set $\deg t_i = e_i$. The first map induces an isomorphism $\Multiproj R_{(f)}[t_1^{\pm 1}, \dots, t_m^{\pm 1}] \to \Spec R_{(f)}$ and the second map is an isomorphism of graded rings. In total, we get the desired isomorphism $D_+(f) \cong \Spec R_{(f)}$. Note that the product $fg$ of two relevant elements $f, g \in R$ is relevant as well and the inclusions we just constructed are compatible in the sense that the following diagram commutes:
\begin{equation*}
    \begin{tikzcd}
    \Spec R_{(fg)} \arrow[d, hook] \arrow[r, hook] & \Spec R_{(f)} \arrow[d, hook] \\
    \Spec R_{(g)} \arrow[r, hook] & \Multiproj R.
    \end{tikzcd}
\end{equation*}
In particular, we see that $\Multiproj R$ is a scheme, which is isomorphic to the scheme defined in (\ref{eq:def_multiproj_as_quotient}). The structure morphism $\Spec R \setminus V(R_+) \to \Multiproj R$ maps every homogeneous prime ideal not containing $R_+$ to itself. 

\begin{lemma}
    \label{lem:multiproj_morphisms}
    Let $A: \Z^m \to \Z^m$ be an injective group homomorphism.
    \begin{enumerate}[label=(\alph{enumi})]
        \item \label{itm:lem_multiproj_morphisms} For every ring homomorphism $\phi: R \to S$ between two $\Z^m$-graded rings with $\phi(R_{\underline d}) \subseteq S_{A(\underline d)}$ for all $\underline d \in \Z^m$, the morphism $\Spec S \to \Spec R$ induces a morphism
        \begin{align}
        \label{eq:morphism_by_graded_ring_homom}
            (\Multiproj S) \, \setminus \, V( \langle \phi(R_+) \rangle ) \to \Multiproj R.
        \end{align}
        \item The inclusion of the graded ring 
        \begin{align*}
            R^{(A)} = \bigoplus_{\underline d \in \Z^m} R_{A(\underline d)} \subseteq R
        \end{align*}
        into $R$ induces an isomorphism $\Multiproj R \to \Multiproj R^{(A)}$.
    \end{enumerate}
\end{lemma}
\begin{proof}
    Since $A$ is injective, the image of a relevant element $f \in R$ under $\phi$ is relevant in $S$. The map $R_{(f)} \to S_{(\phi(f))}$ induces a morphism $\Spec S_{(\phi(f))} \to \Spec R_{(f)}$. For $f, g \in R$ relevant these morphisms can be glued along the inclusion $\Spec R_{(fg)} \hookrightarrow \Spec R_{(f)}$. They therefore define the desired morphism in (\ref{eq:morphism_by_graded_ring_homom}) because the subsets $D_+(\phi(f)) \subseteq \Multiproj S$ for $f \in R$ relevant cover the scheme $(\Multiproj S) \setminus V( \mkern1mu \langle \phi(R_+) \rangle )$.

    By part~\ref{itm:lem_multiproj_morphisms} the inclusion $R^{(A)} \hookrightarrow R$ induces a morphism 
    \begin{align*}
        \Multiproj R \setminus V( \langle R^{(A)}_+ \rangle ) \to \Multiproj R^{(A)}.
    \end{align*}
    First, we show $V( \langle R^{(A)}_+ \rangle ) = \varnothing$. Let $P \in \Multiproj R$ be a homogeneous prime ideal containing $R^{(A)}_+$ and $f \in R$ be a relevant element. We fix homogeneous divisors $g_i \mid f^{n_i}$ (for $i = 1, \dots, m$), such that $\deg g_1, \dots, \deg g_m$ generate a subgroup of $\Z^m$ of finite index. Since $A(e_1), \dots, A(e_m) \in \Z^m$ generate $\Q^m$ as a vector space, the degree $\underline d$ of $f$ can be expressed in the form
    \begin{align*}
        \underline d = \sum_{i=1}^m \frac{p_i}{q_i} \cdot A(e_i)
    \end{align*}
    for some $p_i \in \Z$ and $q_i \in \N$. We set $N = q_1 \dots q_m$. Therefore $f^{N}$ lies in the subring $R^{(A)}$. In the same way we get natural numbers $N_i$ with $g_i^{N_i} \in R^{(A)}$. Then $g_i^{N N_i}$ divides $f^{n_i N N_i}$ and their degrees $\deg g^{N N_i}$ still generate a subgroup of $\Z^m$ of finite index. These degrees lie in the image of $A$. Using the injectivity of $A$ we see that $f^N$ is relevant in $R^{(A)}$. So the ideal $P$ contains $f^N$ and since $P$ is prime, it contains $f$ as well. Hence $R_+ \subseteq P$.

    If $f \in R^{(A)}$ is relevant, then $f$ is also relevant in $R$ and the induced map $R^{(A)}_{(f)} \to R_{(f)}$ is a ring isomorphism. This proves that $\Multiproj R \to \Multiproj R^{(A)}$ is an isomorphism of schemes.
\end{proof}

\begin{example}
    The algebra $R = \K[x,y]$ over $\K$ with the grading $\deg x = (1,0)$ and $\deg y = (1,1)$ shows, that for $\N_0^m$-graded rings the irrelevant ideal does not need to agree with the ideal $R_+'$ generated by all homogeneous elements of degrees $\underline d$ with $d_i \geq 1$ for all $i$. By Lemma~\ref{lem:multiproj_morphisms} we can regrade $R$ via the group homomorphism $A: \Z^2 \to \Z^2$, $(d_1, d_2) \mapsto (d_1 + d_2, d_2)$ and obtain an isomorphism $\Multiproj R \cong \Multiproj R^{(A)}$. The irrelevant ideal of $R^{(A)}$ is equal to $(xy)$ and $\Multiproj R^{(A)}$ is isomorphic to a point. On the other hand $R_+'$ is generated by $y$ and there are two multihomogeneous prime ideals in $R$ not containing $y$: $(0)$ and $(x)$.
\end{example}

\begin{example}
    \label{ex:non_separated_scheme}
    Consider the algebra $R = \K[x,y,z]$ over $\K$ with the grading $\deg x = (1,0)$, $\deg y = (0,1)$ and $\deg z = (1,1)$. Its irrelevant ideal is generated by the relevant elements $xy, xz$ and $yz$, hence the three open subsets defined by these elements cover $\Multiproj R$. We have
    \begin{align*}
        \K[x,y,z]_{(xy)} \cong \K[\tfrac{xy}{z}] \cong \K[x,y,z]_{(xz)} \quad \text{and} \quad \K[x,y,z]_{(yz)} \cong \K[\tfrac{z}{xy}].
    \end{align*}
    The corresponding affine schemes are glued along their intersections
    \begin{align*}
        \K[x,y,z]_{(xy \, \cdot \, xz)} \cong \K[x,y,z]_{(xy \, \cdot \, xz)} \cong \K[x,y,z]_{(xy \, \cdot \, xz)} \cong \K[\tfrac{xy}{z}, \tfrac{z}{xy}]
    \end{align*}
    via the obvious inclusions. Therefore $\Multiproj R$ is a non-separated projective line with two points at infinity.
\end{example}

\begin{lemma}
    \label{lem:closed_subschemes_multiproj}
    Let $I, J \subseteq R$ be multihomogeneous ideals.
    \begin{enumerate}[label=(\alph{enumi})]
        \item The map $R \twoheadrightarrow R/I$ induces a closed immersion $\Multiproj R/I \hookrightarrow \Multiproj R$. Topologically its image coincides with the closed subset $V(I)$. If $I$ is prime, then $\Multiproj R/I$ is an integral scheme.
        \item The scheme-theoretic intersection of $V(I)$ and $V(J)$ is given by $V(I + J)$.
        \item The scheme-theoretic union of $V(I)$ and $V(J)$ is given by $V(I \cap J)$.
    \end{enumerate}
\end{lemma}
\begin{proof}
    All three statements can be checked in the open subschemes $\Spec R_{(f)}$ for $f \in R$ relevant. They are compatible with the projection $R \twoheadrightarrow R/I$ as the scheme $\Multiproj R/I$ is covered by all $D_+(f + I)$ for $f \in R \setminus I$ relevant. 
    
    Let $f \in R$ be relevant and $\phi_f$ be the natural map $R \to \R_f$. The second and third statement follow from the analogous statement for affine schemes and the fact that
    \begin{align*}
        \chi: \set{\text{homogeneous ideals in $R$}} &\to \set{\text{ideals in $R_{(f)}$}} \\
        I &\mapsto \langle \phi_f(I) \rangle \cap R_{(f)}
    \end{align*}
    preserves sums and intersections: $\chi(I + J) = \chi(I) + \chi(J)$ and $\chi(I \cap J) = \chi(I) \cap \chi(J)$.
\end{proof}

\begin{lemma}[{\cite[Lemma 3.6]{kuronya2019geometry}}]
    \label{lem:multiproj_dimension}
    If $R$ is an integral domain and the degrees of its non-zero, homogeneous elements span a subgroup of maximal rank in $\Z^m$, then $\Multiproj R$ is an integral scheme of dimension $\dim \Spec R - m$.
\end{lemma}

\begin{lemma}
    \label{lem:multiproj_as_variety}
    If $R$ is an $\N_0^m$-graded, noetherian, reduced ring, that is finitely generated in total degree $1$ as an $R_0$-algebra, then $\Multiproj R$ is a separated, reduced scheme of finite type over $R_0$.
\end{lemma}
\begin{proof}
    We fix homogeneous generators $s_1, \dots, s_k$ of $R$ as an $R_0$-algebra of total degree $1$. The irrelevant ideal is equal to the direct sum $R_+'$ of all subgroups $R_{\underline d}$ with $d_i \geq 1$ for all $i = 1, \dots, m$: As $R$ is $\N_0^m$-graded, every relevant element is contained in $R_+'$. Conversely each monomial in the generators $s_1, \dots s_k$, that lies in $R_+'$, is relevant and therefore $R_+ = R_+'$. This ideal is generated by the finite set $S$ of all monomials in the generators $s_1, \dots s_k$ of degree $(1, \dots, 1)$, hence $\Multiproj R$ is covered by the open subsets $D_+(f)$ for $f \in S$. Using Proposition 3.3 in \cite{brenner2003ample} we see that $\Multiproj R$ is separated. Furthermore Proposition 2.5 in \loccit{} implies, that the morphism $\Multiproj R \to \Spec R_0$ is of finite type. Finally, since every localization of $R$ is reduced, $\Multiproj R$ is covered by reduced affine schemes and thus is reduced itself.
\end{proof}

\subsection{Multiprojective varieties}
\label{subsec:multiproj_varieties}

We summarize some properties of multiprojective varieties, \ie closed subvarieties $X$ of a product $\PP \coloneqq \PP(V_1) \times \dots \times \PP(V_m)$ of projective spaces, where $V_1, \dots, V_m$ are finite-dimensional vector spaces over an algebraically closed field $\K$. It is rather difficult to find the theory of multiprojective varieties in the literature, as it is a direct generalization of the theory of embedded projective varieties $Y \subseteq \PP(V)$.

We fix a closed subvariety $X \subseteq \PP$. The \textbf{multicone} $\hat X \subseteq V_1 \times \dots \times V_m$ of $X$ is the closure in $\PP$ of the preimage of $X$ under the morphism
\begin{align*}
    \pi: (V_1 \setminus \set{0}) \times \dots \times (V_m \setminus \set{0}) \twoheadrightarrow \PP(V_1) \times \dots \times \PP(V_m).
\end{align*}
A point $(v_1, \dots, v_m) \in V_1 \times \dots \times V_m$ is contained in the multicone, if and only if there exist non-zero vectors $w_i \in V_i$ such that $v_i \in \K w_i$ and $\pi(w_1, \dots, w_m) \in X$. The coordinate ring $\K[X] \coloneqq \K[\hat X]$ of the multicone is called the \textbf{multihomogeneous coordinate ring} of $X$. Its prime spectrum is isomorphic to $\hat X$. Note that the multihomogeneous coordinate rings of two multiprojective varieties may not be isomorphic, even if the varieties are isomorphic, so $\K[X]$ does depend on the embedding of $X$ into a product of projective spaces. The $\N_0$-grading on the polynomial ring $\K[V_i] \cong \bigoplus_{d \in \N_0} \operatorname{Sym}^d V_i^*$ induces an $\N_0^m$-grading on $\K[\PP] = \K[V_1] \otimes_\K \dots \otimes_\K \K[V_m]$ which corresponds to the $(\K^\times)^m$-action given by component-wise scalar multiplication:
\begin{align*}
    (t_1, \dots, t_m) \cdot (v_1, \dots, v_m) = (t_1 v_1, \dots, t_m v_m)
\end{align*}
for all $t_1, \dots, t_m \in \K^\times$ and $v_i \in V_i$. The (affine) vanishing ideal $I(\hat X) \subseteq \K[\PP]$ is multihomogeneous, hence $\K[\hat X]$ also graded by $\N_0^m$.

For the rest of this section we write $R = \K[X]$. Every multihomogeneous element $f \in R$ defines the closed subset of all $x \in X$ with $f(x) = 0$ and each closed subset of $X$ is of the form
\begin{align*}
    V_\PP(I) = \set{x \in X \mid \text{$f(x) = 0$ for all $f \in I$ multihomogeneous}}
\end{align*}
for a multihomogeneous ideal $I \subseteq R$. Conversely, every closed subset $Y \subseteq X$ defines the multihomogeneous ideal
\begin{align*}
    I_\PP(Y) = \langle \set{f \in R \mid \text{$f$ multihomogeneous and $f(Y) = 0$}} \rangle \subseteq R.
\end{align*}
In this notation the multicone of $X$ is equal to the (affine) vanishing set $V(I_\PP(X))$. 

The projective Nullstellensatz can also be generalized to the multiprojective setting. It involves the ideal quotient
\begin{align*}
    (I : J) = \set{r \in R \mid rJ \subseteq I}
\end{align*}
of two ideals $I, J \subseteq R$. We say $I$ is \textbf{\boldmath{}$J$-saturated}, if $(I : J) = I$. As we have seen in the proof of Lemma~\ref{lem:multiproj_as_variety}, the irrelevant ideal of $R$ is given by
\begin{align*}
    R_+ = \bigoplus_{\underline d \in \N_0^m \atop d_i \geq 1 \forall i} R_{\underline d}.
\end{align*}

\begin{proposition}[{Multiprojective Nullstellensatz, \cite[Section 1.8]{feigin2021relative}}]
    If $I \subseteq R$ is a multihomogeneous ideal, then
    \begin{align*}
        I_\PP(V_\PP(I)) = (\sqrt{I} : R_+).
    \end{align*}
    In particular, we have a bijection $Y \mapsto I_\PP(Y)$ between the closed subvarieties $Y \subseteq X$ and all $R_+$-saturated, multihomogeneous radical ideals in $R$, which do not contain $R_+$. Irreducible closed subvarieties of $X$ correspond to multihomogeneous prime ideals in $R$ not containing $R_+$ (as they are automatically $R_+$-saturated).
\end{proposition}

\begin{remark}
    If $R$ is an $\N_0^m$-graded, reduced $\K$-algebra, that is finitely generated by elements of total degree $1$, then $R$ is isomorphic to the multihomogeneous coordinate ring of a multiprojective variety $X$ and $\Multiproj R \cong X$.
\end{remark}

It was shown in \cite{herrmann1997reduction} that multiprojective varieties also have an associated Hilbert polynomial: There exists a unique polynomial $H_R \in \Q[x_1, \dots, x_m]$, such that $H_R(\underline d) = \dim \K[X]_{\underline d}$ holds for all $\underline d \geq \underline d'$ (component-wise comparison) for some $\underline d' \in \N_0^m$. Here it is essential that the algebra $\K[X]$ is generated by elements of total degree one, otherwise the Hilbert polynomial is replaced by a function, which is only a quasi-polynomial on certain cones and glued together along their facets. The total degree $\deg H_R$ of the Hilbert polynomial is equal to the dimension of $X$ and it can be uniquely written in the form
\begin{align*}
    H_R(\underline d) = \sum_{\underline k \in \N_0^m} a_{\underline k} \, \binom{d_1 + k_1}{k_1} \cdots \binom{d_m + k_m}{k_m}
\end{align*}
with coefficients $a_{\underline k} \in \Z$. For $k_1 + \dots + k_m = \deg H_R$ these numbers are called the \textbf{multidegrees} of $X$ and they are non-negative. We denote them by $\deg_{\underline k}(X) = a_{\underline k}$. There is a useful criterion proved in \cite{castillo2020multidegrees} for determining which multidegrees are non-zero and thus actually appear in the Hilbert polynomial. It states that $\deg_{\underline k}(X)$ is positive, if and only if
\begin{align*}
    \sum_{i \in I} k_i \leq \dim \pi_I(X)
\end{align*}
holds for all subsets $I \subseteq [m]$, where $\pi_I: \prod_{i=1}^m \PP(V_i) \twoheadrightarrow \prod_{i \in I} \PP(V_i)$ is the natural projection.

\begin{remark}
    Let $\underline k \in \N_0^m$ with $k_1 + \dots + k_m = \dim X$. The multidegree $\deg_{\underline k}(X)$ can also be interpreted as the number of points in the intersection of $X$ in $\prod_{i=1}^m \PP(V_i)$ with a subspace $\PP(L_1) \times \dots \times \PP(L_m)$ in general position, where $L_i \subseteq V_i$ is a non-zero linear subspace of codimension $k_i$.
\end{remark}

The homogeneous component $G_R$ of the Hilbert polynomial $H_R \in \Q[x_1, \dots, x_m]$ of the highest total degree it equal to
\begin{align*}
    G_R = \sum_{\underline k} \frac{\deg_{\underline k}(X)}{k_1! \cdots k_m!} \, x_1^{k_1} \cdots x_m^{k_m},
\end{align*}
where the sum runs over all $\underline k \in \N_0^m$ with $k_1 + \dots + k_m = \dim X$. The value of $G_R$ at a point $\underline d \in \N_0^m$ can also be written as
\begin{align*}
    \label{eq:volume_function}
    G_R(\underline d) = \lim_{n \to \infty} \frac{\dim R_{n \underline d}}{n^{\dim X}}.
\end{align*}
This function $G_R: \N_0^m \to \R$ is sometimes called the \textit{volume function} of $R$. Its connection to convex geometry via global Newton-Okounkov bodies was studied in \cite{cid2021multigraded} and we use the ideas of this paper for the above sections on Newton-Okounkov theory.

\begin{lemma}[Multiprojective Jacobi-criterion]
    \label{lem:jacobi_crit}
    We identify the multihomogeneous coordinate ring of $\PP \coloneqq \prod_{i=1}^m \PP^{n_i}$ with the polynomial ring $S = \K[x_{i,j} \mid (i,j) \in J]$ for $J = \set{(i,j) \in \N_0^2 \mid 1 \leq i \leq m, 0 \leq j \leq n_i}$. Let $X \subseteq \PP$ be a closed subvariety and $f_1, \dots, f_r$ be multihomogeneous generators of the vanishing ideal $I_\PP(X) \subseteq S$. Then a point $x = ([v_1], \dots, [v_m]) \in X$ with $v_i \in \K^{n_i + 1} \setminus \set{0}$ is smooth, if and only if the rank of the Jacobian matrix
    \begin{equation}
    \label{eq:jacobian}
    \left( \frac{\partial f_k}{\partial x_{i,j}}(v_1, \dots, v_m) \right)_{k=1, \dots, r, \mkern1mu j \in J}
    \end{equation}
    is at least $n_1 + \dots + n_m + m - \dim X$.
\end{lemma}
\begin{proof}
    Define the affine space $V = \prod_{i=1}^m \K^{n_i + 1}$. We can assume w.\,l.\,o.\,g. that $x$ is contained in the affine patch $U = \set{(u_1, \dots, u_m) \in V \mid x_{1,0}(u_1) = \dots = x_{m,0}(u_m) = 1}$. Let $\iota: U \hookrightarrow V$ be the inclusion. The (affine) coordinate ring of $U$ can be identified with the polynomial ring $S' = \K[x_{i,j} \mid (i,j) \in J']$ for $J' = \set{(i,j) \in \N_0^2 \mid 1 \leq i \leq m, 1 \leq j \leq n_i}$. The ideal $I(X \cap U) \subseteq S'$ is the dehomogenization of $I_\PP(X)$, \ie the image under the comorphism $\iota^*: S \to S'$ of $\iota$. As the smoothness of $x$ can be checked locally in $U$, the affine Jacobi-criterion implies that $x$ is smooth, if and only if the rank of the matrix
    \begin{align*}
        A' = \left( \frac{\partial \circ \iota^*(f_k)}{\partial x_{i,j}}(\iota(\overline v)) \right)_{k=1, \dots, r, j \in J'}
    \end{align*}
    Since $\iota^*$ commutes with the partial derivatives in a coordinate of $J'$, this is a submatrix of the matrix $A$ in (\ref{eq:jacobian}). The columns not contained in this submatrix are linearly dependent to the columns of $A'$, as
    \begin{align*}
        \sum_{j=0}^{n_j} x_{i,j} \cdot \frac{\partial f}{\partial x_{i,j}} = \deg(f)_i \cdot f
    \end{align*}
    holds for every multihomogeneous polynomial $f \in S$. Therefore $x$ is smooth if and only if $\operatorname{rank} A = \operatorname{rank} A' \geq \dim V - \dim (X \cap U) = n_1 + \dots + n_m + m - \dim X$.
\end{proof}

\begin{corollary}
    \label{cor:smooth_in_codim_1}
    Let $X \subseteq \prod_{i=1}^m \PP(V_i)$ be a closed subvariety. Then $X$ is smooth in codimension one, if and only if its affine multicone $\hat X$ is smooth in codimension one.
\end{corollary}
\begin{proof}
    Let $V = V_1 \times \dots \times V_m$ and $Z_i \subseteq V$ be the preimage of $\set{0} \subseteq V_i$ under the linear projection $\hat\pi_{i}: V \twoheadrightarrow V_i$. The open subvariety $U = V \setminus (\bigcup_{i=1}^m Z_i)$ is the preimage of $\prod_{i=1}^m \PP(V_i)$ under the natural morphism
    \begin{align*}
        \pi: \prod_{i=1}^m V_i \setminus \set{0} \twoheadrightarrow \prod_{i=1}^m \PP(V_i).
    \end{align*}
    By Lemma~\ref{lem:jacobi_crit}, a point $x \in \hat X \cap U$ is smooth, if and only if its corresponding projective point $\pi(x) \in X$ is smooth, since the rank of the Jacobian matrix in (\ref{eq:jacobian}) is independent of the $(\K^\times)^m$-action.
    
    Now suppose that $X$ is smooth in codimension one and let $S$ be an irreducible component of the subvariety $\mathrm{Sing}(\hat X)$ of singular points in the multicone $\hat X$. If $S \cap U \neq \varnothing$ then $S \cap U$ is an irreducible component of $\hat X \cap U$. Hence $\dim S = \dim S \cap U = \dim \pi(S \cap U) + m \leq \dim X - 2 + m = \dim \hat X - 2$, since $\mathrm{Sing}(X)$ has at least codimension $2$ in $X$. Thus $\hat X$ is smooth in codimension one.
    
    Conversely if $\hat X$ is smooth in codimension one and $S$ is an irreducible component of the subvariety $\mathrm{Sing}(X)$ of singular points in $X$, then $\pi^{-1}(S)$ is contained in an irreducible component $S'$ of $\mathrm{Sing}(\hat X)$. Therefore $\dim S = \pi^{-1}(S) - m \leq \dim \hat X - 2 - m = \dim X - 2$ and $X$ is smooth in codimension one.
\end{proof}

We close this section with a lemma which is useful for computing the quasi-valuation of a Seshadri stratification via the decomposition of $R$ into its homogeneous components. The action of the torus $(\K^\times)^m$ on $\hat X$ induces an action on $R$, where an element $\underline t \in (\K^\times)^m$ acts on $g \in R$ via the left translation $\underline t \cdot g =: g^{\underline t}$, where $g^{\underline t}$ is the regular function on $\hat X$ with $g^{\underline t}(x) = g(\underline t^{-1} \cdot x)$ for all $x \in \hat X$.

\begin{lemma}
    \label{lem:span_homog_components}
    For every $h \in R$, the linear subspace generated by the multihomogeneous components $h_{\underline d}$, $\underline d \in \N_0^m$, of $h$ coincides with the linear subspace, which is spanned by all function $h^{\underline t}$ for $\underline t \in (\K^\times)^m$.
\end{lemma}
\begin{proof}
    It suffices to show this statement for $X = \prod_{i=1}^m \PP(V_i)$. By choosing a basis of every vector space $V_i$, we identify $R$ with the polynomial ring in the variables $x_{i,j}$ for $i \in \set{1, \dots, m}$ and $j \in \set{0, \dots, n_i}$, where $n_i = \dim V_i - 1$. The torus $(\K^\times)^m$ acts as scalars on each subspace $R_{\underline d}$, $d \in \N_0^m$. Hence every function $h^{\underline t}$ for $\underline t \in (\K^\times)^m$ can be written as a linear combination of the multihomogeneous components $h_{\underline d}$. It remains to show the other inclusion of vector spaces. 

    For $\underline c = (c_1, \dots, c_m) \in \N_0^m$ let $R_{\underline c} \subseteq R$ be the linear subspace of all $h \in R$, such that $h_{\underline d} \neq 0$ only if $d_k < c_k$ holds for all $k = 1, \dots, m$. We prove by induction over $m$, that for all $\underline c, \underline d \in \N_0^m$ there exists a finite set $S \subseteq (\K^\times)^m$, such that the multihomogeneous component $h_{\underline d}$ of every $h \in R_{\underline c}$ can be written as $h_{\underline d} = \sum_{\underline t \in S} a_{\underline t} h^{\underline t}$, where the scalars $a_{\underline t} \in \K$ are independent of $h$.

    The induction base $m = 0$ is trivial. So now let $m \geq 1$ and fix a primitive $c_m$-th root of unity $\zeta \in \K^\times$. Let $B_m$ be the basis of the algebra $R_m = \K[x_{m,j} \mid 1 \leq j \leq n_m]$ of all monomials in the variables $x_{m,j}$ and let $h \in R_{\underline c}$, which we write in the form $h = \sum_{g \in B_m} f_g \cdot g$, where $f_g$ lies in the ring $R'$ of polynomials in the variables $x_{i,j}$ for $i \in \set{1, \dots, m-1}$ and $j \in \set{1, \dots, n_i}$. We define $h_j$, $j \in \N_0$, to be the sum of all $f_g \cdot g$, where $g$ is of degree $j$ in $R_m$. For every $i = 0, \dots, c_m$ we have
    \begin{align*}
        h^{(1, \dots, 1, \zeta^i)} = \sum_{j=0}^{c_m} \zeta^{ij} h_j.
    \end{align*}
    As the matrix $A = (\zeta^{ij})_{i,j=0, \dots, c_m}$ is a Vandermonde-matrix and its determinant is non-zero by the choice of $\zeta$, we get $\langle h^{(1, \dots, 1, \zeta^i)} \mid i = 0, \dots, c_m \rangle_\K = \langle h_j \mid j = 0, \dots, c_m \rangle_\K$.

    Fix a tuple $\underline d \in \N_0^m$ and let $\underline d' = (d_1, \dots, d_{m-1})$. By induction the following equation holds for a finite set $S' \subseteq (\K^\times)^{m-1}$:
    \begin{align*}
        h_{\underline d} &= \! \sum_{g \in B_m \atop \deg(g) = d_m} \! (f_g)_{\underline d \mkern1mu '} \mkern2mu g = \sum_{g \in B_m \atop \deg(g) = d_m} \sum_{\underline s' \in S'} a_{\underline s'} (f_g)^{\underline s'} g = \sum_{\underline s' \in S'} a_{\underline s'} \left( \sum_{g \in B_m \atop \deg(g) = d_m} \! (f_g \cdot g)^{(\underline s', 1)} \right) \\
        &= \sum_{\underline s' \in S'} a_{\underline s'} (h_{d_m})^{(\underline s', 1)} = \sum_{\underline s' \in S'} \sum_{i=0}^s a_{\underline s'} \mkern1mu a_i h^{(\underline s', \zeta^i)}
    \end{align*}
    The scalars $a_{i}$ are the entries in the $d_m$-th row of the inverse matrix of $A$. We see that the products $a_{\underline s'} a_i$ only depend on the choice of $\underline c$ and $\underline d$.
\end{proof}

\vskip 6pt

\section*{List of notations}
\pagestyle{empty}

\vskip 4pt

\begin{abbrv}
\item[{$b_{p,q}$}] bond of a covering relation $p > q$
\item[{$\operatorname{Cone} S$}] convex cone generated by a set $S$
\item[{$\operatorname{Conv} S$}] convex hull generated by a set $S$
\item[{$\hat X$}] multicone of embedded projective variety $X$
\item[{$\Gamma$}] fan of monoids
\item[{$\Gamma_C$}] monoids associated to a chain $C \subseteq A$
\item[{$\Gamma_C^{(d)}$}] Veronese submonoid of $\Gamma_C$
\item[{$\mathrm{gr}_{\mathcal V} R$}] associated graded algebra to $\mathcal V$ 
\item[{$\mathrm{gr}_{\mathcal V, \mathfrak C} R$}] subalgebra of $\mathrm{gr}_{\mathcal V} R$
\item[{$\mathrm{gr}_{\mathcal V, \mathfrak C}^{(\underline d)} R$}] Veronese subalgebra of $\mathrm{gr}_{\mathcal V, \mathfrak C} R$
\item[{$\mathcal L^{C}$}] lattice generated by $\Gamma_C$
\item[{$\Lambda^+$}] monoid of dominant weights of $G$
\item[{$[k]$}] set $\set{1, \dots, k}$ for $k \in \N$
\item[{$\max_Q$}] maximal lift (p. \pageref{eq:def_min_Q_max_Q})
\item[{$\min_Q$}] minimal lift (p. \pageref{eq:def_min_Q_max_Q})
\item[{$\pi_Q$}] projection map (p. \pageref{eq:def_pi_Q})
\item[{$P_i$}] a maximal parabolic subgroup in Dynkin type $\texttt{A}$
\item[{$r$}] dimension of the stratified variety $X$
\item[{$\sigma_C$}] cone of multidegrees of $\Gamma_C$
\item[$\supp \underline a$] support of an element in some $\Q^A$ (p. \pageref{txt:def_support})
\item[{$\vert \underline d \vert$}] total degree $d_1 + \dots + d_m$ of $\underline d \in \N_0^m$
\item[{$V$}] ambient affine space of the stratification
\item[{$W_\lambda$}] stabilizer of $\lambda$ in the Weyl group
\item[{$W_Q$}] Weyl subgroup of $Q$
\item[{$X_I$}] projection of a multiprojective variety $X$ to the factors in $I$
\item[{$X_\theta$}] Schubert variety in $G/Q$ to $\theta \in W/W_Q$
\end{abbrv}

\vskip 20pt

\printbibliography


\end{document}